\documentclass[11pt,reqno]{amsproc}
\usepackage{lmodern}
 \usepackage[T1]{fontenc}
\usepackage{textcomp}

\usepackage[margin=1in]{geometry}
\usepackage{booktabs} 
\usepackage{amsmath, amsthm, amssymb}
\usepackage[colorlinks=true, pdfstartview=FitV, linkcolor=blue,citecolor=blue, urlcolor=blue,bookmarksdepth=2]{hyperref}
\usepackage[usenames,dvipsnames]{color}
\usepackage{mathtools}
\usepackage{biblatex}
\usepackage[export]{adjustbox}
\usepackage{subcaption}

\usepackage{comment}
\excludecomment{comment}  

\definecolor{labelkey}{rgb}{0,0,1}

\usepackage{float}
\usepackage{graphicx}
\usepackage{mathtools}

\usepackage{ifthen}
\newcounter{TypeOfWork}
\newcounter{Doktorat}
\newcounter{PracaC1}
\newcommand{\Doktorat}[1]{\ifnum \value{TypeOfWork}=\value{Doktorat}{#1}\else{}\fi}
\newcommand{\article}[1]{\ifnum \value{TypeOfWork}=\value{PracaC1}{#1}\else{}\fi}
\title[Rigorous $C^1$ integration of dissipative PDEs]{Rigorous $C^1$ integration of dissipative PDEs }
\author{Jakub Bana\'{s}kiewicz}
\address{Faculty of Mathematics and Computer Science, Jagiellonian University, ul. Łojasiewicza 6, 30-348, Kraków, 
Email: \texttt{jakub.banaskiewicz@uj.edu.pl}}
\thanks{}

\date{\today}

\newcommand{\norm}[1]{\left\|{#1}\right\|}
\newcommand{\skalarProduct}[2]{\left<{#1},{#2} \right>}
\newcommand{\brakets}[1]{\left({#1} \right)}
\newcommand{\set}[1]{\{{#1}\}}

\newcommand{\R}[0]{\mathbb{R}}

\DeclarePairedDelimiter\ceil{\lceil}{\rceil}
\DeclarePairedDelimiter\floor{\lfloor}{\rfloor}

\newtheorem{theorem}{Theorem}[section]

\newtheorem{lemma}[theorem]{Lemma}

\newtheorem{proposition}{Proposition}[section]
\theoremstyle{definition}
\newtheorem{definition}[proposition]{Definition}
\newtheorem{remark}[proposition]{Remark}

\numberwithin{equation}{section}
\newcommand{\fracLaplasianAlfa}[0]{(-\Delta)^\alpha}

\begin{document}

\begin{abstract}
We introduce a new \(C^1\) algorithm for the rigorous integration of dissipative partial differential equations. The algorithm is designed for computer-assisted proofs that require rigorous control of both solutions and their derivatives with respect to initial data. As applications, we establish the existence of locally attracting periodic orbits for initial and boundary value problems for two non-autonomous dissipative PDEs: the Chafee--Infante equation and the Burgers equation with a fractional Laplacian.
\end{abstract}
\maketitle

\section{Introduction}\label{sec:introduction}
The main goal of this article is to present a new $C^1$ algorithm for the rigorous integration of dissipative PDEs. 
The proposed algorithm can be used for computer assisted proof schemes, which need to control both the solution and its derivative with respect to initial data.
We also present the proofs of the existence of a locally attracting periodic orbit for two non-autonomous problems: one governed by the Chafee--Infante equation and one by the Burgers equation with the fractional Laplacian.
We base our method on the $C^0$ algorithm from a series of papers \cite{ZPKuramotoII,ZPKuramotoIII} where it was developed and applied for the Kuramoto--Sivashinsky equation. It was also used for other problems such as the Burgers equation, the Swift--Hohenberg equation, or the Brusselator system (see \cite{BrusselatorPrzyjetaPraca,CYRANKAFOURIER,ZgliczuCyrankaC1}). We begin by presenting our approach to $C^1$ rigorous computations, with a brief comparison with the other ones. Then, we present the benchmark problems to which the method is applied. Next, we describe the results of the $C^1$ computations for these problems.  

\subsection{Algorithm for rigorous \texorpdfstring{$C^1$}{C1} integration of dissipative PDEs} 
Rigorous $C^1$ computation is a challenging problem for PDEs. Not only do we need to work with infinite dimensional phase space, but also we need to estimate the derivatives of solutions with respect to initial data. This requires solving the variational equation associated with the original problem for an infinite number of directions. To our best knowledge, the first $C^1$ computation was presented in article \cite{ArioliKoch2010} to validate the existence of a hyperbolic periodic solution for the Kuramoto--Sivashinsky equation. The technique was based on the Newton--Kantorovich approach. However, that work does not provide a general algorithm for $C^1$ computations. 

The $C^1$ 
computation algorithm based on rigorous integration of PDEs was firstly presented in article \cite{ZgliczuCyrankaC1}, and was used to prove the existence of a locally attracting periodic orbit for the non-autonomous Burgers equation. The method there was based on the estimation of the logarithmic norm, and therefore it is not easy to generalise this method to other problems whose solving requires more subtle information about the derivative.
The most recent approach comes from the papers \cite{wilczak2024symbolic,wilczak2025selfconsistent}. These papers present the $C^1$ integration algorithm and apply it to the Kuramoto--Sivashinsky equation.
The main result of these works is the extension of the result about chaos in this problem, to the existence of an infinite number of heteroclinic solutions (see \cite[Theorem 2]{wilczak2024symbolic}).
The approach presented there is based on the splitting of the operator into blocks and control of their norms during the evolution.

In our technique,  we work directly with the variational equation and by applying the $C^0$ algorithm we estimate the derivative of the flow $D_{x_0}\varphi (t,t_0,x^0) h^0.$ This allows us to estimate the image $D_{x_0} \varphi (t,t_0,X^0) X^{0,h},$ where $X^0$ is the initial data for the original problem and $X^{0, h}$ is the initial data for the variational part. 
From the linearity of the derivative, we can combine the results of different computations by taking finite linear combinations of them. 
For the infinite dimensional tail part we also use $C^0$ algorithm applied to the set $X^{0,h}$ constructed in such a way that it allows us to get information about the image of $D_{x_0} (t,t_0,X^0)h^0$ for every $h^0$ in infinite dimensional spaces representing the 'tail' of Fourier expansion. It is also important that we also allow the set $X^{0,h}$ to be unbounded in the phase space.
The more detailed overview of the algorithm can be found in Section \ref{sec:Overview_C1_algorithm}.

\subsection{The non-autonomous Chafee--Infante problem} The first problem under consideration is the non-autonomous Chafee--Infante equation, which has the following form
\begin{equation}\label{eq:Chafee-InfateNonautonomous}
\begin{cases}
 u_t =   u_{xx}+ \lambda u -b(t)u^3 \;  \text{for}\; (x,t)\in (0,\pi)\times(t_0,\infty),\\
u(t,x) = 0\; \text{for}\; (x,t)\in \{0,\pi\} \times (t_0,\infty),
\\u(t_0,x) = u^0(x),\ \textrm{for}\ x\in(0,\pi),
\end{cases}
\end{equation}
with $\lambda>0$ and $b(t)$ is the function dependent on time.
The autonomous Chafee--Infante equation where, e.g., $b(t)=1$ for $t\in\mathbb{R}$, is one of the few examples where the full structure of the global attractor on the one hand is nontrivial and, on the other hand, can be described purely analytically with respect to the parameter $\lambda>0.$ This attractor contains equilibria, arising from consecutive pitchfork bifurcations of the zero solution and heteroclinic connections between them. The number of equilibria is $2n+1,$, where $n$ is a natural number such that $n<\sqrt{\lambda}\leq n+1.$ This equation was first proposed and investigated by Chafee and Infante in the article \cite{ChafeeInfanteOryginalny}, where the authors construct the equilibria. 
Later, Henry, Matano, Fiedler, Rocha \cite{HENRY1985165,MatanoLapNumber,FiedlerLapNumber} developed the theory which allows one to obtain the heteroclinic connections. This has led to a complete understanding of the attractor structure (see also \cite[Theorem 4.3.12]{hale1988asymptotic}).
In the non-autonomous version of the problem, given by \eqref{eq:Chafee-InfateNonautonomous}, such a description remains unknown, the difficulty arises from the fact that here the equilibria are replaced by non-constant trajectories, called ''non-autonomous equilibria''. The existence of non-autonomous attractors for this problem has been proven in \cite{ChepyzhovVishik}. The partial results on the structure of the attractors were obtained in \cite{ChafeeInfanteLangaCarvalho} (see also \cite[Chapter 13]{LangaCarvalhoBook}). The authors proved there that there exists a global positive solution which attracts all positive initial conditions (\cite[Theorem 4]{ChafeeInfanteLangaCarvalho}). Additionally, from \cite[Theorem 5]{ChafeeInfanteLangaCarvalho} there exist at least  $2n+1$ global solutions to the Chafee--Infante system, where $n$ is a natural number such that $n< \sqrt{\lambda}\leq n+1$. These solutions, for which the number of zeros in the space variable is finite and constant in time, are supposed to be non-autonomous counterparts of the hyperbolic equilibria of the autonomous Chafee--Infante equation. However, it is still unknown if the found solutions are all non-autonomous equilibria, and, to the best of author's knowledge,  there are no results concerning the existence of non-autonomous heteroclinic connections. Moreover, all these results were proved with the assumption that
the non-autonomous term $b(t)$ is positive for every $t\in\mathbb{R}$ and separated from zero and infinity, e.g., $b_0\leq b(t)\leq B_0$ for $0<b_0\leq B_0.$ In the case when $b(t)$ can take a negative value for some $t\in\mathbb{R},$ even the existence of a global nonzero solution remained unknown. The main challenge here is the possibility of blow-up behavior coming from the term $u^3$, which leads to an equation with a positive constant for some time $t$. Using computer‑assisted methods, we are able to overcome this difficulty and prove the following result.
\begin{theorem}\label{th:Chaffe-InfanteMainTheorem}
The Chafee--Infante equation \eqref{eq:Chafee-InfateNonautonomous} for parameters
\begin{equation*}
    (\lambda,b(t)) \in \left\{\left(2,\, \frac{3}{2}\sin(2\pi t) +1\right),\left(2 ,\frac{1}{2}\sin(2\pi t)+1\right),\left(5,\frac{1}{2}\sin(2\pi t)+1\right)\right\}
\end{equation*}
 has the periodic orbit $u^*(t,.):\mathbb{R}\to C_0(0,\pi)$ with a period $1.$ Moreover there exist $\delta>0,$   $D>0$ and $\kappa>0$ such that for every 
initial time $t_0\in\mathbb{R}$ and initial data $u^0\in C_0(0,\pi)$ satisfying $\norm{u^*(t_0)-u^0}_{C_0(0,\pi)}\leq\delta$ there  exists a solution $u:[t_0,\infty)\to C_0(0,\pi)$ with $u(t_0)=u^0$ such that
\begin{equation*}
    \norm{u^*(t)-u(t)}_{C_0(0,\pi)} \leq De^{-\kappa(t-t_0)} \norm{u^*(t_0)-u(t_0)}_{C_0(0,\pi)}\quad\text{for}\;t\in[t_0,\infty).
\end{equation*}
\end{theorem}
This, to our knowledge, is the first result on the existence of a  global nonzero solution with $b(t)$ that changes signs. Moreover, we prove that this solution is locally exponentially attracting in the $C_0(0,\pi)$ space. To prove these results, we use a new developed technique of $C^1$ computations based on rigorous integration of unbounded sets. 

\subsection{The non-autonomous forced Burgers equation with fractional Laplacian} The second problem under consideration is the non-autonomous forced Burgers equation with fractional Laplacian and odd periodic boundary conditions, given by the following equations:
\begin{equation}\label{eq:BurgersFractionalLaplacianEquation}
\begin{cases}
        u_t= -\fracLaplasianAlfa u+\nu(u^2)_x +g(t,x)\; \text{for}\; (x,t)\in \R\times(t_0,\infty),\\
        u(t,x+2\pi)  = u(t,x)\; \text{for}\; (x,t)\in \mathbb{R}\times (t_0,\infty),\\
        u(t,-x)  = -u(t,x)\; \text{for}\; (x,t)\in \mathbb{R}\times (t_0,\infty),
    \\u(t_0,x) = u^{0}(x)\, \textrm{ for}\ x\in\mathbb{R}.
\end{cases}
\end{equation}
In the above problem $\alpha\in(0,1],\;\nu\in \mathbb{R}$, and $g(t,x)$ is an odd-periodic function. The fractional Laplacian is defined in terms of the Fourier expansion of odd-periodic functions
\begin{equation*}
    \fracLaplasianAlfa u =\sum_{k=1}^\infty k^{2\alpha}u_k\sin(kx).
\end{equation*}
If we consider the restriction of $u$ to the interval $[0,\pi],$ then $u$ must satisfy the homogeneous Dirichlet boundary conditions in the endpoints of this interval. In the above equation, we have the interaction between the dissipative term $-\fracLaplasianAlfa$ and the term $(u^2)_x$ that can lead to a blow up for the space derivatives of solutions.
In the autonomous case $g(t,x) =0,$ the characterization of the regularity of the solutions is given in \cite{KiselevNazarovShterenberg2008} and depends on $\alpha.$ For  $\alpha\in[\frac{1}{2},1]$   there exists a global solution for sufficiently smooth initial data  (see  \cite[Theorem 1.1, Theorem 1.3]{KiselevNazarovShterenberg2008}). Especially interesting are the results for $\alpha=\frac{1}{2}$, as that is the critical case for this equation, and the authors of \cite{KiselevNazarovShterenberg2008} develop a new method based on the modulus of continuity to handle this case.
In contrast, for $\alpha\in(0,\frac{1}{2})$ there exists an initial datum that leads to a blow-up of space derivatives in finite time (\cite[Theorem 1.4]{KiselevNazarovShterenberg2008}). We are interested in this equation with nonzero $g(t,x).$ For the standard Laplacian, that is, $\alpha = 1,$ the paper \cite{KalitaZgliczynski2020} contains results on the existence of global solutions which are exponentially attracting all initial data (see \cite[Theorem 4.1]{KalitaZgliczynski2020}). In our article, we focus on the case where $\frac{1}{2}<\alpha\leq 1.$ We prove the existence of a locally attracting periodic orbit. We want to check how close  we can get with the parameter $\alpha$ to the critical case $\frac{1}{2}$. We have proven the following theorem 
\begin{theorem}\label{th:TheoremBurgersFractionalLaplacianEquation}
    The Burgers equation \eqref{eq:BurgersFractionalLaplacianEquation} for  paramteres 
    \begin{equation*}
        \nu=\frac{1}{2},\quad\alpha \in\left\{\frac{1}{2}+\frac{i}{64},\; 11\leq i \leq 32\right\} ,\quad g(t,x)= \left(\frac{1}{2}\sin(2\pi t)+1\right)\sin(x),
    \end{equation*}
 has a periodic orbit $u^*(t,.):\mathbb{R}\to H^2_\text{odd}(-\pi,\pi)$ with a period $1.$ Moreover there exist $\delta>0,$   $D>0$ and $\kappa>0$ such that for every 
initial data $t_0\in\mathbb{R}$ and $u^0\in H^2_\text{odd}(-\pi,\pi)$ satisfying $\norm{u^*(t_0)-u^0}_{H^2_{odd
}(-\pi,\pi)}\leq\delta$ there  exists a solution $u:[t_0,\infty)\to H^2_\text{odd}(-\pi,\pi)$ with $u(t_0)=u^0$ such that
\begin{equation*}
    \norm{u^*(t)-u(t)}_{H^2_\text{odd}(-\pi,\pi)} \leq De^{-\kappa(t-t_0)} \norm{u^*(t_0)-u(t_0)}_{H^2_\text{odd}(-\pi,\pi)}\quad\text{for}\;t\in[t_0,\infty).
\end{equation*}
\end{theorem}

The above result is obtained using computer assisted methods and, in particular, our newly developed $C^1$ algorithm.

\subsection{Application of \texorpdfstring{$C^1$}{C1} algorithm in benchmark problems}
 In the paper, we present the following applications of our $C^1$ algorithm:
 \begin{itemize}
     \item In the case of the Burgers equation with the fractional Laplacian, we integrate the set $X^{h,0} =\left\{u\in H^2_{\text{odd}}(-\pi,\pi): u_k\in \frac{[-1,1]}{k^2},\;k\in\mathbb{N}\right\}$ which contains a certain  ball in $H^2_{\text{odd}}.$ This allows us to prove the attraction in Theorem \ref{th:TheoremBurgersFractionalLaplacianEquation}.
     In these computations, we only need to apply the $C^0$ algorithm to the system of equations \eqref{eq:BurgersVariational}, with initial data $X^0\times X^{h,0}.$ 
     \item We can also integrate unbounded sets, which gives us the control of every mode. 
     We present such a method for the Chafee--Infante equation. Namely, we are able to integrate the set in form
\begin{align*}
    X^{h,0} &=\left\{u\in C_0(0,\pi): u_k\in [-1,1],\;k\in\mathbb{N}\right\}\text{ or even }
    \\X^{h,0}&=\left\{u\in C_0(0,\pi): u_k\in [-k^p,k^p],\;k\in\mathbb{N}\right\}\text{ with $p>0.$}
\end{align*}
This can give us the effective estimate of $\frac{\partial }{\partial x} \varphi (t,t_0,X^0)$ over every mode (see for example \eqref{eq:VarationalDecayingEstimate}).
For these estimates, we use Lemmas \ref{lem:costimesSinDiverging}, \ref{lem:costimesSinSlowConvering} in which we derive the estimates on the convolutions of two series, where one of them is constant or increasing. We use these estimates in the proof of Theorem \ref{th:Chaffe-InfanteMainTheorem}
\item In Section \ref{sec:chafee-Infate Addition}, we perform the $C^1$ computation around the zero solution, to see if the results match the theoretical formulas. We additionally conduct the $C^1$ computation for the numerically found unstable periodic orbit to show the wide applicability of the algorithm, and the possibility to utilise the algorithm to show the hyperbolicity of unstable orbits. 
\end{itemize}
  
\subsection{The further possible applications of \texorpdfstring{$C^1$}{C1} algorithm}
The few possible developments and further applications of this work are following:
\begin{itemize}
    \item Application of $C^1$ algorithm to develop computations for estimations of the derivative of Poincar\'{e} maps, and use them to validate the hyperbolicity of periodic solution in the autonomous system.
    \item Application of $C^1$ algorithm to prove the existence of possibly unstable periodic  orbits using the rigorous Newton method or the Krawczyk method presented, for example, in \cite{ZgliczuKrawczyk}. 
    \item Developing the algorithm of rigorous integration and its $C^1$ counterpart for the systems with non-diagonal linear term like, such as, for example, the Ginzburg--Landau equation. 
\end{itemize}
The developed work can be applied to obtain the following additional results for PDEs, which are possible topics of our further planned work:
\begin{itemize}
    \item Extension of study on the Brusselator system from work \cite{BrusselatorPrzyjetaPraca,ArioliBrusselator}, with the use of the $C^1$ algorithm. We would like to prove that the periodic solution from \cite[Theorem 1.1.]{BrusselatorPrzyjetaPraca} is locally attracting. Additionally, we would like to prove the existence of an unstable periodic orbit, which should exist after period doubling bifurcation (see \cite[Theorem 4.3.]{BrusselatorPrzyjetaPraca} and \cite[Theorem 3.]{ArioliBrusselator}), and the existence of heteroclinic connections between stable and unstable orbits.
    \item Further investigation of the structure of attractors for the non-autonomous Chafee--Infante equation, including the heteroclinic connections between the various solutions, which play the role of non-autonomous equilibria. The final goal of this endeavor would be to fully characterize the non-autonomous attractor for a given parameters. 
\end{itemize}

\subsection{Structure of the article} The structure of the article is as follows. In the next Section \ref{sec:Abstractnonlinearequation} we recall the theory of nonlinear abstract equations in the Banach spaces and the theory of variational equations for them.
 Section \ref{sec:AlgorithmIntegration} is devoted to the description of $C^0$ and $C^1$ integration algorithms for the abstract problem formulated in the preceding section.
 In Section \ref{sec:Chafee-Infante results}, we describe the computer assisted proofs for the Chafee--Infate and Burgers equation.
Finally, Section \ref{sec:Appendix} is the Appendix which contains some auxiliary inequalities and estimates. The main part of it is Section \ref{sec:algebra} which contains the results on the algebra of infinite series utilized in the algorithms.

{\color{black}
The codes of programs used in computer assisted proofs for the Chafee--Infante equation are available on the web site \cite{CodeChaffeInfante}, and for Burgers equation at \cite{CodeBurgers}. They are based on the CAPD library \cite{CAPD}, which contains modules for both nonrigorous and validated numerics for dynamical systems}.

\subsection{Notation} By $\mathbb{N}$ we denote the set of natural numbers starting with $1$. For a normed space $X$ by $\norm{\cdot}_{X}$ we denote the norm in this space and by $B_{X}(x^0,r):=\left\{x\in X: \norm{x^0-x}_X<r\right\}$ the open ball in this space.  By $L^2(a,b)$ we denote the standard $L^2$ Hilbert space in the interval $(a,b).$ By $H^k(a,b)$ we denote the standard Sobolev spaces. In addition, $C_0(a,b)=\{f:[a,b]\to\mathbb{R}: f\text{ is continuous and satisfies } f(a)=f(b)=0\}.$ Additionally, we denote $L^2_{odd}(-\pi,\pi)$, $H^1_{odd}(-\pi,\pi)$, $H^2_{odd}(-\pi,\pi)$ the subspaces of odd functions in interval $(-\pi,\pi),$ whose periodic extensions on $\mathbb{R}$ belong to  $L^2_{loc}(\mathbb{R}),$ $H^1_{loc}(\mathbb{R})$ and $H^2_{loc}(\mathbb{R})$, respectively. For a given Banach space $X,Y$ we denote by $\mathcal{L}(X;Y)$ the space of linear and bounded operators from $X$ to $Y.$ We write $\mathcal{L}(X) = \mathcal{L}(X;X).$

\section{Nonlinear differential equations in Banach spaces}\label{sec:Abstractnonlinearequation}

 In this section we discuss the theory of semilinear problems in Banach spaces, together with their variational equations, which describe the derivative of the solution with respect to the initial data. The link between  variational and original  problem, provided in Lemma \ref{lem:VariationalProperties}, is crucial for further steps of the proofs.

\subsection{Abstract problem}\label{sec21}
\article{
We recall some facts about abstract differential equations in the Banach spaces, which we use in the article.}
Let $Y,\,W$ be Banach spaces such that
$Y$ is continuously and densely embedded in $W.$
We consider the following problem
\begin{equation}\label{eq:AbstractProblem}
\begin{cases}
   \frac{d}{dt}x(t) = Lx(t) + f(t,x(t)) = F(t,x(t)),
   \\ x(t_0) = x^{0},
\end{cases}
\end{equation}
where $L:D(L)\to Y$ is a linear operator, which generates the $C_0$ semigroup $e^{tL}:Y\to Y$.
We assume that $x^0\in Y$, and $f:\mathbb{R} \times Y\to W$ is a continuous mapping. In the article we work with the following assumptions:
\begin{itemize}
    \item [(BL1)]\label{as:BL1} For some constants $M_1,\kappa_1>0$ we have $ \norm{e^{tL}}_Y \leq  M_1 e^{\kappa_1 t}$ for every $t\in[0,\infty).$
    \item [(BL2)]\label{as:BL2} The semigroup $e^{tL}$ can be extended to the $C_0$ semigroup on $W.$ There exist constants $M_2,\kappa_2>0$ and $\gamma\in[0,1)$ such that for every $z\in W$, we have $\norm{e^{tL}z }_Y\leq M_2e^{\kappa_2 t}\frac{1}{t^\gamma}\norm{z}_{W}$ for every $t\in(0,+\infty).$
    \item [(Bf)]\label{as:Bf} For every $R>0$ there exists $C(R)>0$ such that for every $x,y\in Y$ such that $\norm{x}_Y,\norm{y}_Y\leq R$  we have $ \sup_{t\in\mathbb{R}}\norm{f(t,x) - f(t,y)}_{W}\leq C(R) \norm{x-y}_{Y}.$
    \item [(BH)]\label{as:BH} The space $W$ is continuously and densely embedded in some Hilbert space $H$ with the scalar product $\skalarProduct{\cdot}{\cdot}$ and orthogonal basis $\{e_i\}_{i\in\mathbb{N}}$ such that $e_i\in D(L)$ for every $i\in\mathbb{N}.$
    \item [(BD)]\label{as:BD} The operator $L$ is diagonal, that is $Le_i = \lambda_i e_i$ for $\lambda_i\not =0$ and the semigroup satisfies $e^{Lt}e_i =e^{\lambda_i t} e_i.$
\end{itemize}
If the conditions \hyperref[as:BH]{(BH)},\hyperref[as:BD]{(BD)} are satisfied, then for a given
 $x\in Y$, we use following the notation:
 \begin{equation*}
     x_i = \frac{\skalarProduct{e_i}{x}}{\skalarProduct{e_i}{e_i}},
     \quad
     f_i(x) = \frac{\skalarProduct{e_i}{f(x)}}{\skalarProduct{e_i}{e_i}},
     \quad
     F_i(x) = \lambda_ix_i +f_i(x). 
 \end{equation*}
 
\begin{lemma}
    If $W=Y$ then the condition  \hyperref[as:BL1]{(BL1)} implies the condition
     \hyperref[as:BL2]{(BL2)}.
\end{lemma}
\begin{definition}\label{def:DuhamelSolution}
We call the function $x:[t_0,t_0+T]\to Y$ the mild solution, Duhamel solution, or simply solution to the problem \eqref{eq:AbstractProblem} if it satisfies the formula
\begin{equation*}
    x(t) = e^{L(t-t_0)}x^0 + \int_{t_0}^{t}e^{L(t-s)}f(s,x(s))ds,
\end{equation*}
where the equality in $Y$ is supposed to hold for every $t\in [t_0,t_0+T].$
\end{definition}

The following lemma provides the criteria for the local in time existence and uniqueness of mild solutions to problem \eqref{eq:AbstractProblem} (see, e.g., 
\cite[Lemma 2.2]{BrusselatorPrzyjetaPraca}). 
\begin{lemma}\label{lem:AbstractProblemLocalExistence}
Assume that conditions \hyperref[as:BL1]{(BL1)}, \hyperref[as:BL2]{(BL2)}, \hyperref[as:Bf]{(Bf)} are satisfied.
Then for every initial data $x^0\in Y$ and $t_0\in\mathbb{R}$  there exists a unique local in time solution to problem $x:[t_0,t_0+T]\to Y$ in the sense of Definition \ref{def:DuhamelSolution}, where $T$ may depend on $t_0$ and $x^{0}$.  
\end{lemma}
\Doktorat{
\begin{proof}
For a given $t_0\in\mathbb{R},$ $x^{0} \in Y$ and $\delta>0$ consider the set 
\begin{equation*}
    S_\delta:= \left\{y\in C([t_0,t_0+\delta];Y):y(t_0) = x^{0}\text{ for every}\;t\in[t_0,t_0+\delta] \text{ we have }\norm{y(t)}_Y\leq 1+M_1\norm{x^{0} }_Y  \right\},
\end{equation*}
and define the mapping  $\mathcal{T}:C([t_0,t_0+\delta];Y) \to C([t_0,t_0+\delta];Y)$ by the formula 
\begin{equation*}
    \mathcal{T}(y)(t) = e^{L(t-t_0)}x^{0} + \int_{t_0}^te^{L(t-s)}f(s,y(s))ds. 
\end{equation*}
The space $C([t_0,t_0+\delta];Y)$ is equipped with the norm $\sup_{t\in[0,\delta]}\norm{y(t)}_{Y}$.
From the assumptions \hyperref[as:BL1]{(BL1)}, \hyperref[as:BL2]{(BL2)}, \hyperref[as:Bf]{(Bf)} we have
\begin{align*}
    \norm{\mathcal{T}(y)(t)}_Y
    &\leq
    M_1e^{\kappa_1(t-t_0)}\norm{x^{0} }_{Y}+
    \int_{t_0}^t\norm{e^{L(t-s)} f(s,y(s))}_Y ds
    \\&\leq
    M_1e^{\kappa_1(t-t_0)}\norm{x^{0} }_{Y}+
    M_2\int_{t_0}^t e^{(t-s)\kappa_2}\frac{1}{(t-s)^{\gamma}}\norm{ f(s,y(s))}_{W} ds
    \\&\leq
     M_1e^{\kappa_1(t-t_0)}\norm{x^{0} }_{Y}+
    M_2 e^{(t-t_0)\kappa_2}\int_{0}^t\frac{1}{(t-s)^{\gamma}} ds
     \sup_{s\in[t_0,t_0+\delta]}\norm{f(s,y(s))}_{W}
     \\&\leq
     M_1e^{\kappa_1(t-t_0)}\norm{x^{0} }_{Y}+
     \frac{ M_2(t-t_0)^{1-\gamma}}{1-\gamma}
     e^{(t-t_0)\kappa_2}
    \left(C(R)\sup_{s\in[t_0,t_0+\delta]}\norm{y(s)}_Y+\norm{f(s,0)}_{W} \right),
\end{align*}
where $R = 1+M_1\norm{x^{0} }_Y.$
If we pick $\delta\in(0,1)$ such that
\begin{equation*}
  e^{\kappa_1\delta}\leq 1+\frac{1}{2M_1\norm{x^{0}     }_Y}
  \quad\text{and}\quad
    \delta^{1-\gamma}e^{\delta  \kappa_2}\leq \frac{1-\gamma}{2M_2\left(C(R)R+\sup_{s\in[t_0,t_0+1]}\norm{f(s,0)}_{W} \right)},
\end{equation*}
then we have $\mathcal{T}(S_\delta) \subset S_\delta.$ Moreover, we have 
\begin{align*}
  \norm{\mathcal{T}(y_1)(t)-\mathcal{T}(y_2)(t)}_Y \leq  M_2e^{(t-t_0)\kappa_2}C(R)\frac{(t-t_0)^{1-\gamma}}{1-\gamma}
  \sup_{s\in[0,t]} \norm{y_1(s)-y_2(s)}_Y.
\end{align*}
If we take $\delta>0$ such that
\begin{equation*}
    \delta^{1-\gamma}e^{\delta \kappa_2}\leq \frac{1-\gamma}{2M_2C(R)},
\end{equation*}
the mapping $\mathcal{T}$ is a contraction on the set $S_\delta.$ From the Banach fixed point theorem we deduce that $\mathcal{T}$ has a unique fixed point which is a solution to \eqref{eq:AbstractProblem}.
\end{proof}
}
The next lemma provides a result on saturated solutions. 
\begin{lemma}\label{re:procesCon}
Assume that conditions \hyperref[as:BL1]{(BL1)}, \hyperref[as:BL2]{(BL2)}, \hyperref[as:Bf]{(Bf)} are satisfied.
For every $t_0\in\mathbb{R},$ $x^{0}\in Y$ there exists $t_{max}(t_0,x^{0})\in(t_0,\infty]$ such that,
for the
unique solution $x:[t_0,t_{max}(t_0,x^{0}))\rightarrow Y$ to \eqref{eq:AbstractProblem}, the interval $[t_0,t_{max}(t_0,x^{0}))$ is the maximal interval of existence of this solution. Additionally if we consider set set 
$\Omega :=\{(t,t_0,x^{0})\in \R\times\R \times Y: t\in[t_0,t_{max}(t_0,x^{0})) \} $ then the function $\varphi:\Omega\to Y$ given by  formula  
\begin{equation}
    \varphi(t,t_0,x^{0}) = e^{L(t-t_0)}x^{0} + \int_{t_0}^t e^{L(t-s)} f(s,x(s))\, ds.
\end{equation}
defines a local process.
\end{lemma}
\Doktorat{
\begin{lemma}\label{re:semiDynamicalCon}
Assume that conditions \hyperref[as:BL1]{(BL1)}, \hyperref[as:BL2]{(BL2)}, \hyperref[as:Bf]{(Bf)} hold and that $f$ does not depend on $t$ i.e.  $f(t,x) = f(x)$.
Then for $x^{0}\in Y$ there exist $t_{max}(x^{0})\in(0,\infty]$ such that,
for the
unique solution $x:[0,t_{max}(x^{0}))\to Y$ to \eqref{eq:AbstractProblem}, the interval $[0,t_{max}(x^{0}))$ is the maximal interval of existence of this solution. Additionally if we consider set set 
$\Omega :=\{(t,x^{0})\in \R \times Y: t\in[0,t_{max}(x^{0})) \} $ then the function $\varphi:\Omega\to Y$ given by  formula  
\begin{equation}
    \varphi(t,x^{0}) = e^{Lt}x^{0} + \int_{0}^t e^{L(t-s)} f(s,x(s))\, ds.
\end{equation}
defines a local dynamical system.
\end{lemma}
}
The following result concerns the evolution of the solution's Fourier coefficients. 
\begin{lemma}\label{lem:AbstractFourierCoef}
Let $x:[t_0,t_0+T]\to Y$ satisfy \eqref{eq:AbstractProblem}. 
Assume that conditions \hyperref[as:BL1]{(BL1)}, \hyperref[as:BL2]{(BL2)}, \hyperref[as:Bf]{(Bf)}, \hyperref[as:BH]{(BH)}, \hyperref[as:BD]{(BD)} are satisfied.
Then for every $i\in\mathbb{N},$ the Fourier coefficients $x_i(t)$ satisfy the following non-autonomous ODE
\begin{equation*}
    \frac{dx_i}{dt}(t) = \lambda_ix_i(t)+f_i(t,x(t))
    \quad \text{for every $t\in(t_0,t_0+T)$}.
\end{equation*}
\end{lemma}

\Doktorat{
\begin{proof}
If $x(t)$ satisfies \eqref{eq:AbstractProblem} then for every  $i\in\mathbb{N}$ we have
\begin{equation}\label{eq:formula1}
    x_i(t) = e^{\lambda_i(t-t_0) }x^0_i+
    \int_{t_0}^te^{\lambda_i(t-s)} f_i(s,x(s))ds.
\end{equation}

Observe that as the condition \hyperref[as:BH]{(BH)} is satisfied and $f:\mathbb{R}\times Y\to W,$ is continuous, then also for every $i\in\mathbb{N}$ the function $f_i:\mathbb{R}\times Y\to\mathbb{R}$ is continuous. Therefore in formula \eqref{eq:formula1} the function under integral is also continuous.
Hence, we can differentiate this formula and get the ODE for $i$-th Fourier coefficient.
\end{proof}
}
We prove the result that will later be useful in the study of the variational equation. 
\begin{lemma}\label{rem:LongTimeExistence}
Assume that conditions \hyperref[as:BL1]{(BL1)}, \hyperref[as:BL2]{(BL2)}, \hyperref[as:Bf]{(Bf)} are satisfied.
For every $R>0,T>0$ there exist constants $C_1(R,T),\;\lambda_1(R,T)>0$ such that for every $t_0\in\mathbb{R},x^{0}\in Y$,
for which the local process $\varphi$ corresponding to equation \eqref{eq:AbstractProblem} is well defined for $(T+t_0,t_0,x^{0})$  and satisfies
\begin{equation*}
    \norm{\varphi(\tau+t_0,t_0,x^{{0}})}_Y\leq R \text{ for every }\tau\in[0,T],
\end{equation*}
 we have:
\begin{itemize}
    \item For every $h^0\in B_Y(0,\lambda_1(R,T))$ local process is well defined for  $(T+t_0,t_0,x^{0}+h^0).$ 
    \item For every $\tau\in[0,T]$ and $h^0\in B_Y(0,\lambda_1(R,T)),$ the following inequality holds
\begin{equation*}
    \norm{\varphi(\tau+t_0,t_0,x^0+h^0)-\varphi(\tau+t_0,t_0,x^0)}_Y\leq \min\{ C_1(R,T)\norm{h^0}_Y,1 \}.   
\end{equation*}
\end{itemize}
\end{lemma}
\begin{proof}
We use the notation $x(t) = \varphi(t,t_0,x^0).$ 
We will show that the Duhamel operator 
\begin{equation*}
    \mathcal{T}(y)(t) = e^{L(t-t_0) }(x^0+h^0) + \int_{t^{0}}^t e^{L(t-s)}f(s,y(s))ds,
\end{equation*}
is a contraction. Consider the set 
\begin{equation*}
    S_{M_3,\alpha}:= \left \{y\in C([t_0,t_0+T];Y), 
    y(t_0) = h^0+x^{0}, 
    \norm{y(t)-x(t)}_Y\leq\min \left\{ M_3 e^{\alpha(t-t_0)}\norm{h^0}_Y,1 \right\}
    \right\}.
\end{equation*}

Firstly we show that 
$\mathcal{T}(S_{M_3,\alpha})\subset S_{M_3,\alpha}$ for some $M_3,\alpha >0$.
We have that 
\begin{equation*}
     \mathcal{T}(y)(t) - x(t)= e^{L(t-t_0)}h^0 +\int_{t_0}^te^{L(t-s)}(f(s,x(s))-f(s,y(s)))ds.
\end{equation*}
With the use of conditions \hyperref[as:BL1]{(BL1)}, \hyperref[as:BL2]{(BL2)}, \hyperref[as:Bf]{(Bf)} we estimate
\begin{align*}
    \norm{\mathcal{T}(y)(t) - x(t)}_Y &\leq 
    \norm{h^0}_Y\left(M_1e^{\kappa_1(t-t_0)} + C(R+1)M_3M_2e^{\kappa_2T}\int_{t_0}^{t}\frac{1}{(t-s)^\gamma} e^{\alpha (s-t_0)}ds\right). 
\end{align*}
We set  $M_3\geq 2 M_1.$  From \eqref{eq:integralfracExtEstimation} we can find $\alpha\geq\kappa_1$ such that
\begin{equation*}
    \int_{t_0}^{t}\frac{1}{(t-s)^\gamma} e^{\alpha (s-t_0)}ds 
    \leq \frac{e^{\alpha (t-t_0)}}{2C(R+1)M_2e^{\kappa T}}.
\end{equation*}
Then 
$\norm{\mathcal{T}(y)(t)-x(t)}_Y\leq M_3 e^{\alpha(t-t_0)}\norm{h^0}_Y.$
Observe that, if we take $h^0$ such that 
\begin{equation*}
\norm{h^0}_Y\leq \frac{1}{M_3 e^{\alpha T}},
\end{equation*}
then $\norm{\mathcal{T}(y)(t)-x(t)}_Y\leq 1$. Note that above bound on $h^0$ depends only on norm of the solution $x(t),$ the constants from assumptions \hyperref[as:BL1]{(BL1)}, \hyperref[as:BL2]{(BL2)}, \hyperref[as:Bf]{(Bf)} and $T$.
  To show that $\mathcal{T}$ is a contraction  we equip the space  $C([t_0,t_0+T];Y)$ with the  norm $\|y\|_{\beta} = \sup_{t\in[t_0,t_0+T]} \norm{y(t)}_Y e^{-\beta (t-t_0)},$ where $\beta>0$ is an appropriately chosen positive constant. For $y_1,y_2\in S$ we compute
\begin{align*}
     \left\|\int_{t_0}^t e^{L(t-s) }( f(s,y_2(s)) -f(s,y_1(s)))ds\right\|_Y
     &\leq M_2 e^{\kappa_2 T } \int_{t_0}^t \frac{1}{(t-s)^{\gamma}}
   \norm { f(s,y_1(s)) -f(s,y_2(s)))}_{W}  ds\\
      &\leq
      M_2e^{\kappa_2 T }C(R+1)
  \norm { y_1-y_2}_{\beta}
      \int_{t_0}^t\frac{1}{(t-s)^\gamma}e^{\beta(s -t_0)} ds
      .
\end{align*}
From \eqref{eq:integralfracExtEstimation} we can find $\beta$ such that
\begin{equation*}
    \int_{t_0}^t\frac{1}{(t-s)^\gamma}e^{\beta (s-t_0)}
    \leq \frac{e^{\beta (t-t_0)}}{2 M_2 e^{T\kappa_2}C(R+1)}.
\end{equation*}
Then for every $t\in[t_0,t_0+\tau]$ we have
\begin{equation*}
    \left\|\int_{t_0}^t e^{L(t-s) }( f(s,y_2(s)) -f(s,y_1(s)))ds\right\|_Y \leq \frac{1}{2}e^{(t-t_0)\beta}\norm{y_1-y_2}_{\beta}.
\end{equation*}
Hence, $\mathcal{T}$ is a contraction on the set $S$ equipped with the norm $\norm{\cdot}_\beta$. The Banach fixed point theorem gives the unique fixed point of the map $\mathcal{T}$, and the proof is complete.
\end{proof}

\subsection{The non-autonomous Chafee–Infante equation}\label{sec:nonautonomousChafee–InfanteEquation}
We discuss how to represent the non-autonomous Chafee–Infante problem in the abstract framework presented in this section. We consider the Hilbert space $H = L^2(0,\pi)$ and the Banach space $W=Y = C_0(0,\pi).$ 
In the space $H$ the system of functions  $e_k = \sin(kx)$
for $k\in \mathbb{N}$ is an orthogonal basis.  For $u\in L^2(0,\pi)$ we denote 
by $u_i$ the $i$-th coefficients in the Fourier expansion in the sine series of $u.$
We define the operator 
$$
L:D(L)\rightarrow Y \ \ \textrm{as}\ \ Lu = u_{xx} +\lambda u,$$ 
where $D(L) = \{u\in C_0(0,\pi)  \, :\  Lu \in Y\}$. The operator $L$ defines a $C_0$ semigroup on $Y.$ We denote it by $e^{tL}$\Doktorat{ 
(see Example \ref{ex:linearEquation2})}. We define $f:\mathbb{R} \times Y\to Y$ by the formula 
\begin{equation*}
    f(t,u) = -b(t)u^3 .
\end{equation*}
with $b(t)$   smooth enough and globally bounded.
We write the non‑autonomous Chafee--Infante equation \eqref{eq:Chafee-InfateNonautonomous} as the following abstract problem 
\begin{equation}\label{eq:AbstractChafeInfante}
    \begin{cases}
     \frac{d}{dt}u(t) = Lu(t) + f(t,u(t)),\\
      u(t_0) = u^{0}.
    \end{cases}
\end{equation} 

\Doktorat{
The proof the following Lemma is similar to the proof of Lemma \ref{lem:BrusselatorAbstractConditions}, so we skip it.}
The following lemma follows from the maximum principle for the heat equation and the form of nonlinear term.
\begin{lemma}\label{lem:ChafeInfanteAbstractConditions}
    The system \eqref{eq:AbstractChafeInfante} satisfies conditions \hyperref[as:BL1]{(BL1)}, \hyperref[as:BL2]{(BL2)}, \hyperref[as:Bf]{(Bf)}, \hyperref[as:BH]{(BH)}, \hyperref[as:BD]{(BD)}.
\end{lemma}
\article{}
Next theorem follows from Lemmas \ref{lem:AbstractProblemLocalExistence} and \ref{lem:ChafeInfanteAbstractConditions}.
\begin{theorem}
For every $u^0\in Y$ there exists a function $u:[t_0,t_{max}(t_0,v^0))\to Y$ which is the unique solution to the
\eqref{eq:AbstractChafeInfante}, satisfying the Duhamel formula
\begin{equation*}
    u(t) = e^{L(t-t_0)}u^0 + \int_{t_0}^te^{L(t-s)}f(s,u(s))ds.
\end{equation*}
\end{theorem}
\begin{proposition}\label{prop:OddSubspaceChaffeInfante}
Space
$V := \{u\in Y: u_{2i}  = 0, \;\text{for}\;i\in\mathbb{N}\}$ is invariant for system \eqref{eq:Chafee-InfateNonautonomous}. Moreover, if $u_0\in V$ then $u(t,\frac{\pi}{2}+x) = u(t,\frac{\pi}{2}-x)$ 
for every $t\in [t_0,t_{max}(t_0,u_0))$ and $x\in[0,\frac{\pi}{2}].$
\end{proposition}

\subsection{The periodically forced Burgers equation with fractional Laplacian }
We consider equation \eqref{eq:BurgersFractionalLaplacianEquation} with $g(t,\cdot)\in  H^2_{odd}(-\pi,\pi).$
As $u(t,x)$ is odd and $2\pi$-periodic it follows  that $u(t,0)=u(t,\pi)=0,$ so $u$ satisfies Dirichlet boundary conditions on $(0,\pi)$.  
We introduce the spaces:
\begin{equation*}
    Y = H^2_{\text{odd}} (-\pi,\pi),\qquad W=H^1_{\text{odd}}(-\pi,\pi),\quad \qquad H = L^2_{odd}(-\pi,\pi),
\end{equation*}
with the following norms:
\begin{equation*}
    \norm{u}_Y = \norm{u_{xx}}_{L^2(-\pi,\pi)}\qquad\norm{u}_{W} = \norm{u_x}_{L^2(-\pi,\pi)}\qquad\norm{u}_{H} = \norm{u}_{L^2(-\pi,\pi)}.
\end{equation*}
In the space $H$ the system of functions  $e_k = \sin(kx)$
for $k\in \mathbb{N}$ is an orthogonal basis.  As in the case of the Chafee--Infante problem, for $u\in L^2_{odd}(-\pi,\pi)$ we denote by $u_k$ the $k$-th coefficient in the Fourier expansion in the sine series of $u.$
 The fractional Laplacian is defined in terms of the Fourier expansion, that is, for $u\in H^2_{odd}(-\pi,\pi)$ we have
\begin{equation*}
    (-\Delta)^\alpha u =\sum_{k=1}^\infty k^{2\alpha}u_k\sin(kx).
\end{equation*}
We define the linear operator $L:D(L) \to Y,$ with $D(L)=\{u\in Y: Lu\in Y \}$ in the following way
\begin{equation*}
    L :=-(-\Delta)^\alpha.
\end{equation*} 
This operator defines $C_0$ semigrup $e^{tL}$ from $Y$ to $Y$ by the formula
\begin{equation*}
    e^{tL}u = \sum_{k=1}^\infty e^{-tk^{2\alpha}}u_k^0\sin(kx).
\end{equation*}
The nonlinearity is defined as
\begin{equation*}
    f(t,u) = \nu(u^2)_{x} +g(t).
\end{equation*}
With previously defined spaces and functions, we formulate the Burgers equation with fractional Laplacian \eqref{eq:BurgersFractionalLaplacianEquation} as the following abstract problem 
\begin{equation}\label{eq:BurgersFrationalAbstract}
    \begin{cases}
     \frac{d}{dt}u(t) = Lu(t) + f(t,u(t)),\\
      u(t_0) = u^{0}.
    \end{cases}
\end{equation} 

\begin{lemma}\label{lem:BurgersFrationalConditions}
    Let $\alpha\in (\frac{1}{2},1]$. The system \eqref{eq:BurgersFrationalAbstract} satisfies conditions \hyperref[as:BL1]{(BL1)}, \hyperref[as:BL2]{(BL2)}, \hyperref[as:Bf]{(Bf)}, \hyperref[as:BH]{(BH)}, \hyperref[as:BD]{(BD)}.
\end{lemma}
\begin{proof}
It is easy to observe that the semigroup is generated by the fractional Laplacian satisfies the bound 
\begin{equation*}
    \norm{e^{tL}u_0}_{Y}\leq\norm{u_0}_{Y},
\end{equation*}
which implies condition \hyperref[as:BL1]{(BL1)}.
For \hyperref[as:BL2]{(BL2)} we observe that $e^{tL}$ can be extended to the $C_0$ semigroup on the space $W$, which satisfies 
\begin{equation*}
    \norm{e^{tL}u_0}_{Y}\leq\frac{C}{t^\frac{1}{2\alpha}}\norm{u_0}_{W}\quad \text{for some $C>0$ }.
\end{equation*}
The above bound follows from the computations that use the fact that the global maximum of $e^{-2tk^{2\alpha}}k^2$ with respect to the variable $k$ is equal to $
(2e\alpha t)^{-\frac{1}{\alpha
}}$.
\begin{align*}
    \norm{e^{tL}u^0}_{Y}^2 = \pi\sum_{k=1}^\infty (e^{-2tk^{2\alpha}}k^2)(k|u^0_k|)^2 \leq  \pi\sum_{k=1}^\infty \frac{1}{(2e\alpha t)^{\frac{1}{\alpha}}
    }(k|u^0_k|)^2 = \frac{1}{(2e\alpha t)^{\frac{1}{\alpha}}}\norm{u^0}^2_{W}.
\end{align*}
 For $f$ we have
\begin{align*}
     \norm{f(t,u_1)-f(t,u_2)}_{W} &= \nu\norm{(u_1^2)_{xx}-(u_2^2)_{xx}}_{L^2(-\pi,\pi)}
     \\&\leq2 \nu\norm{\left((u_1)_{x}\right)^2-\left((u_2)_{x}\right)^2}_{L^2(-\pi,\pi)}+2\nu\norm{u_1((u_1-u_2)_{xx})}_{L^2(-\pi,\pi)} \\
     &+2\nu\norm{(u_2)_{xx}(u_1-u_2)}_{L^2(-\pi,\pi)}.
\end{align*}
Observe that there exists some $\overline{C}>0,$ that for all $u\in Y $  we have 
\begin{equation*}
    \norm{u}_{L^\infty(-\pi,\pi)}\leq \overline{C}\norm{u_{x}}_{L^2(-\pi,\pi)} 
    \quad \text{and}
    \quad
    \norm{u_{x}}_{L^\infty(-\pi,\pi)}\leq \overline{C}\norm{u_{xx}}_{L^2(-\pi,\pi)}. 
\end{equation*}
We estimate each term 
\begin{align*}
    \norm{((u_1)_{x})^2-((u_2)_{x})^2}_{L^2(-\pi,\pi)}&\leq \norm{(u_1-u_2)_x}_{L^2(-\pi,\pi)}\norm{(u_1+u_2)_x}_{L^\infty(-\pi,\pi)}\\
    &\leq \overline{C}\norm{u_1-u_2}_Y\norm{u_1+u_2}_{Y},
    \\
    \norm{u_1((u_1-u_2)_{xx})}_{L^2(-\pi,\pi)}&\leq \norm{u_1}_{L^\infty(-\pi,\pi)}\norm{(u_1-u_2)_{xx}}_{L^2(-\pi,\pi)}\leq \overline{C} \norm{u_1}_Y\norm{u_1-u_2}_Y,
    \\
    \norm{(u_2)_{xx}(u_1-u_2)}_{L^2(-\pi,\pi)}&\leq \norm{(u_2)_{xx}}_{L^2(-\pi,\pi)}\norm{u_1-u_2}_{L^\infty(-\pi,\pi)}\leq \overline{C} \norm{u_2}_Y\norm{u_1-u_2}_{Y}.
\end{align*}
So, with $C(R) = 8\nu \overline{C}R$ we obtain the following inequality
\begin{equation*}
   \norm{f(t,u_1)-f(t,u_2)}_{W} \leq  C(R)\norm{u_1-u_2}_{Y},
\end{equation*}
which proves the \hyperref[as:Bf]{(Bf)} condition. Conditions  \hyperref[as:BH]{(BH)}, \hyperref[as:BD]{(BD)} follow from the definition of spaces $Y,\;H$ and the operator $L.$
\end{proof}
\subsection{Variational equations}
To study the variational equation, we additionally assume that:
\begin{itemize}
    \item [(BV1)]\label{as:BV1} The function $f(t,\cdot):Y\to W$ is Fr\'{e}chet differentiable at every point $x\in Y$ with a linear and bounded derivative $Df(t,x)\in \mathcal{L}(Y;W)$.
    \item [(BV2)]\label{as:BV2} For every $R>0$ there exists $C_2(R)>0,\;\lambda_2(R)>0$ such that for every $\norm{x}_{Y}\leq R$ and $\norm{h}_{Y}\leq  \lambda_2(R) $ we have
    \begin{equation}\label{bounds}
        \sup_{t\in\mathbb{R}}
        {\norm{f(t,x+ h )-f(t,x) - Df(t,x)h}_{W}}\leq \|h\|_{Y}^2 C_2(R )\quad\text{and} \quad
        \sup_{t\in\mathbb{R}}
        \norm{Df(t,x)}_{\mathcal{L}(Y;W)}\leq C_2(R).
    \end{equation}
     \item [(BV3)]\label{as:BV3} The mapping $Df:\mathbb{R}\times Y\to \mathcal{L}(Y;W)$ is continuous.
\end{itemize}

With these assumptions, we formulate the variational system
\begin{equation}\label{eq:AbstractVariationalProblem}
\begin{cases}
   \frac{d}{dt}x(t) = Lx(t) + f(t,x(t)),
   \\ \frac{d}{dt}h(t) = Lh(t) + Df(t,x(t))h(t) ,
   \\ x(t_0) = x^0,\quad h(t_0)=h^0.
\end{cases}
\end{equation}
We can write this system of equations as the abstract problem \eqref{eq:AbstractProblem} is the following way
\begin{equation}\label{eq:AbstractVariationalProblemAsOneEquation}
\begin{cases}
   \frac{d}{dt}(x(t),h(t)) = \overline{L}(x(t),h(t)) + \overline{f}(t,x(t),h(t)),
   \\ (x(t),h(t)) = (x^0,h^0).
\end{cases}
\end{equation}
with the product spaces $\overline{Y} = Y\times Y$ and $\overline{W}= W\times W$ equipped with the norms $\norm{(x,h)}_{\overline{Y}} = \norm{x}_{Y}+\norm{h}_{Y}$ and
$\norm{(x,h)}_{\overline{W}} = \norm{x}_{W}+\norm{h}_{W}.$ The linear operator 
$\overline{L}:D(L)\times D(L)\to \overline{Y}$ and the nonlinear term  $\overline{f}: \mathbb{R}\times \overline{Y}\to \overline{W}$ are given by the formulas
\begin{equation*}
    \overline{L}(x,h) = (Lx,Lh), \quad\overline{f}(t,x,h)=(f(t,x),Df(t,x)h).
\end{equation*}
\begin{lemma}\label{rem:VartionalAsNormalAbstract}
Assume the conditions \hyperref[as:BL1]{(BL1)}, \hyperref[as:BL2]{(BL2)}, \hyperref[as:Bf]{(Bf)},\hyperref[as:BV1]{(BV1)}-\hyperref[as:BV1]{(BV3)}. Then conditions \hyperref[as:BL1]{(BL1)}, \hyperref[as:BL2]{(BL2)}, \hyperref[as:Bf]{(Bf)} are also satisfied for the 
system \eqref{eq:AbstractVariationalProblemAsOneEquation}.
\end{lemma}

\begin{lemma}\label{lem:VariationalProperties}
Assume that conditions \hyperref[as:BL1]{(BL1)}, \hyperref[as:BL2]{(BL2)}, \hyperref[as:Bf]{(Bf)}, \hyperref[as:BV1]{(BV1)}-\hyperref[as:BV1]{(BV3)} are satisfied.
For every $R,\;T>0$ there exist constants $C_3(R,T),\;\lambda_3(R,T)>0$ such that for every $t_0\in\mathbb{R},x^{0}\in Y$
for which the local process $\varphi$ corresponding to the equation \eqref{eq:AbstractProblem} is well defined for $(T+t_0,t_0,x^{0})$  and satisfies
\begin{equation*}
    \norm{\varphi(\tau+t_0,t_0,x^{{0}})}_{Y}\leq R\text{ for every }  \tau\in[0,T],
\end{equation*} 
the following assertions hold:
\begin{itemize}
    \item For every $h^0\in Y,$ there exists a unique function $h:[t_0,t_0+T]\to Y$ such that the pair $(\varphi(t_0,t,x^0),h(t))$ is a solution to  \eqref{eq:AbstractVariationalProblem}.
    \item The function $h(t)$ satisfies
    \begin{equation*}
        \frac{\partial \varphi}{\partial h^0}(t,t_0,x^0):=\lim_{\lambda\to 0} \frac{\varphi(t,t_0,x^0+\lambda h^0) - \varphi(t,t_0,x^0)}{\lambda} = h(t)\quad \textrm{for}\quad t\in [t_0,t_0+T].
    \end{equation*}
    
    \item For every $t\in[t_0,t_0+T]$ the function $\varphi(t,t_0,\cdot):Y\to Y$ is Fr\'{e}chet differentiable at $x^0$ with the derivative $V(t,t_0,x^0)h^0 := h(t)$
    \item For every $t\in[t_0,t_0+T]$ and $h^0\in B_Y(0,\lambda_3(R,T))$ the following estimates hold
    \begin{align}
\label{eq:ineq1}\norm{\varphi(t,t_0,x^0+h^0)-\varphi(t,t_0,x^0) - V(t,t_0,x^0)h^0}_{Y}&\leq C_3(R,T)\norm{h^0}_Y^2,\\
   \norm{V(t,t_0,x^0)}_{\mathcal{L}(Y)}& \leq C_3(R,T).\nonumber
    \end{align}

\end{itemize}

\end{lemma}
\begin{proof}
From Lemma \ref{rem:LongTimeExistence} we know that the process $\varphi(t,t_0,x+h)$ is well defined 
 and
$$\norm{\varphi(t,t_0,x^0+h)}_Y\leq R + 1 \ \ \textrm{for every}\ \ 
\norm{h}_Y\leq \lambda_1(R,T),t\in[t_0,t_0+T].$$
As conditions \hyperref[as:BV1]{(BV1)}-\hyperref[as:BV3]{(BV3)} are satisfied, for $|\lambda |\leq \frac{1}{\norm{h^0}_Y}\min\{\lambda_1(R,T),\lambda_2(R+1) \}$
 we can consider the function
\begin{align*}
    g(t,\lambda) &= \frac{1}{\lambda}\left(\varphi(t,t_0,x^0+\lambda h^0) - \varphi(t,t_0,x^0) \right) 
    \\&= e^{(t-t_0)L}h^0 + \int_{t_0}^t e^{L(t-s) } \left(Df(s,\varphi(s,t_0,x^0)) g(s,\lambda)   
    +  
    \frac{1}{\lambda} \Lambda (s,\varphi(s,t_0,x^0),\lambda g(s,\lambda) )\right)  ds,
\end{align*}
where
\begin{equation*}
    \Lambda (t,x,h) := f(t,x+ h )-f(t,x) - Df(t,x)h.
\end{equation*}

Let $\lambda_n $ be a sequence that tends to $0.$ To observe that $\{g(t,\lambda_n)\}_{n\in\mathbb{N}}$ is a Cauchy sequence in the space $C([t_0,t_0+T];Y)$ with the norm 
$\sup_{t\in[t_0,t_0+T]} \norm{h(t)}_Y,$ we compute 
\begin{align*}
  g(t,\lambda_n) -g(t,\lambda_m) =&
  \int_{t_0}^t e^{L(t-s)}Df(s,\varphi(s,t_0,x^0))(g(s,\lambda_n) -g(s,\lambda_m) )ds
   \\& + \frac{1}{\lambda_n}\int_{t_0}^t  e^{L(t-s)}\Lambda(s,\varphi(s,t_0,x^0),\lambda_n g(s,\lambda_n))ds
    \\&-\frac{1}{\lambda_m}\int_{t_0}^t  e^{L(t-s)}\Lambda(s,\varphi(s,t_0,x^0),\lambda_m g(s,\lambda_m))ds.
\end{align*}
From conditions \hyperref[as:BV2]{(BV2)},\hyperref[as:BL2]{(BL2)} we have for some $\alpha>0$
\begin{align*}
    &\norm{\int_{t_0}^t e^{L(t-s)}Df(s,\varphi(s,t_0,x^0))(g(s,\lambda_n) -g(s,\lambda_m) )ds}_Y
     \\
    &
    \leq M_2e^{T \kappa_2}C_2(R+1)
    \int_{{t_0}}^t
    \frac{1}{(t-s)^\delta} e^{(s-t_0)\alpha}ds
    \left (\sup_{s\in[t_0,t_0+T]}e^{-(s-t_0)\alpha} 
    \norm{g(s,\lambda_n)- g(s,\lambda_m)}_Y \right).
\end{align*}
Also, from Lemma \ref{rem:LongTimeExistence} it follows that
\begin{align*}
    \norm{\int_{t_0}^t  e^{L(t-s)}\Lambda(s,\varphi(s,t_0,x^0),\lambda g(s,\lambda))ds}_Y&\leq M_2{e^{T\kappa_2}} C_2(R+1)\frac{T^{1-\gamma}}{1-\gamma}|\lambda|^2 \sup_{s\in[t_0,t_0+T]}\|g(s,\lambda)\|_Y^2
    \\
    &\leq 
    M_2{e^{T\kappa_2}} C_2(R+1)C_1(R,T)^2 \frac{T^{1-\gamma}}{1-\gamma}|\lambda|^2\norm{h^0}_Y
^2.
\end{align*}
From \eqref{eq:integralfracExtEstimation} we can pick $\alpha>0$ such that
\begin{equation*}
    \int_{{t_0}}^t
    \frac{1}{(t-s)^\gamma} e^{\alpha(s-t_0)}ds \leq \frac{1 }{2M_2e^{T \kappa_2}C_2(R+1)}e^{\alpha(t-t_0)}.
\end{equation*}
We obtain
\begin{align*}
    \norm{ g(t,\lambda_n) -g(t,\lambda_m)}_Y &\leq
    e^{\alpha (t-t_0)}
    \frac{1}{2}\sup_{s\in[t_0,t_0+T]}e^{-\alpha (s-t_0)} 
    \norm{g(s,\lambda_n)- g(s,\lambda_m)}_Y  \\&+
    2 M_2{e^{T\kappa_2}} C_2(R+1)C_1(R,T)^2 \frac{T^{1-\gamma}}{1-\gamma}\max\{|\lambda_m|,|\lambda_n|\}\|h^0\|_Y^2. 
\end{align*} 
It follows that
\begin{equation*}
 \norm{ g(t,\lambda_n) -g(t,\lambda_m)}_Y \leq \overline{C}(R,T) \max\{|\lambda_m|,|\lambda_n|\}\norm{h^0}_Y^2,
\end{equation*}
with 
\begin{equation*}
    \overline{C}(R,T)  := 4e^{T(\alpha+\kappa_2)}M_2C_2(R+1)C_1(R,T)^2 \frac{T^{1-\gamma}}{1-\gamma}.
\end{equation*}
As the space $C([t_0,t_0+T];Y)$ is complete, the sequence  $ \{g(t,\lambda_n)\}$ converges in that space to some function $h(t).$
By passing to the limit, we obtain
\begin{equation}\label{eq:h(t)_estimations}
     \norm{ g(t,\lambda) -h(t)}_Y \leq \overline{C}(R,T)|\lambda| \norm{h^0}_Y^2 \text{ for } |\lambda|\leq\frac{1}{\norm{h^0}}_Y \min\{\lambda_1(R,T),\lambda_2(R+1)\}.
\end{equation}
We observe that $(\varphi(t,t_0,x^0),h(t))$ is the unique solution to problem \eqref{eq:AbstractVariationalProblem}. 

Indeed, from the dominated convergence theorem we see that the function $h(t)$ satisfies the following Duhamel formula
\begin{equation*}
    h(t) = e^{(t-t_0)L}h^0+\int_{t_0}^te^{(t-s)L}Df(s,\varphi(s,t_0,x))h(s)ds.
\end{equation*}
The uniqueness of $h(t)$ follows from Lemma \ref{rem:VartionalAsNormalAbstract}.
Additionally, from the definition of $g(t,\lambda)$ we see that 
\begin{equation*}
   \frac{\partial \varphi}{\partial h^0}(t,t_0,x_0) = h(t). 
\end{equation*}
Observe that  the mapping $V(t,t_0,x^0) h^0 = h(t) $ is linear and bounded, leading from $Y$ to $Y.$ Linearity follows from the linearity of the variational equation. If $\norm{h^0}_Y\leq 1,$ we can find a sufficiently small $|\lambda|<1$ such that from  \eqref{eq:h(t)_estimations} and Lemma \ref{rem:LongTimeExistence} we have
\begin{align*}
    \norm{h(t)}_{Y} \leq 
    \norm{h(t) - g(t,\lambda)}_{Y} +
    \norm{g(t,\lambda)}_{Y}
    \leq \overline{C}(R,T)+ C_1(R,T).
\end{align*}
which shows the boundness of $V.$
Now let $\norm{h^0}_Y\leq\min\{\lambda_1(R,T),\lambda_2(R+1)\}.$ Then from inequality \eqref{eq:h(t)_estimations} we obtain
\begin{align*}
  \norm{\varphi(t,t_0,x^0+h^0)-\varphi(t,t_0,x^0)
  -V(t,t_0,x^0)h^0}_Y\leq \overline{C}(R,T) \norm{h^0}_Y^2,
\end{align*}
So for every $t\in[t_0,t_0+T]$ and $h^0\in Y$ such that 
$\norm{h^0}_Y\leq\min\{\lambda_1(R,T),\lambda_2(R+1)\}$ we have \eqref{eq:ineq1}
which concludes the proof.
\end{proof}
We present another way to describe the variational equation.
Namely, if $x:[t_0,t_0+T] \to Y$ is a solution to \eqref{eq:AbstractProblem} then if we set $A(t) = Df(t,x(t)),$ then we can rewrite the variational problem \eqref{eq:AbstractVariationalProblem} as the following non-autonomous linear equation
\begin{equation}\label{eq:linearNonautonmous}
\begin{cases}
   \frac{d}{dt}z(t) = Lz(t) + A(t)z(t),
   \\ z(t_0)=z^0.
\end{cases}
\end{equation}
For this equation, we consider the following assumption which for variational equation follows from the continuity of $x$ and \hyperref[as:BV1]{(BV1)},\hyperref[as:BV3]{(BV3)} when $A(t)=Df(t,x(t))$
:
\begin{itemize}
    \item[(BLn)]\label{as:BLn} For some $t_0\in\mathbb{R}$ and $T>0$ the map $A:[t_0,T+t_0]\to\mathcal{L}(Y;W)$ is a continuous function.
\end{itemize}

\begin{lemma}
    If the system \eqref{eq:linearNonautonmous} satisfies conditions \hyperref[as:BL1]{(BL2)},\hyperref[as:BL2]{(BL2)},\hyperref[as:BLn]{(BLn)} then for every $x^0\in Y$ it has the unique solution $z(t):[t_0,t_0+T]\rightarrow Y$. Moreover, the mapping $J(t,t_0)x^0 = z(t)$ is linear and bounded, leading from $Y$ to $Y$ and satisfies the following formula
    \begin{equation*}
        J(t,t_0) = e^{tL} + \int_{t^0}^te^{(t-s)L}A(s)J(s,t_0)ds.
    \end{equation*}
\end{lemma}

\subsection{Variational equation for non-autonomous Chafee–-Infante PDE}
With the same spaces and the same operators $L$ and $f$ as in Section \ref{sec:nonautonomousChafee–InfanteEquation}, we formulate the variational equation for the 
Chafee–Infante PDE. First, let us observe that the mapping 
$Df(t,u):\mathbb{R}\times Y\to \mathcal{L}(Y)$ given by the formula 
\begin{equation*}
    Df(t,u)h =- 3b(t)u^2h, 
\end{equation*}
is the Fr\'{e}chet derivative of $f(t,\cdot):Y\to Y$ at point $u$. Moreover $Df(t,u)$ satisfies conditions \hyperref[as:BV1]{(BV1)}-\hyperref[as:BV1]{(BV3)} which can be easily observed from the formula
\begin{equation*}
    f(t,u+h)-f(t,u) = -3b(t) u^2h -b(t)(3uh^2 +h^3).
\end{equation*}
Therefore, we can consider the system
 \begin{equation}\label{eq:Chafee–InfanteVariational}
    \begin{cases}
        u_t= u_{xx}+\lambda u -b(t)u^3\  \text{ for }\; (x,t)\in (0,\pi)\times(t_0,\infty),\\
        h_t = h_{xx}+\lambda h -3b(t)u^2h \ \text{ for }\; (x,t)\in (0,\pi)\times(t_0,\infty),\\
        u(t,x) = h(t,x) = 0\; \text{for}\; (x,t)\in \{0,\pi\} \times (t_0,\infty),
    \\u(t_0,x) = u^{0}(x),\quad\ h(t_0,x) = h^{0}(x) \textrm{ for}\ x\in(0,\pi).
\end{cases}
\end{equation}
From Lemma \ref{rem:VartionalAsNormalAbstract} we know that the above problem has a unique local solution for all $u^0\in C_0(0,\pi)$ and $h^0\in C_0(0,\pi).$ 

\subsection{Variational equation for Burgers equation with fractional Laplacian}
We formulate the variational equation for the 
Fractional Burgers equation. First, let us observe that the mapping 
$Df(t,u):\mathbb{R}\times Y \to \mathcal{L}(Y;W)$ given by the formula
\begin{equation*}
    Df(t,u)h = 2\nu(uh)_x, 
\end{equation*}
is the Fr\'{e}chet derivative of $f(t,\cdot):Y\to W$ at point $u$. We have
\begin{align*}
        \norm{f(t,u+h)-f(t,u) -2\nu(uh)_x}_{W} = \nu \norm{(h^2)_{x}}_{W}&\leq 2\nu\left(\norm{h_{x}^2}_{L^2} +\norm{hh_{xx}}_{L^2}\right) 
        \\&\leq 4\nu C \norm{h_{xx}}_{L^2}^2 =  4\nu C\norm{h}_{Y}^2,
\end{align*}
for some constant $C>0.$
Therefore,  $Df(t,u)$ satisfies conditions \hyperref[as:BV1]{(BV1)}-\hyperref[as:BV1]{(BV3)}.
So, we can consider the system
 \begin{equation}\label{eq:BurgersVariational}
    \begin{cases}
        u_t= -(-\Delta)^\alpha u+\nu(u^2)_x +g(t,x)\  \text{ for }\; (x,t)\in (0,\pi)\times(t_0,\infty),\\
        h_t = -(-\Delta)^\alpha h+2\nu(uh)_x\ \text{ for }\; (x,t)\in (0,\pi)\times(t_0,\infty),\\
        u(t,x) = h(t,x) = 0\ \text{for}\; (x,t)\in \{0,\pi\} \times (t_0,\infty),
    \\u(t_0,x) = u^{0}(x),\quad\ h(t_0,x) = h^{0}(x)\  \textrm{ for}\ x\in(0,\pi).
\end{cases}
\end{equation}
Lemma \ref{rem:VartionalAsNormalAbstract} ensures that the above problem has a unique local solution for all $u^0\in Y$ and $h^0\in Y.$ 
\section{Algorithm of integration}\label{sec:AlgorithmIntegration}

\subsection{Overview of \texorpdfstring{$C^0$}{Lg} algorithm}\label{sect22}

 For a given local processes  $\varphi:\Omega\to Y,$ a set $X(t_0)$ and a time step $\tau>0$, the goal of $C^0$ algorithm is to effectively construct another set $X(t_0+\tau)$, such that $\varphi(\tau+t_0,t_0 ,X(t_0)) \subset X(t_0+\tau)$. Moreover, it has to ensure that $\varphi(\tau+t_0,t_0,x)$ is well-defined for every $x \in X(t_0)$.

As the space $Y$ is infinite dimensional, we need to have a  suitable representation of sets from this space. 
In the abstract problem \eqref{eq:AbstractProblem} together with conditions \hyperref[as:BL1]{(BL1)}, \hyperref[as:BL2]{(BL2)}, \hyperref[as:Bf]{(Bf)}, \hyperref[as:BH]{(BH)}, \hyperref[as:BD]{(BD)}, we assume that  our phase space $Y$ is embedded in a Hilbert space $H$ with the basis $\{e_k\}_{k=1}^\infty .$ 
We represent $H = H_P\oplus H_Q.$
where $H_P = \text{span}\{e_1,\dots,e_n\} \cong \mathbb{R}^n$ and $H_Q$ is an orthogonal complement of $H_P$ in $H.$ By $P,Q$ we will denote orthogonal projections on the spaces $H_P$ and $H_Q.$ 
Namely we represent the sets $X(t_0), X(t_0+\tau) \subset Y$ as
\begin{equation*}
    X(t_0) =  X_{P}(t_0)  + X_{Q}(t_0)\qquad X(t_0+\tau) =  X_{P}(t_0+\tau) + X_{Q}(t_0+\tau),
\end{equation*}
where $X_{P}(t_0), X_{P}(t_0+\tau) \subset H_P$ are sets in finite dimensional space and $X_{Q}(t_0),X_{Q}(t_0+\tau) \subset H_Q$ are infinite dimensional sets, which need some finite representation.
We realize it by considering some inequalities which are uniform with respect to coefficients in the Fourier series. 
This decomposition of the set will be called a $P, Q$ representation. We also decompose nonlinearity into functions $f^1:\mathbb{R}\times (Y\cap H_{P})\to Y\cap H_{P}$ and $f^2:\mathbb{R}\times (Y\cap H_{P})\times (Y\cap H_{Q})\to Y$ which satisfy 
\begin{equation}\label{eq:fdecomposition}
    f(t,x) = f^1(t,Px)+f^2(t,Px,Qx), 
\end{equation}
where $f^1$ is the Galerkin projection of the nonlinearity $f$ on the $H_P$ space (that is, $f^1(t,Px)= Pf(t,Px)$).
For a more detailed overview of the $C^0$ algorithm, see \cite{BrusselatorPrzyjetaPraca}.

\subsection{Overview of \texorpdfstring{$C^1$}{Lg} algorithm} \label{sec:Overview_C1_algorithm}
In the computer assisted results on the local attractivity of  periodic solutions for considered PDEs presented in Section \ref{sec:Chafee-Infante results}, we use Lemma \ref{lem:AbstractNonAutonomousAttracting}. To use this result, we need to perform a computation that allows us to efficiently estimate the norm of solution to the variational equation associated with the original problem. To this end, we use the $C^1$ integration algorithm presented below.

From the $C^0$ algorithm, we get estimate (in terms of the evolving sets) on the evolution of the $\varphi(t,t_0,x)$ for $x \in X(t_0).$ But the $C^0$ algorithm does not give the information about the behavior of the variational equation, that is we do not control the Fr\'{e}chet derivative $D_{x^0}\varphi(t,t_0,x^0) = V(t,t_0,x^0).$ 

To control this mapping we can  apply $C^0$ algorithm to the system \eqref{eq:AbstractVariationalProblem}, to get estimation on the evolution of \eqref{eq:AbstractVariationalProblem} for some given $h^0\in Y,$ or, in other words, the estimation on  $V(t,t_0,X(t_0))h^0.$ If the phase space were finite-dimensional, then this approach would allow us to obtain information about the entire linear mapping. Namely, if $\{e_1,\ldots,e_n\}$ were the basis of this phase space, then an estimate on $V(t,t_0,X(t_0))e_i$ for every $i\in\{1,\ldots,n\}$ would give an estimation on the mapping $V(t,t_0,X(t_0)).$ But, as we work with infinite dimensional phase spaces, we need to adapt this approach. To this end, we consider sets $X^h(t_0)\subset Y\cap H_Q$ such that for some $n\in\mathbb{N}^0$ and for every $h\in Y$ there exist $a_1,\ldots,a_n,a_{n+1}\in\mathbb{R}$ such that
\begin{equation*}
    h\in a_1e_1+\ldots + a_ne_n + a_{n+1}X^h(t_0).
\end{equation*}
Then from the linearity of $V(t,t_0,x)$ we obtain that
\begin{equation*}
    V(t,t_0,x^0)h\in a_1V(t,t_0,x^0)e_1 +\ldots +a_nV(t,t_0,x^0)e_n + a_{n+1}V(t,t_0,x^0)X^h(t_0). 
\end{equation*}
The set $X^h(t_0)$ can potentially be unbounded in the phase space $Y$, which is the main difficulty. \textcolor{black}{In particular, in our proof of Theorem \ref{th:Chaffe-InfanteAtractingOrbit1} on attractivity of the periodic solution for the Chafee--Infante system  we use as the initial data the set which has form
\begin{equation*}
   X^{h}(t_0)= \left\{x\in C_0(0,\pi): x_k\in [-1,1] \text{ for } k\geq 3  \text{ and } x_k = 0 \text{ for } k = 1,2,3\right\},
\end{equation*}
which is unbounded in $C_0(0,\pi)$.
}
To overcome this difficulty, we develop the estimates of the convolution of the cosine and sine series, in which the coefficients of the sine series can be unbounded or slowly decaying (see Lemmas \ref{lem:costimesSinDiverging} and \ref{lem:costimesSinSlowConvering}). We also prove Lemma \ref{lemma:enclosureLinearBounded}, which allows us to find the enclosure for  unbounded sets of  initial data for the linear non-autonomous equations in  Banach spaces. Combining these two tools, we  rigorously integrate forward in time the unbounded sets.
So in general, we firstly estimate the mapping $V(t,t_0,x)h$ for $h\in H_P$ by applying the $C^0(0,\pi)$ algorithm for the basis of $H_P.$ Then, for $h\in H_Q\cap Y$ we use the estimate coming from integration of the set $X^h(t_0).$  In a similar way, for the fractional Burgers equation we consider the set 
\begin{equation*}
    X^{h}(t_0)=\left\{x\in H_\text{odd}^2(-\pi,\pi): x_k\in \frac{[-1,1]}{k^2} \text{ for } k\in\mathbb{N}\right\},
\end{equation*}
which is unbounded in the space $Y=H^2_\text{odd}(-\pi,\pi).$
\subsection{Representation of sets}\label{sec:descritionSet}
The important role in the algorithms is played by the sequences $\{V_i\}_{i=1}^\infty$ of intervals which we call infinite interval vectors. For a given infinite interval vector $V$ we denote the $i$-th interval by $V_i$ and its left and right ends by
$V_i^-,V_i^+$, respectively.
\textcolor{black}{
 \begin{definition}
     For a given  infinite interval vectors $\{V_i\}_{i=1}^\infty$ we call the set
      \begin{equation*}
          \{v\in Y : v_i\in V_i \}
      \end{equation*} a  representation of infinite interval vector.
 \end{definition}
}
It is possible that representation of an interval vector is not a bounded set in $Y$. 
Whenever it will not lead to confusion we  use the same notation for infinite interval vectors and their representations.
We define several useful operations on the infinite interval vectors. First, for infinite interval vector  $V$ we
define the quantities
\begin{equation*}
    V^- = \{[V^-_i,V^-_i]\}_{i=1}^\infty,
    \quad
    V^+ = \{[V^+_i,V^+_i]\}_{i=1}^\infty.
\end{equation*}
For a given interval $I$ the we define multiplication an infinite interval vector $V$ by the interval $I$ as
\begin{equation*}
I\, V= V\, I =  \{I V_i\}_{i=1}^\infty.
\end{equation*}
For two infinite intervals vectors
$V$ and
$W$ we define their sum and element-wise product as
\begin{equation*}
    V+W = \{V_i + W_i\}_{i=1}^\infty,\quad V*W = \{V_iW_i\}_{i=1}^\infty.
\end{equation*}
We say that vector $V$ is a subset of $W$ and denote by
$V\subset W$ if and only if
$V_i\subset W_i$ for every $i\in\mathbb{N}.$
Additionally we define
$V\subset_\text{int} W$ if and only if
$V_i\subset \text{int} W_i$ for every $i\in\mathbb{N}.$ We define the convex hull of two infinite intervals vectors as
\begin{equation*}
  \text{conv}\{V,W\} =\{ \text{conv}\{V_i\cup W_i\} \}_{i=1}^\infty.
\end{equation*}
The intersection of two infinite vectors is defined in the following way
\begin{equation*}
  V \cap W =\{ V_i\cap W_i \}_{i=1}^\infty.
\end{equation*}

In the algorithm we consider such sets $X = X_P + X_Q$, for which there exist infinite interval vectors whose representation contains the set $X.$
\Doktorat{
\input{Wersja1/Algorytm/SetsForBrusselator}
}

In the case of the Chafee-Infante equation we consider the infinite interval vector $U$ with the polynomial decay i.e. for which there exists and $s>1$ and $n\in\mathbb{N}$ such that for every sequence $u_i\in U_i$ we have
\begin{equation}\label{eq:representationInfiniteVectorsChafeInfante}
    u_i\in \frac{[C_U^-,C_U^+]}{i^s}\quad
    \text{for $i\geq n$},
\end{equation}
where $C_U^-\leq C_U^+$ are constants. \Doktorat{The sets $X=X_P+X_Q$ are in the similar form as in Brusselator system. Namely $X_P$ are the subset of $H_P = \text{span}\{\sin(x),\ldots,\sin(kx)\},$ and $X_Q$ are given by the formula \eqref{eq:representationInfiniteVectorsChafeInfante} with $s>1.$ 
Similarly, as in the Brusselator system we can restrict to the space where $u$ has only odd nonzero coefficients.}\article{
For the Chafee--Infante equation the sets $X_P,$ from decomposition of $X,$ are the subset of $H_P = \text{span}\{\sin(x),\ldots,\sin(kx)\},$ and $X_Q$ are given by the formula \eqref{eq:representationInfiniteVectorsChafeInfante} with $s>1.$ Additionally we can restrict space $C^0$ to the subspace of functions which have only odd nonzero coefficients, as it is invariant for the equation.
}

For the variational equation for the Chafee--Infante system we \Doktorat{again }use pair of infinite interval vectors $(U,U^h)$ with polynomial bounds, which means that for some $n\in\mathbb{N}$ and $s_1,s_2\in\mathbb{R}$ and every sequence $u_i\in U_i$ and $h_i\in U^h_i$ we have
\begin{equation}\label{eq:representationInfiniteVectorsChafeInfanteVariational}
    u_i\in \frac{[C_U^-,C_U^+]}{i^{s_1}},\quad
    h_i\in \frac{[C_{U^h}^-,C_{U^h}^+]}{i^{s_2}},\quad
    \text{for $i\geq n$},
\end{equation}
where $C_U^-\leq C_U^+$ and $C_{U^h}^-\leq C_{U^h}^+$ are constants. Note that now $U$ and $U^h$ can have different rate defining the polynomial bound. This is crucial for integrating the unbounded sets and $C^1$ computations for this variational equation. \article{The pair $(U,U^h)$ can be easily re-indexed into one interval vector.}
The element $(u,h)\in H_Q\cap Y$ belongs to $X_Q$ if it satisfies \eqref{eq:representationInfiniteVectorsChafeInfanteVariational} with $s_1>1,$ and $s_1+s_2 > 1$ in case when $s_2\leq 0,$ and $s_1-s_2 > 1$ in the case when $s_2\in(0,1],$ and no further restrictions if $s_2\geq 1.$ Note that this allows that the representation of $ U^h_i$ is a possibly unbounded set.

For the Burgers equation we can use the same representation \eqref{eq:representationInfiniteVectorsChafeInfante} as for the Chafee--Infante equation.
For the variational equation associated with the Burgers equation we can also use the representation \eqref{eq:representationInfiniteVectorsChafeInfanteVariational}.
For the rates $s_1$ and $s_2$, we assume $s_1 > 1$ and $s_2 > 2$.  
\subsection{Computation of nonlinear terms}\label{sect24} In the course of the algorithm for a given set $X=X_P+ X_Q$ and the time interval $[t_0,t_0+\tau]$, we need to compute the set $X^1$ such that $ f([t_0,t_0+\tau],X)\subset X^1.$
This set, represented as $X^1 = X_P^1+ X_Q^1$  constitutes the estimates for $f([t_0,t_0+\tau],X)$ and hence it should be as small as  possible.
The set $X^1$ is used in further steps of algorithm.

\Doktorat{
For the Brusselator problem we have 
\begin{equation*}
    f(u,v) = (u^2v + A\sin(x),Bu-u^2v).
\end{equation*}
For functions $(u,v)\in Y$ the components of $f$ can be represented in the following sine Fourier series with the coefficients $a_i, b_i$ dependent on $u, v$
\begin{equation*}
    u^2v + A\sin(x) = \sum_{i=1}^\infty a_i \sin(i x),\quad Bu-u^2v = \sum_{i=1}^\infty b_i \sin(i x).
\end{equation*}
Set $X^1_P$ is represented as the cuboid (or parallelepiped) in $H_P$ and $X^1_Q$  is  described by the polynomial decay of  Fourier coefficients.
The first step of finding $X^1$ is estimating the square of $u$ which is represented in the cosine Fourier series. For this purpose we use Lemma  \ref{lem:sinTimesSin}.
Having computed the coefficients of $u^2$ we need to find the coefficients of $(u^2)v.$ To this end we use Lemma \ref{lem:cosTimesSin}. Finally, we use Lemma
\ref{lem:add} to compute the representation of sums of particular terms which appear in the definition of $f$.
We also need to compute the image $L(X)$ but as $L$ is a diagonal operator we only need to multiply every given coefficient by the corresponding eigenvalue of $L.$ The result of such multiplication is given in Lemma \ref{lem:mult}.
Additionally in the algorithm we  need the decomposition
 \begin{equation}\label{eq:fdecomposition}
     f(t,p+q) = f(t,p) +f_2(t,p,q),
 \end{equation}
 where $p\in H_P$, $q\in H_Q$ and $f_2(t,p,q) = f(t,p+q)-f(t,p)$.  This decomposition is required for the formulation of the differential inclusion.
For the Brusselator system we can write
$f(u_P+u_Q,v_P+v_Q) = f(u_P,v_P) + f_2(u_P,v_P,u_Q,v_Q)$ where
\begin{equation*}
    f_2(u_P,v_P,u_Q,v_Q) = (
    (2u_Pu_Q+u_Q^2)(v_P+v_Q)+ u_P^2v_Q,
    Bu_Q-(2u_Pu_Q+u_Q^2)(v_P+v_Q)- u_P^2v_Q
    ).
\end{equation*}}
The nonlinearity in the Chafee--Infante problem, for which we perform such procedure, is given by the following function
\begin{equation*}
    f(t,u) = -b(t)u^3.
\end{equation*}
\Doktorat{
To compute $u^3$ we use Lemmas \ref{lem:sinTimesSin} and
\ref{lem:cosTimesSin} in the similar way as in the Brusselator system.} The first step of finding $X^1$ is estimating the square of $u$ which is represented in the cosine Fourier series. For this purpose we use Lemma  \ref{lem:sinTimesSin}.
Having computed the coefficients of $u^2$ we need to find the coefficients of $(u^2)u,$ using Lemma \ref{lem:cosTimesSin}. Then we  multiply obtained estimate on $u^3$ by the estimate of the term $-b(t)$ on a given time interval, using the Lemma \ref{lem:mult}.
\begin{comment}
   The function $f_2$ from the decomposition \eqref{eq:fdecomposition} has the following form
\begin{equation*}
    f_2(t,u_P,u_Q)= -(A\sin(2\pi t) + B)(u_Q(u_Q+2u_P)(u_P+u_Q) + u_P^2u_Q).
\end{equation*} 
\end{comment}

In the case of variational equation for the Chafee--Infante problem, we have that
\begin{equation*}
    f(t,u,h) = -b(t)(u^3,3u^2h).
\end{equation*}
The main difficulty here lies in computing the term $u^2h.$ In case when $s_1$ and $s_2$ are both larger than $1$ we can use Lemmas \ref{lem:sinTimesSin} and
\ref{lem:cosTimesSin} with $s= \min\{s_1,s_2\} .$ In the case when $s_2\leq 1$ we  first compute $u^2$ and then use Lemma \ref{lem:costimesSinDiverging} when $s_2\leq 0$ and $s_1+s_2>1$ or Lemma \ref{lem:costimesSinSlowConvering} when $s_2\in (0,1]$ and $s_1-s_2>1.$ So in this case we have 
\begin{equation*}
  (u^2h)_i\in \frac{[D^-,D^+]}{i^{s_2}},  
\end{equation*}
 where constants $D^-\leq D^+$ follow from the application of Lemmas \ref{lem:cosTimesSin}, \ref{lem:sinTimesSin}, \ref{lem:costimesSinSlowConvering}, and $s_2$ is the same as for $h.$ 
 
 For the Burgers equation, we need to compute nonlinear term $(u^2)_x.$ We firstly compute $u^2$ using Lemma \ref{lem:cosTimesSin}, furthermore we estimate the space derivative which is straightforward. For the variational equation, we need to estimate $(uh)_x.$ This is simpler than in the case of the Chafee--Infante equation, because we start with the set with decay $s_2=2.$ We use Lemma \ref{lem:sinTimesSin} to obtain the estimate on $(uh)$ as we work with decays $s_1,s_2\geq 2.$ Then we obtain the estimate for the derivative.

\subsection{Computation of  the enclosure}\label{sec:enclosure} We start this section with the definition of an enclosure.
\begin{definition}
The set $X([t_0,t_0+\tau])$ is a enclosure of the set $X^0 \subset Y$ for times $t_0\in\mathbb{R}$ and $\tau > 0$ if $\varphi(t,t_0,X^0) \subset X([t_0,t_0+\tau])$ for every $t\in [t_0,t_0+\tau]$.
\end{definition}
The following Lemma can be used to validate if for the given set of initial data $X^0$ the set $X^0+Z$ is an enclosure for the abstract problem \eqref{eq:AbstractProblem}.

\begin{lemma}\label{lemma:enclosure}
 Assume that conditions \hyperref[as:BL1]{(BL1)}, \hyperref[as:BL2]{(BL2)}, \hyperref[as:Bf]{(Bf)}, \hyperref[as:BH]{(BH)}, \hyperref[as:BD]{(BD)} hold and let $\{X^0_i\}_{i=1}^\infty$ be a countable family of intervals $X^0_i = [x_i^-,x_i^+]$ such that the set $X^0:= \set{x\in Y: x_i\in X_i^0}$ is bounded in $Y$. Moreover, define another set
$Z:= \set{z\in Y: z_i\in Z_i},$ also bounded in $Y$, where $Z_i = [z^-_i,z^+_i]$ are intervals that contain zero. We assume that for some $t_0\in \mathbb{R},$ $\tau>0$ and every $i\in \mathbb{N}$ there holds
\begin{equation}\label{eq:EnclosureCondition}
    [g_i^-,g_i^+]\cap
    [h_i^-,h_i^+]\subset \textrm{\rm int } Z_i,
\end{equation}
where
\begin{equation}\label{eq:gminusgplus}
        g^-_i =
        \min_{s\in[0,\tau]}
        \left[(e^{\lambda_i s} - 1)
        \left(\frac{f^-_i}{\lambda_i} +x^-_i\right)\right],
        \quad
        g^+_i =
        \max_{s\in[0,\tau]}
        \left[(e^{\lambda_i s} - 1)
        \left(\frac{f^+_i}{\lambda_i} +x^+_i\right)\right],
    \end{equation}
\begin{equation}\label{eq:hminushplus}
    h^-_i = \min_{s\in[0,\tau]}
        s(\lambda_i(x_i^-+z_i^-) +f_i^-) ,
        \quad
        h^+_i =
        \max_{s\in[0,\tau]}
        s(\lambda_i(x_i^+ + z_i^+) +f_i^+) ,
\end{equation}
     and $f^-_i,f^+_i\in\mathbb{R}$  satisfy
    \begin{equation}
        f^-_i\leq f_i([t_0,t_0+\tau],X^0+Z)\leq f^+_i.
    \end{equation}
Then for every $x^0\in X^0$ there exists a continuous function $x:[t_0,t_0+\tau]\to Y$
which is the unique solution to $\eqref{eq:AbstractProblem}.$ Moreover, for every $t\in[t_0,t_0+\tau]$ we have
\begin{enumerate}
    \item  $x_i(t)\in x^0_i+[0,t-t_0]F_i([t_0,t_0+\tau],X^0+Z),$ for every $i\in \mathbb{N},$
    \item $x^-_i+g_i^-\leq x_i(t)\leq x^+_i+g_i^+$ for every $i\in \mathbb{N},$
    \item $x_i(t)\in e^{(t-t_0) \lambda_i} x^0_i + \frac{e^{\lambda_i (t-t_0)} -1}{\lambda_i}[f^-_i,f^+_i]$ for every $i\in \mathbb{N}.$
\end{enumerate}
\end{lemma}
The proof of above Lemma is simple adaptation of \cite[Lemma 2.7]{BrusselatorPrzyjetaPraca} to the case of non autonomous nonlinearity.

The next lemma is a modification of the previous lemma to the case of non-autonomous linear equation \eqref{eq:linearNonautonmous}. The main difference is that now the sets $X^0$ and $Z$ do not have to be bounded in $Y.$ This allows to get uniform bounds for evolution from any initial condition in  $Y.$ The possibility to handle unbounded sets follows from the fact that for linear equation we can globally estimate the Lipschitz constant.
\begin{lemma}\label{lemma:enclosureLinearBounded}
    Assume that conditions \hyperref[as:BL1]{(BL1)}, \hyperref[as:BL2]{(BL2)}, \hyperref[as:BLn]{(BLn)}, \hyperref[as:BH]{(BH)}, \hyperref[as:BD]{(BD)} hold and let $\{X^0_i\}_{i=1}^\infty$ be a countable family of intervals $X^0_i = [x_i^-,x_i^+]$ such that the set $X^0:= \set{x\in Y: x_i\in X_i^0}$. Moreover, define another set
$Z:= \set{z\in Y: z_i\in Z_i},$ in $Y$, where $Z_i = [z^-_i,z^+_i]$ are intervals that contain zero. We assume that for some $t_0\in \mathbb{R},$ $\tau>0$ and every $i\in \mathbb{N}$ there holds
\begin{equation}\label{eq:EnclosureCondition1}
    [g_i^-,g_i^+]\cap
    [h_i^-,h_i^+]\subset \text{\rm int } Z_i,
\end{equation}
where
\begin{equation}\label{eq:gminusgplus1}
        g^-_i =
        \min_{s\in[0,\tau]}
        \left[(e^{\lambda_i s} - 1)
        \left(\frac{f^-_i}{\lambda_i} +x^-_i\right)\right],
        \quad
        g^+_i =
        \max_{s\in[0,\tau]}
        \left[(e^{\lambda_i s} - 1)
        \left(\frac{f^+_i}{\lambda_i} +x^+_i\right)\right],
    \end{equation}
\begin{equation}\label{eq:hminushplus1}
    h^-_i = \min_{s\in[0,\tau]}
        s(\lambda_i(x_i^-+z_i^-) +f_i^-) ,
        \quad
        h^+_i =
        \max_{s\in[0,\tau]}
        s(\lambda_i(x_i^+ + z_i^+) +f_i^+) ,
\end{equation}
     and $f^-_i,f^+_i\in\mathbb{R}$  satisfy
    \begin{equation}
        f^-_i
        \leq 
        (A([t_0,t_0+\tau])(X^0+Z))_i
        \leq 
        f^+_i.
    \end{equation}
Then for every $x^0\in X^0$ there exists a continuous function $x:[t_0,t_0+\tau]\to Y$
which is a unique solution to \eqref{eq:linearNonautonmous}. Moreover for every $t\in[t_0,t_0+\tau]$ we have
\begin{enumerate}
    \item  $x_i(t)\in x^0_i+[0,t-t_0]F_i([t_0,t_0+\tau],X^0+Z),$ for every $i\in \mathbb{N},$
    \item $x^-_i+g_i^-\leq x_i(t)\leq x^+_i+g_i^+$ for every $i\in \mathbb{N},$
    \item $x_i(t)\in e^{(t-t_0) \lambda_i} x^0_i + \frac{e^{\lambda_i (t-t_0)} -1}{\lambda_i}[f^-_i,f^+_i]$ for every $i\in \mathbb{N}.$
\end{enumerate}
\end{lemma}
\begin{proof}
We define the operator $\mathcal{T}:C([t_0,t_0+\tau];Y)\to C([t_0,t_0+\tau];Y) $ in the following way
\begin{equation}
    \mathcal{T}(g)(t) = e^{L(t-t_0)}x^0 + \int_{t_0}^t e^{L(t-s)}A(s)g(s)ds.
\end{equation}
We consider the set
\begin{equation*}
    S_\tau = \{g\in C([t_0,t_0+\tau];Y):g(t_0)=x^0\; \text{and for every $t\in[t_0,t_0+\tau]$ we have } g(t)\in X^0+ Z\}.
\end{equation*}
We will show that the map $\mathcal{T}$ satisfies assumptions of the Banach fixed point theorem.
The proof of the fact that $\mathcal{T}(S_\tau)\subset S_\tau$ is the same as in the corresponding argument  in \cite[Lemma 2.7]{BrusselatorPrzyjetaPraca}.
 To show that $\mathcal{T}$ is a contraction  we equip the space  $C([t_0,t_0+\tau];Y)$ with the  norm $\|y\|_{\alpha} = \sup_{t\in[t_0,t_0+\tau]} \norm{y(t)}_Y e^{-\alpha(t-t_0)},$ where $\alpha>0$ is an appropriately chosen positive constant. For $g_1,g_2\in S,$ with use of the conditions \hyperref[as:BL2]{(BL2)} and \hyperref[as:Bf]{(Bf)}, we compute
\begin{align*}
     \left\|\int_{t_0}^t e^{L(t-s) } A(s)(g_2(s) -g_1(s))ds\right\|_Y
     &\leq M_2 e^{T \kappa_2} \int_{t_0}^t \frac{1}{(t-s)^{\gamma}}
     e^{s\alpha } e^{-s\alpha} \norm {A(s)(g_2(s) -g_1(s))}_{Y^1}  ds\\
      &\leq
      M_2e^{\tau \kappa_2}
      \sup_{s\in[t_0,t_0+\tau]} \norm{A(s)}_{\mathcal{L}(Y;Y_1)}
  \norm { g_1-g_2}_{\alpha}
      \int_{t_0}^t\frac{1}{(t-s)^\gamma}e^{\alpha s} ds
      .
\end{align*}

From Proposition \ref{eq:integralfracExtEstimation}  we can find $\alpha>0$ such that every $\delta >0$ we have
\begin{equation*}
    \int_{t_0}^t\frac{1}{(t-s)^\gamma}e^{\alpha(t-s)} ds
    \leq 
    \frac{e^{\alpha(t_0-t)}}{2M_2e^{\tau\kappa_2}\sup_{s\in[t_0,t_0+\tau]} \norm{A(s)}_{\mathcal{L}(Y;Y_1)}},
\end{equation*}
then for every $t\in[t_0,t_0+ \tau]$ we have
\begin{equation*}
    \left\|\int_{t_0}^t e^{L(t-s) }A(s)
    \left(g_2(s) -g_1(s)\right)ds\right\|_Y \leq \frac{1}{2}e^{\alpha(t-t_0)}\norm{g_1-g_2}_{\alpha}.
\end{equation*}
Hence, $\mathcal{T}$ is a contraction on the set $S_\tau$ equipped with the norm $\norm{\cdot}_\alpha$. The Banach fixed point theorem gives the unique fixed point of the map $\mathcal{T}$, and the proof is complete.
\end{proof}
    Now we recall the algorithm of finding the enclosure, as we are using it in $C^0$ and $C^1$ computations.
   In Lemma \ref{lemma:enclosure} the sets $X^0,Z$ are
 infinite interval vectors, whose representations are bounded.
 If $X^0$ is not an infinite interval vector, we can still find an enclosure, provided we can choose the interval vector $X^1$ such that $X^0 \subset X^1$ and $X^1$ is bounded.
 The following algorithm takes as an input the infinite interval vector $X^0,$ the infinite interval vector $Z$ such that  $0\in Z_i$ for every $i\in\mathbb{N},$ initial time $t_0\in\mathbb{R}$ and the step $\tau>0.$ The algorithm  validates if $X^0 + Z$ is an enclosure for the process given by the equation \eqref{eq:AbstractProblem} for $\tau>0$ by checking assumption \eqref{eq:EnclosureCondition} of Lemma  \ref{lemma:enclosure}.

 \begin{enumerate}
     \item Find the infinite vectors $V^f$ and $V^L$
     such that $f([t_0,t_0+\tau],X^0+Z) \subset V^f $ and
     $L(X^0+Z)\subset V^L.$
     \item Construct the infinite vector $V^{\text{ND}}$ such that
     $[0,\tau](V^L+V^f)\subset V^{\text{ND}} $
    \item Compute the infinite vectors $V^{E_1}$ and $V^{L^{-1}}$ which satisfy
    \begin{equation*}
        e^{[0,\tau]\lambda_i} - 1 \subset (V^{E_1})_i,
        \qquad
        \frac{1}{\lambda_i} \in (V^{L^{-1}})_i \ \ \textrm{for every}\ \ i\in \mathbb{N}.
    \end{equation*}
     \item Compute the infinite vectors $V^{\text{D1}}$ and $V^{\text{D2}}$ which satisfy
     \begin{equation*}
         V^{E_1}*((V^f)^+*V^{L^{-1}} + X^+  ) \subset V^{\text{D1}},
         \quad
         V^{E_1}*((V^f)^-*V^{L^{-1}} + X^- ) \subset V^{\text{D2}}.
     \end{equation*}
     \item Find infinite interval vectors  $V^{\text{D}}$ such that
     \begin{equation*}
        \text{conv}\{V^{\text{D1}}, V^{\text{D2}} \} \subset V^{\text{D}}.
     \end{equation*}
     \item Compute $Z^1$ such that  $V^\text{ND}\cap V^\text{D}\subset Z^1.$
     \item Check if $Z^1 \subset_{\text{int}} Z$ holds.

 \end{enumerate}
 For every $i\in\mathbb{N}$ the interval $[h_i^-,h_i^+]$ from  \eqref{eq:hminushplus} is contained in the interval $V^{ND}_i,$ where $V^{ND}$ is the infinite interval vector obtained in step (2).
 Similarly, for every $i\in\mathbb{N}$ the interval $[g_i^-,g_i^+]$ from \eqref{eq:gminusgplus}  is contained in the interval $V^{D}_i,$ where $V^{D}$ is the
 infinite interval vector from step (5). If the condition in step (7) is satisfied, we have $V^{ND}_i\cap V^{D}_i \subset \text{int} Z_i$ for every $i\in\mathbb{N},$ which implies that the assumptions of Lemma \ref{lemma:enclosure} are satisfied and the set $X^0 + Z$ is an enclosure for our initial data $X^0$ and the time step $\tau>0$ .

 The same algorithm also works for  the non-autonomous linear equation \eqref{eq:AbstractVariationalProblem}, \eqref{eq:linearNonautonmous} even when the sets $X^0$ and $Z$ are unbounded.
 In the case when representations of $X^0$ and $Z$ are bounded in $Y$ we can apply the Lemma \ref{lemma:enclosure}, to justify the correctness of the algorithm. For the non-autonomous linear equation, \eqref{eq:linearNonautonmous} considered in Lemma \ref{lemma:enclosureLinearBounded}, the representations of $X^0$ and $Z$ do not have to be bounded, so the algorithm can work also in the unbounded case. For the variational equation \eqref{eq:AbstractVariationalProblem} we can integrate the initial set in the form $X^0 = X^{0,x}\times X^{0,h}$ and $Z = Z^x\times Z^{h},$ the sets $X^{0,x},Z^x$ have to be bounded but sets $X^{0,h}$ and $Z^{h}$ can be unbounded in $Y.$ To justify the correctness of the algorithm in this case we firstly apply Lemma \ref{lemma:enclosure} to show that $X^{0,x}+ Z^x$ is the enclosure for the $x$ variable of the equation \eqref{eq:AbstractVariationalProblem}. Then by treating the variational part of the equation as the non-autonomous linear equation which satisfies the   conditions of Lemma \ref{lemma:enclosureLinearBounded} we observe that $X^{0,h}+ Z^h$ is the enclosure of  the $h$ variable of the equation \eqref{eq:AbstractVariationalProblem}.

 If condition $(7)$ does not hold, we can modify the set $Z$ and repeat the validation procedure. The reasonable guess is to take $Z = [0,c]Z^1,$ where $c>1.$ Another possibility is to decrease the time step. The following Lemma \ref{lem:whyItWorks} shows that for a certain class of sets, it is always possible to find the enclosure using the above algorithm with a sufficiently small time step $\tau>0$. The proof can be found in Lemma 2.7 in \cite{BrusselatorPrzyjetaPraca}.

 \begin{lemma}
 \label{lem:whyItWorks}
Let $\{a_i\}_{i=1}^\infty$ be a sequence of positive numbers. Assume that
\begin{itemize}
    \item Conditions \hyperref[as:BH]{(BH)}, \hyperref[as:BD]{(BD)} hold.
    \item For some $i_0\in\mathbb{N}$ we have $\lambda_i < 0$ for every $i\geq i_0.$
    \item For every set $A\subset Y$ such that
    \begin{equation*}
        \sup_{x\in A} |x_i| = O(a_i),
    \end{equation*}
    the following holds
    \begin{equation*}
        \sup_{t\in\mathbb{R}}\sup_{x\in A} \left|\frac{f_i(t,x)}{\lambda_i}\right| = o(a_i).
    \end{equation*}
\end{itemize}
We consider the sequences of the intervals
$\{X^0_i\}_{i=1}^\infty$ $\{Z_i\}_{i=1}^\infty$ given by $X^0_i = [x_i^-,x_i^+],$ $Z_i = [z_i^-,z_i^+],$ with $-\infty<x_i^-\leq x_i^+<\infty,$ and $-\infty<z_i^-<0< z_i^+<\infty.$  We define the sets $X^0:= \set{x\in Y: x_i\in X_i^0}$ and $Z:= \set{z\in Y: z_i\in Z_i^0}$. Additionally, we assume that
\begin{itemize}
    \item The intervals $X_i$ contain zero for
    every $i\geq i_0,$ where $i_0\in\mathbb{N}$
    and
    \begin{equation*}
         \quad \sup_{x_i\in X_i^0}|x_i| = O(a_i).
    \end{equation*}
    \item The sequences $z_i^+$ and $z_i^-$ satisfy
    \begin{equation*}
        z_i^+ = \Theta(a_i),
        \quad \text{and}\quad
        |z_i^-| = \Theta(a_i).
    \end{equation*}
\end{itemize}
Under these assumptions for every $t_0\in\mathbb{R}$ there exists $\tau>0$ such that condition \eqref{eq:EnclosureCondition} holds, that is,
\begin{equation}
    [g_i^-,g_i^+]\cap
    [h_i^-,h_i^+]\subset \text{\rm int } Z_i,
\end{equation}
 where $g^-_i,g^+_i,h^-_i,h^+_i$ are defined as in Lemma \ref{lemma:enclosure}.
\end{lemma}

\Doktorat{
Using the above lemma we deduce that for the Brusselator system we can always find an enclousure for any set described in the Section \ref{sec:descritionSet}, by choosing sufficiently short time-step.

  \begin{remark}
 For every $C>0,\;\varepsilon>0,\;s>1$ there exists $\tau>0$ such that for every initial data $(u^0,v^0)\in Y$ satisfying
 \begin{equation*}
     |u^0_k|\leq\frac{C}{|k|^s}
     \quad\text{and}\quad
     |v^0_k|\leq\frac{C}{|k|^s},
 \end{equation*}
 the solution $(u(t),v(t))$ to \eqref{eq:BrusselatorPDE} satisfies
 \begin{equation*}
     |u_k(t)|\leq\frac{C+\varepsilon}{|k|^s}
     \quad\text{and} \quad
     |v_k(t)|\leq\frac{C+\varepsilon}{|k|^s}
     \quad\text{for}\;t\in[0,\tau].
 \end{equation*}

 \end{remark}
 One can easily construct an example of a problem without this property. To this end, let us consider the logistic model of population growth with diffusion and homogeneous Dirichlet boundary conditions. The corresponding equation together with the initial and boundary conditions are the following
\begin{equation}\label{eq:logisticEquation}
\begin{cases}
 u_t = d  u_{xx} + u- u^2\   \text{for}\; (x,t)\in [0,\pi]\times(0,\infty),\\
u(t,x) = 0\; \text{for}\; (x,t)\in \{0,\pi\}\times (0,\infty),
\\u(0,x) = u^0(x),
\end{cases}
\end{equation}
 for $d>0.$ As we impose the Dirichlet boundary conditions, from the equation we observe that $u_{xx}$ is always $0$ at the boundary. But, in general, the fourth derivative $u_{xxxx}$ can have nonzero values at the boundary.  This implies that, the coefficients in the expansion of the solution in the sine trigonometric series  cannot have the arbitrarily fast polynomial decay.

 We write the equation  in above example in the form \eqref{eq:AbstractProblem} with $L = d u_{xx}+u$ and $f(u)=u^2.$ For the initial condition $u_0 = \sqrt{\frac{\pi}{8}}\sin(x)$ we have that
 \begin{equation*}
      f(u_0) = \frac{\pi}{8}(1-\cos(2x)) = \sum_{i=1}^\infty \frac{(-1+(-1)^{i} )}{-4n+n^3}\sin(ix).
 \end{equation*}
So using Lemma \ref{lemma:enclosure} we have a chance to find the time step $\tau>0$ and $C>0$ such that the solution satisfies
\begin{equation*}
    |u_k(t)|\leq \frac{C}{k^s}\quad \text{for}\quad t\in[0,\tau],
\end{equation*}
 only for $s\in(1,5).$ 

Note that for Brusselator problem, the constant term $A$ from the ODE 
\begin{equation}\label{eq:BrusselatorODE2}
	\begin{cases}
		\frac{du}{dt} = - (B+1)u +u^2v + A  \ \   \text{for}\ \  t\in \R,\\
		\frac{dv}{dt} = Bu - u^2v  \ \   \text{for}\ \  t\in \R.
	\end{cases}
\end{equation}
has been replaced by the term $A\sin(x)$ in our PDE extension of the Brusselator system. In principle, we could consider the PDE version of the problem with the constant term $A$, that is the following system 
\begin{equation}\label{eq:BrusselatorPDEILL}
\begin{cases}
 u_t = d_1  u_{xx}- (B+1)u +u^2v + A  \;  \text{for}\; (x,t)\in (0,\pi)\times(0,\infty),\\
 v_t = d_2 v_{xx} + Bu - u^2v  \;  \text{for}\; (x,t)\in (0,\pi)\times(0,\infty),\\
u(t,x)= v(t,x) = 0\; \text{for}\; (x,t)\in \{0,\pi\} \times (0,\infty),
\\u(0,x) = u^0(x),\ v(0,x) = v^0(x)\ \textrm{for}\ x\in(0,\pi).
\end{cases}
\end{equation}
In such a case, we observe that $u_{xx}(t,x) = -\frac{A}{d_1}\not = 0$ for the boundary points  $x\in \{0,\pi\}.$ and for every $t\in\mathbb{R}^+,$ which means that compatibly conditions are not met. 
Let us represent $u$ in term of the sine Fourier series \eqref{eq:decomposion}. Assume that the series $\sum_{i=1}^\infty |u_i(t)|i^2$ is convergent.  Then, we can differentiate the Fourier series  twice, and we obtain 
\begin{equation*}
    u_{xx}(t,x) = \sum_{i=1}^\infty -u_i(t)i^2\sin(i x)
\end{equation*}
Therefore $u_{xx}(t,x) = 0$ for $x\in\{0,\pi\},$ which is not true, and hence, by contradiction, the series  $\sum_{i=1}^\infty |u_i(t)|i^2$ has to diverge.   Therefore, the Fourier series of the solution cannot converge fast. This unwelcome effect does not occur when we consider term $A\sin(x)$ instead of $A$ in the system \eqref{eq:BrusselatorPDE}, because the function $Asin(x)$ on $[0,\pi]$ is a restriction of a smooth, odd and $2\pi$ periodic function.
}

\Doktorat{
 We deduce}\article{From above lemma, we deduce} that we can always find the enclosure for the variational equation \eqref{eq:Chafee–InfanteVariational}, for the considered systems. Note that we can find it also for the sets in which the variational part is unbounded. We present it use for Chafee--Infante equation in the following remark. 
\begin{remark}
    Let $C>0,\varepsilon>0, t_0\in\mathbb{R}$. Moreover, let $s_1 >1$ and  $s_2\in\mathbb{R}$ be such that one of the following conditions hold
    \begin{itemize}
        \item $s_1+s_2 > 1$ if $s_2\leq 0$,
        \item $s_1-s_2>1$ if $s_2\in (0,1]$,
        \item $s_2>1$. 
    \end{itemize}  
    There exist $\tau>0$ such that for every initial data $(u^0,h^0)\in C_0(0,\pi)\times C_0(0,\pi)$ satisfying
    \begin{equation*}
        |u_k^0|\leq\frac{C}{k^{s_1}},\quad |h_k^0|\leq\frac{C}{k^{s_2}}, 
    \end{equation*}
    the solution $(u(t),h(t))$ to \eqref{eq:Chafee–InfanteVariational} satisfies
    \begin{equation*}
             |u_k(t)|\leq\frac{C+\varepsilon}{|k|^{s_1}}
     \quad\text{and} \quad
     |h_k(t)|\leq\frac{C+\varepsilon}{|k|^{\min\{s_1,s_2\}}}
     \quad\text{for}\;t\in[t_0,t_0+\tau].
    \end{equation*}
    
\end{remark} 
 \subsection{Evolution of sets}\label{sec:Evolving set}
 We assume that $X([t_0,t_0+\tau]) = X_P([t_0,t_0+\tau])+ X_Q([t_0,t_0+\tau])$ is an enclosure for the initial data $X = X_P+ X_Q$ and the time step $\tau>0.$  The following procedure is used to find the set $X(\tau)$ such that $\varphi(\tau+t_0,t_0,X )X\subset X(\tau).$

\begin{enumerate}
    \item Compute an infinite interval vector $V$ such that
    \begin{equation*}
        f^2([t_0,t_0+\tau],X_P([t_0,t_0+\tau]),X_Q([t_0,t_0+\tau]))\subset V,
    \end{equation*}
    where $f^2$ comes from the splitting of the nonlinearity \eqref{eq:fdecomposition}.
    \item Solve the system of ordinary differential inclusions
    \begin{equation}\label{inclusion}
        \frac{d}{dt}Px \in L Px+Pf(t,Px) + PV,
    \end{equation}
        \item Compute the infinite vectors $V^{E_1}, V^{E_2}$ and $V^{L^{-1}}$ which satisfy
    \begin{equation*}
        e^{\tau\lambda_i} -1 \in (V^{E_1})_i,
        \qquad
        e^{\tau\lambda_i}  \in (V^{E_2})_i,
        \qquad
        \frac{1}{\lambda_i} \in (V^{L^{-1}})_i \ \ \textrm{for every}\ \ i\in \mathbb{N}.
    \end{equation*}

    \item Compute $V^3$ such that
        $V^{E_2} * X + V^{E_1} * V^{L^{-1}} * V^2 \subset V_3$.

    \item Return the set $(PV_3 \cap X_{P1}) + QV_3$.
\end{enumerate}

The detailed description how to rigorously solve the differential inclusion \eqref{inclusion} can be found in \cite{ControlKapelaZgliczynski}.
For the step (2) we can consider differential inclusion only on part of variables represented explicitly.
This can be beneficial for the computational time.\Doktorat{
 As, for the Brusselator system, we use the infinite interval vectors of type \eqref{eq:representationInfiniteVectors}, we need to determine the decay rate $s$ for $V^{E_2}.$ It is possible to impose arbitrarily fast polynomial decay in this term. On the other hand, the vector $V^{E_1}$ can be represented by \eqref{eq:representationInfiniteVectors} with $s=0$. Finally, the representation of $V^{L^{-1}}$ decays with $\overline{s} = 2$ determined by the decay of the inverses of the eigenvalues of $L$ (which have to decay to zero as the considered problem is disspative). So, $V_3$ can have higher $s$ then the initial data $X$. Maximal increase of the rate $s$ in $V_3$ is equal to $\overline{s}$. Theoretically in every time step it is possible to increase $s$ by the value of $\overline{s}.$ But it can lead to overestimates on some variables so it is sometimes beneficial to keep the old $s.$ In the code we are using heuristic algorithm which estimates upper bound of resulting series to decide if it worth to increase  the exponent. In  Section \ref{sec:ComputerAssistedProofBrusselator} we observe this effect during integration of Brusselator system. The same algorithm of evolution of sets works in the case of variational equation.} For the variational equation we work with the infinite interval vectors $(U,U^h)$ with polynomial bound \eqref{eq:representationInfiniteVectorsChafeInfanteVariational}. Then constants $s_1$ and $s_2$ may also increase during the integration, as the $V^{L^{-1}}$ decays by $\overline{s}=2$  for the this problem. From this effect, the evolution of sets with the unbounded variational part may, during integration, become a bounded set. In Section \ref{sec:resultChafe-Infante} we present the example that the unbounded set becomes bounded upon evolution.

\section{Results for non-autonomous Chafee--Infante and fractional Burgers equations}\label{sec:Chafee-Infante results}
In this section we present computer assisted proofs of our main results for the Chafee--Infante and Burgers problems, that is the existence of locally attracting periodic orbits. 
We  use the $C^1$ algorithm to obtain an estimate on the solution to the variational equation along the periodic orbit.
\subsection{Abstract results}
In this section, we work with problem \eqref{eq:AbstractProblem} and we assume that conditions \hyperref[as:BL1]{(BL1)}, \hyperref[as:BL2]{(BL2)}, \hyperref[as:Bf]{(Bf)}, \hyperref[as:BH]{(BH)}, \hyperref[as:BD]{(BD)}, \hyperref[as:BV1]{(BV1)}, \hyperref[as:BV2]{(BV2)}, \hyperref[as:BV3]{(BV3)} are satisfied. We also assume that $f$ is periodic in time, that is $f(t,x) = f(t+\overline{t},x)$ for some $\overline{t}>0$ and every $(t,x)\in  \mathbb{R}\times Y.$ We denote by $\varphi$ the local process induced by equation \eqref{eq:AbstractProblem}. Moreover, we denote by $V(t,t_0,x^0)$ the Fr\'{e}chet derivative of the map $\varphi(t,t_0,\cdot):Y\to Y$ at $x^0.$\Doktorat{  The following theorem is analogous to Theorem \ref{th:PoincareFixedPoint}. The difference is that now we use time shift map instead of Poincar\'{e} map.} Observe that in the following theorem we do not need  any estimate on the solution of the variational equation, so in the application of this theorem it is enough to use a $C^0$ integration algorithm.
\begin{lemma}\label{lem:AbstractNonAutonomousOrbit}
Assume that $X^0\subset Y$ is a compact and convex set.
If $\varphi(\overline{t},0,X^0)\subset X^0$ then there exists a global periodic solution $x:\mathbb{R}\to Y$ with period $\overline{t}$. This means that $x(\overline{t}+t)= x(t)$ for every $t\in \mathbb{R}$, and $\varphi(t,t_0,x(t_0)) = x(t)$ for every $t_0\in\mathbb{R}$ and $t\geq t_0.$
\end{lemma}
\begin{proof}
By the continuity of map $\varphi(\overline{t},0,\cdot)$ and the Schauder fixed point theorem, we observe that there exists $\overline{x}\in X^0$ such that $\varphi(\overline{t},0,\overline{x}) = \overline{x}.$
We define $x(t) = \varphi(t ,-n\overline{t},\overline{x})$ for $n\in\mathbb{N}$ and $t\geq - n\overline{t}.$ Then $x(t)$ is well defined and it does not depend on $n.$ 
Indeed for $n_1,n_2\in\mathbb{N},$ such that $n_1\geq n_2$ we have that 
\begin{align*}
    \varphi(t ,-n_1\overline{t},\overline{x}) &= 
    \varphi( t ,-n_2\overline{t},\varphi(-n_2\overline{t} ,-n_1\overline{t},\overline{x})) 
    \\&=
    \varphi( t ,-n_2\overline{t},\varphi^{n_1-n_2}(\overline{t},0,\overline{x} )) = \varphi( t ,-n_2\overline{t},\overline{x}), 
\end{align*}
which shows that $x(t)$ does not depend on $n$. Now we compute 
\begin{equation*}
    \varphi(t,t_0,x(t_0))=  \varphi(t,t_0,
    \varphi(t_0,-n\overline{t},\overline{x}) =\varphi(t,-n\overline{t},\overline{x}) = x(t).
\end{equation*}
We also have
\begin{equation*}
    x(\overline{t}+t)=  \varphi(\overline{t}+t, -(n-1)\overline{t},\overline{x}))  = \varphi(t, -n\overline{t},\overline{x}) = x(t),
\end{equation*}
which ends the proof.
\end{proof}
The next lemma allows us to show that the periodic orbit that we have found is attracting. To validate the assumption of this lemma we need to have estimates on the solution of variational equation, which we obtain from the $C^1$ integration algorithm. 
\begin{lemma}\label{lem:AbstractNonAutonomousAttracting}
    Assume that $x(t):\mathbb{R}\to Y$ is a periodic solution to \eqref{eq:AbstractProblem} with period $\overline{t}.$ Moreover, let
    \begin{equation*}
        \norm{V(\overline{t},0,x(0))}_{\mathcal{L}(Y) } = a < 1.
    \end{equation*}
    For every $\varepsilon>0$ there exist $\delta>0$ and $D>0$ such that for every $t_0\in\mathbb{R}$, $t\geq t_0$ and $h\in B_Y(x(t_0),\delta)$ the function $\varphi(t,t_0,x(t_0)+h)$ is well defined and the following estimate holds
    \begin{equation*}
        \norm{\varphi(t,t_0,x(t_0) + h) - x(t)}_Y\leq De^{\ln(a+\varepsilon)\frac{t-t_0}{\overline{t}}}\norm{h}_Y.
    \end{equation*}
    
    \begin{proof}
    Define $R := \sup_{t\in\mathbb{R}}\norm{x(t)}_Y.$ From  Lemma \ref{lem:VariationalProperties} 
    for $\norm{h}_Y\leq \min\left\{\frac{\varepsilon}{C_3(R,\overline{t})} ,\lambda_3(R,\overline{t}) \right\}$ with $\varepsilon>0$ such that $\varepsilon+a<1,$  we have
    \begin{equation*}
        \norm{\varphi(\bar{t},0,x+h) -x(\overline{t})}_Y
        \leq a\norm{h}_Y + C_3(R,\overline{t})\norm{h}_Y^2 \leq (a+\varepsilon) \norm{h}_Y.
    \end{equation*}
 Using  Lemma \ref{rem:LongTimeExistence} we observe that
        \begin{align*}
            \norm{\varphi(t,t_0,x(t_0)+h) -x(t)}_Y
            &=
            \norm{
            \varphi(t,n_2\overline{t},
            \varphi(n_2\overline{t} ,n_1\overline{t},
            \varphi(n_1\overline{t} ,t_0,x(t_0)+h )) -x(t)}_Y 
            \\
            &\leq
            \norm{\varphi(t,n_2\overline{t},
            \varphi^{n_1-n_2}(\overline{t} ,0,
            \varphi(n_1\overline{t} ,t_0,x(t_0)+h )) -x(t)}_Y
            \\
            &\leq
            C_1(R,\overline{t})^2(a+\varepsilon)^{n_2-n_1} \norm{h}_Y
            \\
            &\leq
            C_1(R,\overline{t})^2(a+\varepsilon)^{-2} e^{\ln{(a+\varepsilon)}\frac{t-t_0}{\overline{t}}} \norm{h}_Y,
        \end{align*}
        for $\norm{h}_Y\leq \min\left\{\frac{\varepsilon} {C_1(R,\overline{t})C_3(R,\overline{t})},\lambda_1(R,\overline{t}),\frac{\lambda_3(R,\overline{t})}{C_1(R,\overline{t})}\right\} $ with $n_2 = \floor{\frac{t}{\overline{t}}}$ and $n_1 = \ceil{\frac{t_0}{\overline{t}}}.$ The proof is complete. 
    
    \end{proof}
\end{lemma}
\subsection{Results on the stable solutions of Chafee--Infante problem.}\label{sec:resultChafe-Infante}
In this section, we prove the main theorem about the existence of a locally attracting periodic orbit for the Chafee--Infante system. The result is presented in the next theorem. 
\begin{theorem}\label{th:Chaffe-InfanteAtractingOrbit1}
The Chafee-Infante problem for $\lambda = 2,\,b(t)=\frac{1}{2}\sin(2\pi t) +1$ has a periodic orbit $u^*(t):\mathbb{R}\to C_0(0,\pi)$ with period $1$ which is symmetric with respect to the line $x=\frac{\pi}{2},$ that is $u^*(t,\frac{\pi}{2}+x) = u^*(t,\frac{\pi}{2}-x)$ for every $x\in[0,\frac{\pi}{2}]$ and $t\in \R$. Moreover, there exist $\delta>0,$   $D>0$ and $a<0$ such that for all 
initial data $t_0\in\mathbb{R}$ and $u^0\in C_0(0,\pi)$ satisfying $\norm{u^*(t_0)-u^0}_{C_0(0,\pi)}\leq\delta$ there exists a global solution $u:[t_0,\infty)\to C_0(0,\pi)$ with $u(t_0)=u^0$ and the following estimate holds 
\begin{equation*}
    \norm{u^*(t)-u(t)}_{C_0(0,\pi)} \leq De^{a(t-t_0)}  \norm{u^*(t_0)-u(t_0)}_{C_0(0,\pi)}\quad\text{for}\;t\in[t_0,\infty).
\end{equation*}
\end{theorem}
We start the proof procedure by finding the approximate fixed point for the map 
\begin{equation*}
    \varphi(1,0,\cdot):C_0(0,\pi)\to C_0(0,\pi),
\end{equation*}
which is used for the construction of set $X^0$ for which we validate the Lemma \ref{lem:AbstractNonAutonomousOrbit}. We have numerically iterated this map and got the point 
\begin{align*}
    u^*(x) = 1.2703\sin(x)+10^{-2}\cdot4.30299\sin(3x)+10^{-3}\cdot1.63332\sin(5x)+10^{-5}\cdot6.25849\sin(7x),
\end{align*}
which is the approximation of the fixed point for this map. We define the set $X^0$ in the following way
\begin{equation*}
    X^0 = X_P^0 + X_Q^0,
\end{equation*} 
with
\begin{equation*}
    X_P^0 = u^*(x)+10^{-4}\sum_{k=1}^4[-1,1]\sin((2k-1)x),
\end{equation*}
\begin{equation*}
    X_Q^0 = \left\{u\in C_0(0,\pi): u_k\in\frac{[-1,1]}{k^4}\ \  \text{for $i$ odd and $k>7$}\;\ \textrm{and}\  u_k = 0 \;\text{otherwise}    \right\}.
\end{equation*}
From the rigorous integration of the Chafee--Infante system we have the following:
\begin{equation*}
    \varphi(1,0,X^0)\subset X^1,
\end{equation*}
where $X^1 = X ^1_P+X^1_Q$ with
\begin{align*}
    X_P^1 &= u^*(x)+ 10^{-5}[-1.64257, 1.85394]\sin(x) + 10^{-6}[-1.90213, 1.93443]\sin(3x) \\&+ 10^{-7}\sin(5x)10^{-7}[-1.893, 2.30257] +10^{-7}\sin(7x)[-5.35159, 9.18297],
\end{align*}
and
\begin{equation*}
    X_Q^1 = \left\{u\in C(0,\pi): u_k\in\frac{[-0.1644, 0.326846]}{k^4}\ \ \text{for $k$ odd and $k>7$ and}
     \; u_k = 0 \;\text{otherwise}    \right\}.
\end{equation*}
We observe that $X^1\subset X^0,$ and $X^0$ is convex and compact in $C_0(0,\pi).$ From Lemma \ref{lem:AbstractNonAutonomousOrbit}, this gives the existence of a periodic orbit for the Chafee--Infante problem. Now we have the following estimations from the rigorous integration of the variational equation. 
\begin{align*}
    V(1,0,X^0)\sin(x)&\subset
    [0.143998, 0.144082]\sin(x) + 10^{-7}[-5.79177, 5.96521]\sin(2x) 
    \\&+ 10^{-2}[1.5544, 1.5554]\sin(3x) + 10^{-8}[-6.33397, 7.02181]\sin(4x) 
    \\&+ 10^{-4}[9.86374, 9.87354]\sin(5x) + 10^{-7}[-1.23369, 1.64851]\sin(6x)  
    \\&+
    10^{-5}[5.22783, 5.42035]\sin(7x) + 10^{-7}[-7.10216, 1.01582]\sin(8x)
    \\&+\sum_{k=9}^\infty \frac{[-0.0559152, 0.11911]}{k^{3.5}}\sin(kx),
\\
    V(1,0,X^0)\sin(2x)&\subset
    10^{-6}[-1.87253, 1.89317]\sin(x) + 10^{-2}[1.67772, 1.67855]\sin(2x) \\&+ 10^{-7}[-2.28335, 2.32344]\sin(3x) + 10^{-3}[1.25572, 1.25643]\sin(4x) \\&+ 10^{-8}[-3.13828, 3.81326]\sin(5x) + 10^{-5}[7.22471, 7.23647]\sin(6x) \\&+ 10^{-7}[-1.15233, 1.58731]\sin(7x) + 10^{-6}[3.63068, 3.78777]\sin(8x),
    \\&+\sum_{k=9}^\infty \frac{ [-0.00199783, 0.00355262]}{k^{3}}\sin(kx),
    \\
     V(1,0,X^0)\sin(3x)&\subset 10^{-2}[2.41132, 2.41403]  \sin(x) + 10^{-7}[-5.41089, 5.72121]  \sin(2x) \\&+ 10^{-3}[2.71165, 2.7148]  \sin(3x) + 10^{-8}[-4.41425, 4.74869]  \sin(4x) \\&+ 10^{-4}[1.7158, 1.7184]  \sin(5x) +10^{-8}[-2.33457, 3.07558]  \sin(6x) \\&+ 10^{-6}[9.09545, 9.43261]  \sin(7x) + 10^{-7}[-1.20473, 1.74117]  \sin(8x)
     \\&+\sum_{k=9}^{\infty}\frac{[-0.00960379, 0.0206284]}{k^{3.5}}\sin(kx).
 \end{align*}
 We also integrate the unbounded set
 \begin{equation*}
     X^{0,h} = \{h\in C_0(0,\pi): h_k\in [-1,1] \text{ for } k\geq 3  \text{ and } h_k = 0 \text{ for } k = 1,2,3\},
 \end{equation*}
and obtain
\begin{align*}
    V(1,0,X^0)X^{0,h}&\subset
    10^{-3}[-2.18135, 2.18131]\sin(x) + 10^{-3}[-1.9155, 1.91549]\sin(2x) 
    \\&+ 10^{-4}[-2.45136, 2.45132]\sin(3x) + 10^{-4}[-1.43496, 1.43496]\sin(4x) 
    \\&+ 10^{-5}[-1.55205, 1.55202]\sin(5x) + 10^{-6}[-8.2667, 8.26666]\sin(6x) + 
    \\&+
    10^{-7}[-8.66491, 8.66291]\sin(7x) + -10^{-7}[-4.46976, 4.46866]\sin(8x)
    \\&+\sum_{k=9}^\infty \frac{[-0.000899498, 0.000892498]}{k^{3}}\sin(kx).
\end{align*}
Using Lemmas \ref{lem:C0operatorNormEstimation} and \ref{lem:C0NormEstimation} we compute
\begin{equation*}
    \sup_{x\in X^0}\norm{V(1,0,x)}_{\mathcal{L}(C_0(0,\pi))}\leq 0.421383.
\end{equation*}
Hence, we can apply Lemma \ref{lem:AbstractNonAutonomousAttracting} to deduce that the periodic orbit, for which we previously proved the existence, is locally attracting. Moreover, the estimate  on the norm of $V$ gives the estimate on the rate of exponential convergence in the same Lemma.
We also integrate the set
 \begin{equation*}
     X^{1,h} = \left\{h\in C_0(0,\pi): h_k\in [-k^2,k^2] \text{ for } k\in\mathbb{N}\right\}.
 \end{equation*}
 and get the estimate

\begin{align*}
    V(1,0,X^0)X^{1,h}&\subset
    [-0.524123, 0.524116]\sin(x) + [-0.119484, 0.119484]\sin(2x) 
    \\&+ 10^{-2}[-5.83969, 5.83961]\sin(3x) + 10^{-3}[-8.95083, 8.95078]\sin(4x) 
    \\&+ 10^{-3}[-3.6993, 3.69925]\sin(5x) + 10^{-4}[-5.16246, 5.16239]\sin(6x) + 
    \\&+
    10^{-4}[-2.05122, 2.05075]\sin(7x) + -10^{-5}[-3.16219, 3.15943]\sin(8x)
    \\&+\sum_{k=9}^\infty \frac{[-0.0702182, 0.0696598]}{k^{2.5}}\sin(kx).
\end{align*}
 This allows us to get the decay estimate on the evolution of all Fourier modes. Namely, we have

    \begin{align}
        V(1,0,X^0)\sin(jx)&\subset\frac{1}{j^2}\biggl(
        [-0.524123, 0.524116]\sin(x) + [-0.119484, 0.119484]\sin(2x) \label{eq:VarationalDecayingEstimate}
        \notag\\&+ 10^{-2}[-5.83969, 5.83961]\sin(3x) + 10^{-3}[-8.95083, 8.95078]\sin(4x) 
        \notag\\&+ 10^{-3}[-3.6993, 3.69925]\sin(5x) + 10^{-4}[-5.16246, 5.16239]\sin(6x) + 
        \\&+
        10^{-4}[-2.05122, 2.05075]\sin(7x) + 10^{-5}[-3.16219, 3.15943]\sin(8x)
        \notag\\&+\sum_{k=9}^\infty \frac{[-0.0702182, 0.0696598]}{k^{2.5}}\sin(kx)
        \biggr),\notag
    \end{align}

for every $j\in\mathbb{N}$.
\begin{figure}[H]
    \centering
    \includegraphics[scale=0.5]{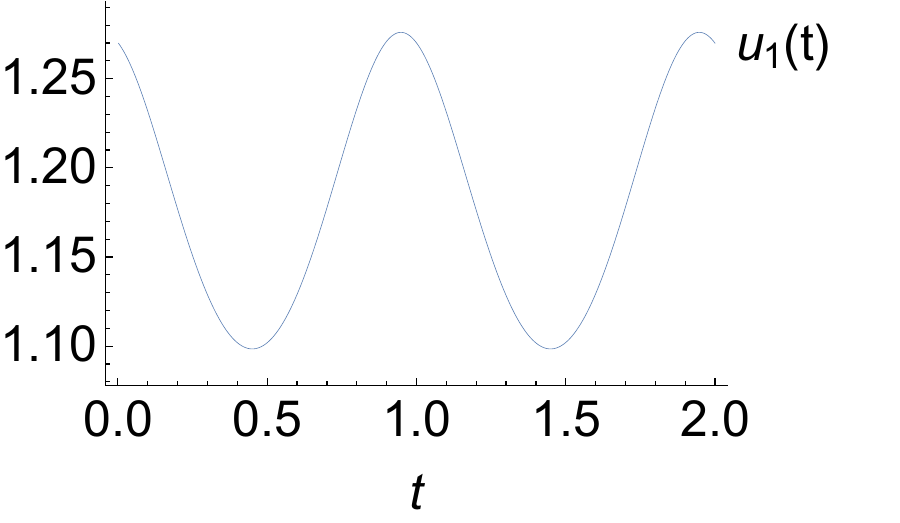}
    \quad
    \includegraphics[scale=0.5]{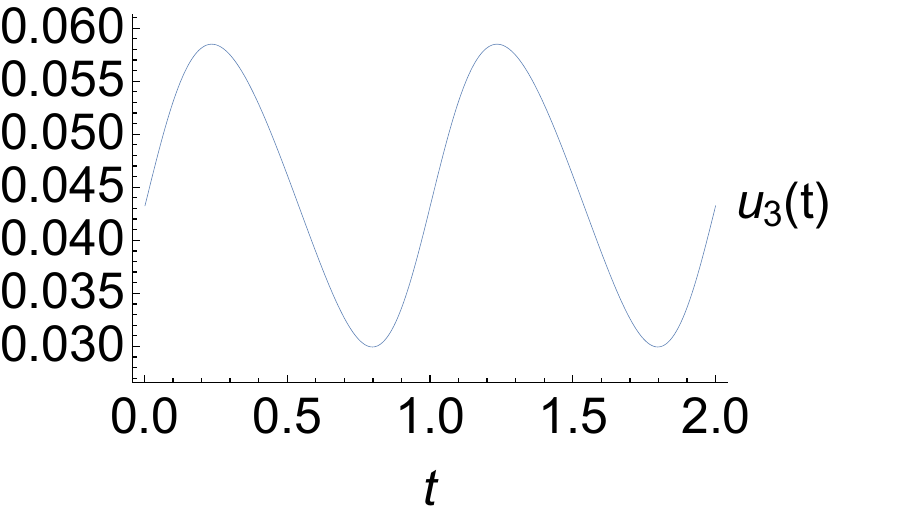}
    \noindent
    \includegraphics[scale=0.5]{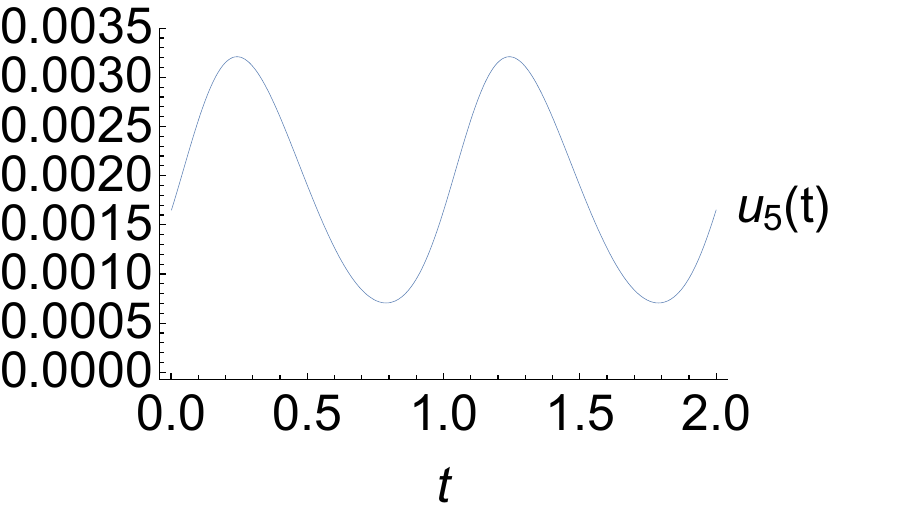}
    \includegraphics[scale=0.5]{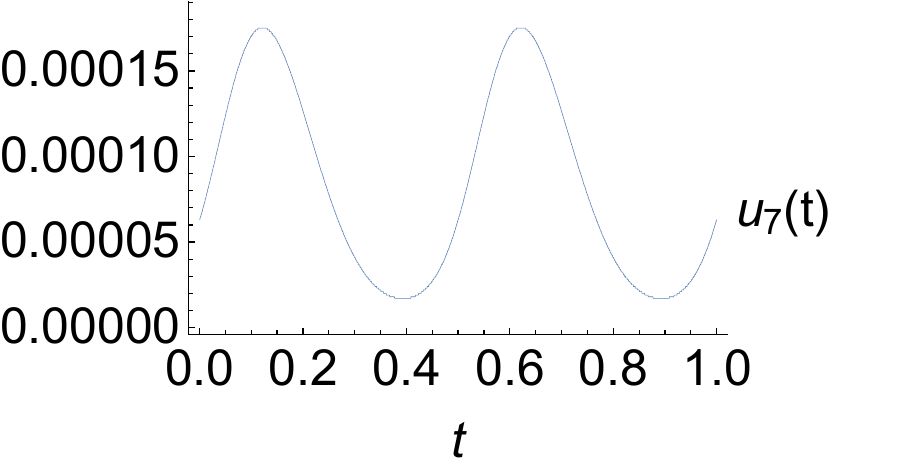}
    \caption{The approximation of evolution of Fourier modes for the periodic solution $u^*(t)$ from Theorem \ref{th:Chaffe-InfanteAtractingOrbit1}  for $i=1,3,5,7.$}
\end{figure}

Using the same scheme of the proof, as in the case of Theorem \ref{th:Chaffe-InfanteAtractingOrbit1} we prove the next result
\begin{theorem}\label{th:Chaffe-InfanteAtractingOrbit2}

The Chafee--Infante equation for $\lambda = 2,\,b(t)=\frac{3}{2}\sin(2\pi t) +1$ has a periodic orbit $u^*(t,x):\mathbb{R}\to C_0(0,\pi)$ with period $1$. This orbit is symmetric with respect to the line $x=\frac{\pi}{2}$, that is $u^*(t,\frac{\pi}{2}+x) = u^*(t,\frac{\pi}{2}-x)$ for every $x\in[0,\frac{\pi}{2}],\;t\in\R.$ Moreover, there exist $\delta>0,$   $D>0$ and $a<0$ such that for every 
initial data $t_0\in\mathbb{R}$ and $u^0\in C_0(0,\pi)$ satisfying $\norm{u^*(t_0)-u^0}_{C_0(0,\pi)}\leq\delta$ there exists a global solution $u:[t_0,\infty)\to C_0$ with $u(t_0)=u^0$ and the following holds 
\begin{equation*}
    \norm{u^*(t)-u(t)}_{C_0(0,\pi)} \leq De^{a(t-t_0)} \norm{u^*(t_0)-u(t_0)}_{C_0(0,\pi)}\quad\text{for}\;t\in[t_0,\infty).
\end{equation*}
\end{theorem}
The function $b(t)$ from the above theorem is interesting, as the nonlinear term at some times $t$ is multiplied by the positive number, which implies the possible blow-up of the solution.
Consider the autonomous ODE 
\begin{equation*}
    \frac{d}{dt}u(t) = u^3.
\end{equation*}
For $u(0)\not=0$ we have $\lim_{t\to t_{\max}(u(0))} |u(t)|= \infty$ with $t_{\max}(u(0))<\infty.$
The known results about the existence of the solutions for the non-autonomous Chafee--Infante  equation rely on the assumption that the term multiplying the $u^3$ is non-positive for all times (see \cite{ChafeeInfanteLangaCarvalho}). 
In Theorem \ref{th:Chaffe-InfanteAtractingOrbit2} we obtain the existence of a global nonzero solution for the situation when for some times $t$ the term $u^3$ is multiplied by a positive constant, which is also a new result in the theory of the existence of solutions for the non-autonomous Chafee--Infante equation.
\begin{figure}[H]
    \centering
    \includegraphics[scale=0.5]{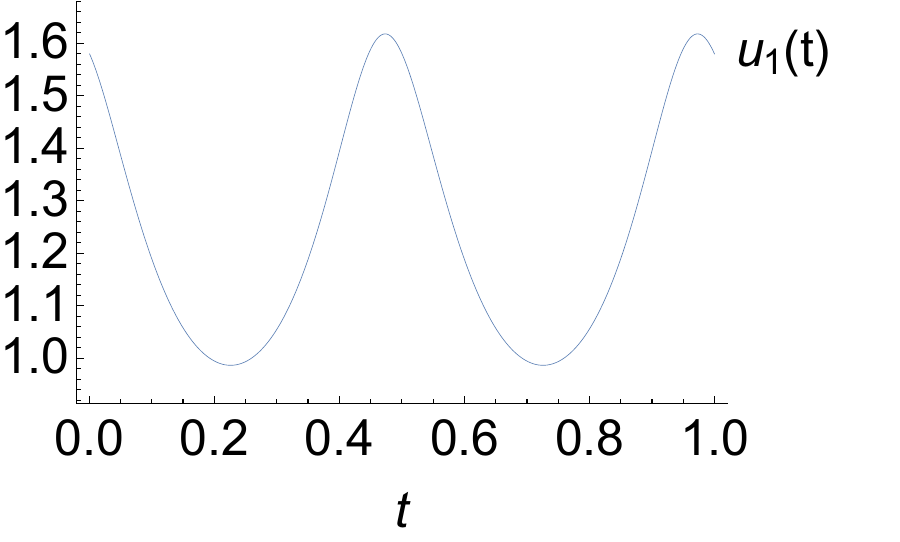}
    \quad
    \includegraphics[scale=0.5]{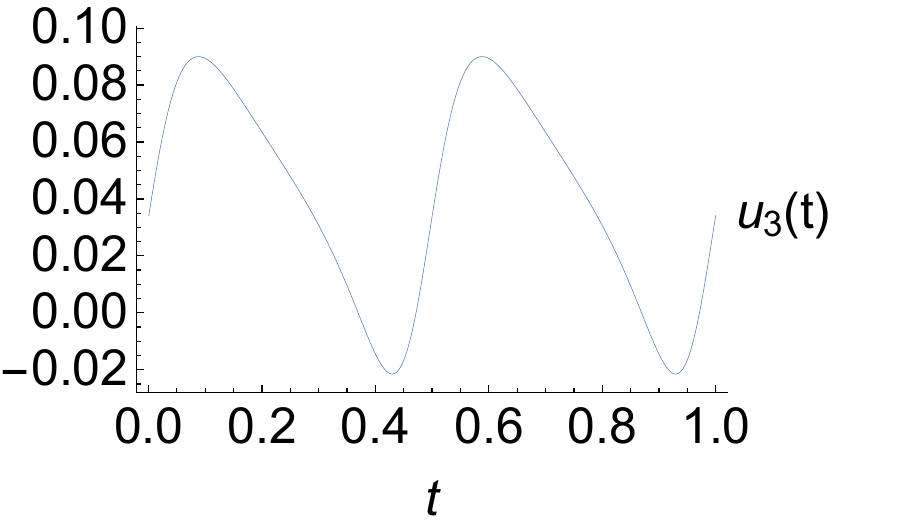}
    \noindent
    \includegraphics[scale=0.5]{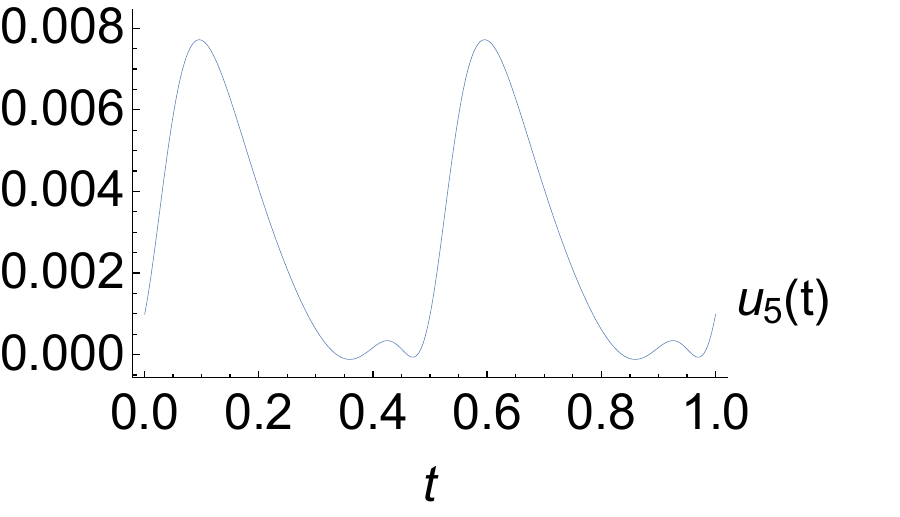}
    \includegraphics[scale=0.5]{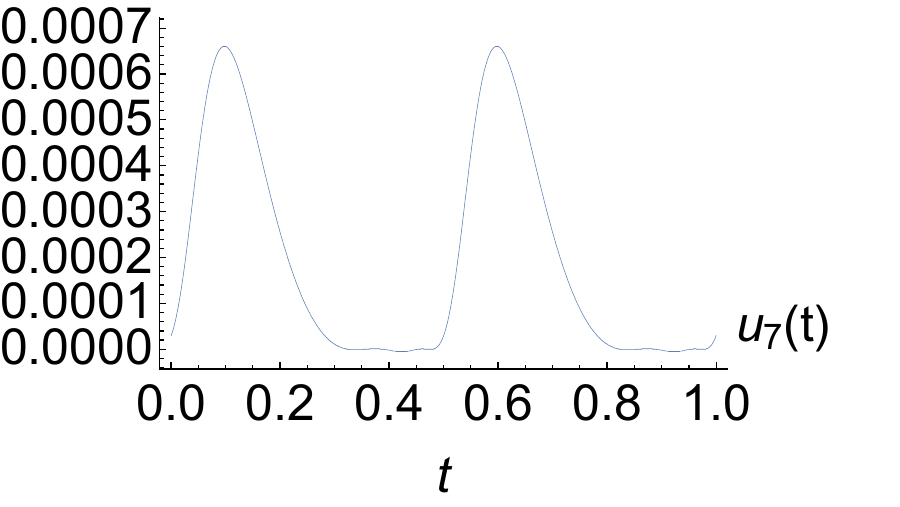}
        \caption{The approximation of evolution of Fourier modes for the periodic solution $u^*(t)$ from Theorem \ref{th:Chaffe-InfanteAtractingOrbit2}  for $i=1,3,5,7.$ }
\end{figure}
For $\lambda = 2$ the  linearization at zero, which is a fixed point, has one usable direction. For $\lambda = 5$ the linearization of zero has two unstable directions. We also prove the existence in this case.
\begin{theorem}\label{th:Chaffe-InfanteAtractingOrbit3 }
The Chafee--Infante equation for $\lambda = 5,\,b(t)=\frac{1}{2}\sin(2\pi t) +1$ has a periodic orbit $u^*(t,x):\mathbb{R}\to C_0(0,\pi)$ with period $1$. This orbit is symmetric with respect to the line $x=\frac{\pi}{2},$ that is $u^*(t,\frac{\pi}{2}+x) = u^*(t,\frac{\pi}{2}-x)$ for every $x\in[0,\frac{\pi}{2}],\;t\in \R.$

Moreover, there exist $\delta>0,$   $D>0$ and $a<0$ such that for every 
initial data $t_0\in\mathbb{R}$ and $u^0\in C_0(0,\pi)$ satisfying $\norm{u^*(t_0)-u^0}_{C_0(0,\pi)}\leq\delta$ there exists a global solution $u:[t_0,\infty)\to C_0(0,\pi)$ with $u(t_0)=u^0$ and the following holds 
\begin{equation*}
    \norm{u^*(t)-u(t)}_{C_0(0,\pi)} \leq De^{a(t-t_0)} \norm{u^*(t_0)-u(t_0)}_{C_0(0,\pi)},\quad\text{for}\;t\in[t_0,\infty).
\end{equation*}
\end{theorem}

For these parameters the linearization at zero, which is fixed point, has two unstable directions
\begin{figure}[H]
    \centering
    \includegraphics[scale=0.5]{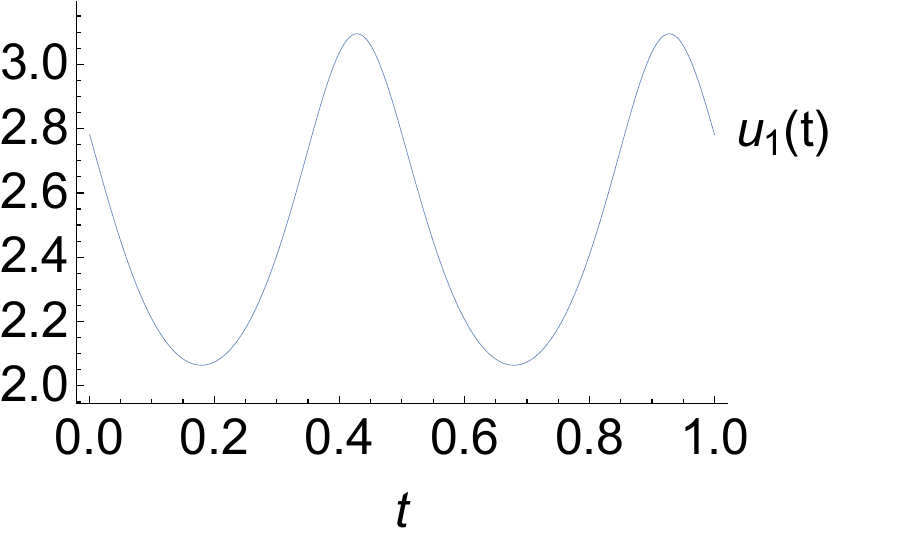}
    \quad
    \includegraphics[scale=0.5]{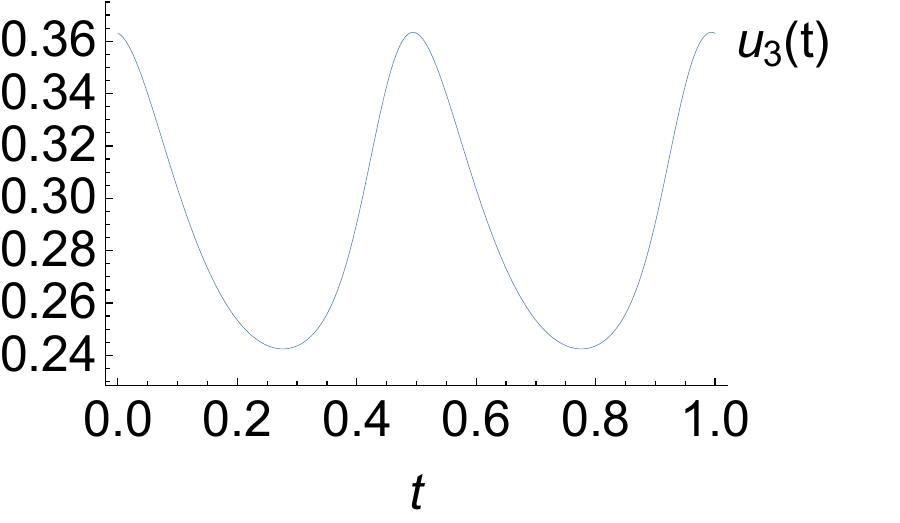}
    \noindent
    \includegraphics[scale=0.5]{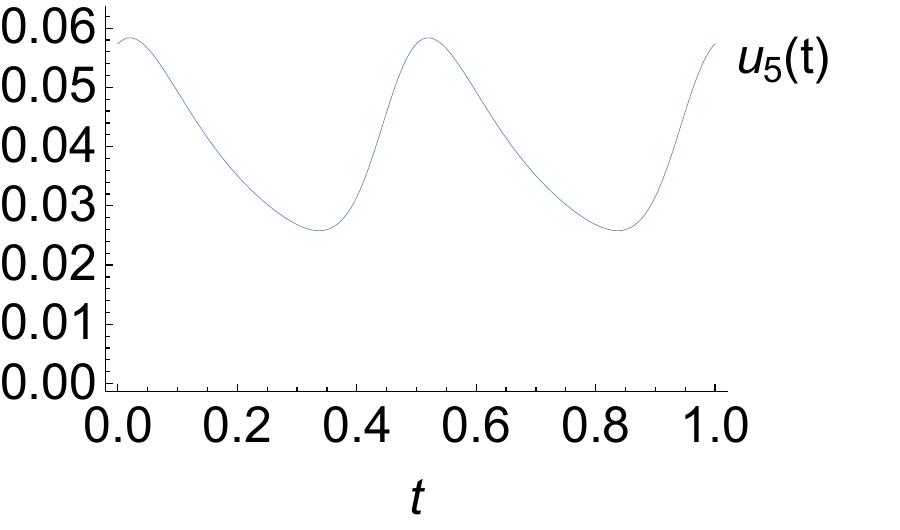}
    \includegraphics[scale=0.5]{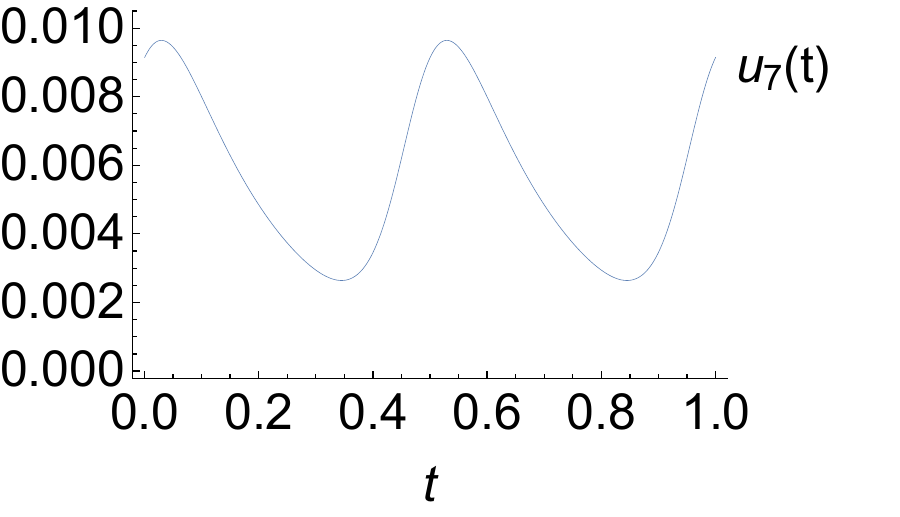}
    \caption{The approximation of evolution of Fourier modes for the periodic solution $u^*(t)$ from Theorem \ref{th:Chaffe-InfanteAtractingOrbit3 }  for $i=1,3,5,7.$}
\end{figure}
\subsection {Rigorous \texorpdfstring{$C^1$}{C1} computations around the unstable orbits for the Chafee--Infante system}\label{sec:chafee-Infate Addition}
In this section we demonstrate that the $C^1$ computations can be 
preformed for the orbit of the system, which are not necessarily stable. The results of these computations can further used for schemes of computer assisted proofs which require the estimate of the derivative of solution such as, for example the Krawczyk and the Newton--Kantorovich method of fining zeros, and validating the cone conditions. These techniques are useful to determine the full dynamics of the system including the heteroclinic connection between orbits. 

To demonstrate that the $C^1$ algorithm can be used in unstable context, we realize a $C^1$ computation for the Chafee--Infante system around zero, which is an unstable equilibrium. The linearized equation for the solution $u(t)\equiv 0$ has the form
 \begin{equation}\label{eq:Chafee–InfanteVariational-2}
    \begin{cases}
        h_t = h_{xx}+\lambda h\quad \text{ for }\; (x,t)\in (0,\pi)\times(t_0,\infty),\\
        h(t,x) = 0
         \ \; \text{for}\; (x,t)\in \{0,\pi\} \times (t_0,\infty),
    \\ h(t_0,x) = h^{0}(x)\ \  \textrm{ for}\ x\in(0,\pi).
\end{cases}
\end{equation}
This system can be solved explicitly and the solution is given by the following formula for the  linear operator $V$, which becomes diagonal
\begin{equation*}
    V(1,0,0)\sin(kx) = e^{\lambda-k^2 }\sin(kx).
\end{equation*}
To check if our $C^1$  computations are correct, we conduct them for set $X_0$ centered around zero in the following way:
\begin{equation*}
    X^0= \{u\in C_0(0,\pi): u_k\in[-0.001,0.001]\text{ for $1\leq k\leq 8$}\ \ \text{and}\ \ \ u_k\in\frac{[-1,1]}{k^5}\ \text{for}\ k>8  \}.
\end{equation*}
For parameters $\lambda = 5$ and $b(t)= \frac{1}{2}\sin(2\pi t)+1,$ we obtained the following results:
\begin{align*}
    V(1,0,X^0)\sin(x) &\subset \sin(x)[54.5786, 54.6091]+10^{-3}\sin(2x)[-1.49879, 1.49879]
    \\&+10^{-3}\sin(3x)[-4.42365, 4.42365]+10^{-4}\sin(4x)[-3.79411, 3.79411]
    \\&+10^{-5}\sin(5x)[-1.05063, 1.05063]+10^{-7}\sin(6x)[-1.38014, 1.38014]
    \\&+10^{-9}\sin(7x)[-3.13962, 3.13962]+10^{-11}\sin(8x)[-9.8782, 9.8782]
    \\&+\sum_{k=9}^\infty\sin(kx)\frac{10^{-4}[-8.46563, 8.46563]}{k^9},
\end{align*}
\begin{align*}
V(1,0,X^0)\sin(2x) &\subset 10^{-4}\sin(x)\,[-4.9245,\,4.9245]+\sin(2x)\,[2.71768,\,2.71846] \\
&+10^{-5}\sin(3x)\,[-3.05667,\,3.05667]
+10^{-4}\sin(4x)\,[-1.89699,\,1.89699] \\
&+10^{-5}\sin(5x)\,[-1.59162,\,1.59162]
+10^{-7}\sin(6x)\,[-4.33905,\,4.33905] \\
&+10^{-9}\sin(7x)\,[-5.58493,\,5.58493]
+10^{-10}\sin(8x)\,[-1.23607,\,1.23607]
\\&+\sum_{k=9}^\infty\sin(kx)\frac{10^{-3}[-1.48816, 1.48816]}{k^9},
\end{align*}
\begin{align*}
V(1,0,X^0)\sin(3x) &\subset10^{-4}\sin(x)\,[-1.21297,\,1.21297]
+10^{-6}\sin(2x)\,[-8.14237,\,8.14237] \\
&+10^{-2}\sin(3x)\,[1.83116,\,1.83168]
+10^{-7}\sin(4x)\,[-1.79892,\,1.79892] \\
&+10^{-6}\sin(5x)\,[-1.13882,\,1.13882]
+10^{-8}\sin(6x)\,[-9.33953,\,9.33953] \\
&+10^{-9}\sin(7x)\,[-2.50739,\,2.50739]
+10^{-11}\sin(8x)\,[-3.1764,\,3.1764]
\\&+\sum_{k=9}^\infty\sin(kx)\frac{10^{-4}[-2.66474, 2.66474]}{k^9},
\end{align*}
\begin{align*}
V(1,0,X^0)\sin(kx) &\subset \frac{1}{k^3}\biggl( 10^{-2}\sin(x)\,[-9.41054,\,9.41054]
+10^{-3}\sin(2x)\,[-5.36919,\,5.36919]  \\
&+10^{-5}\sin(3x)\,[-5.07224,\,5.07224] 
+10^{-3}\sin(4x)\,[-1.07028,\,1.07028] \\
&+10^{-7}\sin(5x)\,[-3.19303,\,3.19303]
+10^{-8}\sin(6x)\,[-6.15958,\,6.15958] \\
&+10^{-9}\sin(7x)\,[-4.87869,\,4.87869]
+10^{-10}\sin(8x)\,[-1.29837,\,1.29837] 
\\&+\sum_{k=9}^\infty\sin(jx)\frac{10^{-4}[-6.37198, 6.37198]}{k^9}\biggr) .
\end{align*}
We can validate the results of the computations $C^1$ by verifying if $V(1,0,0)\sin(kx) = \sin(kx)e^{\lambda-k^2}\in V(1,0,X_0)\sin(kx)$ for every $k\in\mathbb{N.}$ 
\begin{comment}
We have that
\begin{align*}
    e^4\sin(x)&\in[54.598,54.99]\sin(x) \in  V(1,0,X^0)\sin(x)\\
    e\sin(2x)&\in[2.71828,2,71829]\sin(2x) \in  V(1,0,X^0)\sin(2x) \\
    e^{-4}\sin(3x)&\in[0.0183156,0.0183157]\sin(3x) \in  V(1,0,X^0)\sin(3x) 
    e^{-11}\sin(3x)&\in[0.0183156,0.0183157]\sin(3x) \in  V(1,0,X^0)\sin(3x) 
\end{align*}
\end{comment}

For the same set of parameters (that is, $\lambda=5$ and $b(t)=\frac{1}{2}\sin(2\pi t)+1$), we also perform $C^1$ computations for the estimation of the periodic orbit that was numerically found by the Newton method, with the following value at zero:
\begin{equation*}
    u^*(x) = 1.55005\sin(2x) + 0.0206521\sin(6x) + 0.000307565\sin(10x)
\end{equation*}
We also computed numerically the approximate eigenvalues and eigenvectors of $V(1,0,u^*).$ The  three largest eigenvalues and corresponding eigenvectors are the following:
\begin{align*}
    &\lambda_1= 9.39278,\\
    & h^1(x)=\sin(x) - 0.0622682\sin(3x) + 0.0477333\sin(5x) - 0.0019565\sin(7x) + 0.0012689\sin(9x),\\
    &\lambda_2= 0.141303,\\
& h^2(x)=\sin(2x) + 0.0407946\sin(6x) + 0.00101412\sin(10x),\\
   & \lambda_3 = 0.00205114,\\
& h^3(x) = 0.252499\sin(x) + \sin(3x) + 0.0128624\sin(5x) + 0.0346569\sin(7x) + 0.00034405\sin(9x).
\end{align*}
For each numerical eigenvector, we rigorously compute the value of the derivative $V(1,0,u^*)$. 
We obtain results 
\[
\begin{aligned}
V(1,0,u^*)h^1 &\in   \left[ 9.38888,\ 9.39677 \right]\sin(x)  +  \left[ -0.58526,\ -0.584491 \right] \sin(3x)\\
&+\left[ 0.448155,\ 0.44855 \right]\sin(5x)  +  10^{-2}\left[ -1.84853,\ -1.82517 \right]\sin(7x)\\
& + 10^{-2} \left[ 1.18584,\ 1.19841 \right]\sin(9x)\\
& +10^{-3}  \left[ -1.05746,\ 1.08118 \right]\sin(2x) + 10^{-4} \left[ -2.63108,\ 2.51109 \right]\sin(4x) \\ 
&+ 10^{-5} \left[ -8.03344 \,\ 7.69113 \right]\sin(6x) \\
& +  10^{-5}\left[ -9.32192 ,11.5735 \right]\sin(8x) + 10^{-4}\left[ -1.82383,\ 1.13402 \right]\sin(10x) \\
&+ \sum_{k=11}^\infty\frac{[-83.2595, 84.9202]}{k^4}\sin(kx),
\end{aligned}
\]

\[
\begin{aligned}
V(1,0,u^*)h^2 &\in  [0.141277,\ 0.141332] \sin(2x) + 10^{-3}  [5.76309,\ 5.766] \sin(6x)  \\  &+10^{-4}[1.41087,\ 1.47455] \sin(10x) \\
&+ 10^{-3}  [-1.48017,\ 1.38319] \sin(x) + 10^{-4}  [-1.06345,\ 1.13117] \sin(3x) \\
&+ 10^{-7}  [-1.25325,\ 1.93574] \sin(4x) + 10^{-5}  [-7.39712,\ 6.91201] \sin(5x)\\ 
&+ 10^{-6}  [-4.51753,\ 6.30446] \sin(7x) + 10^{-6}  [-0.538893,\ 1.41615] \sin(8x) \\
&+ 10^{-6}  [-3.89992,\ 3.72937] \sin(9x)\\
&+ \sum_{k=11}^{\infty} \frac{[-5.87352,\ 16.3835]}{k^5} \sin(kx),
\end{aligned}
\]

\[
\begin{aligned}
V(1,0,u^*)h^3 &\in 10^{-4}  [4.15171,\ 4.21634] \sin(x) +  10^{-3}  [2.0574,\ 2.05793] \sin(3x)\\
&+ 10^{-5}  [2.14784,\ 2.17875] \sin(5x) + 10^{-5}  [7.13218,\ 7.13414] \sin(7x)\\
&+10^{-7}  [5.75659,\ 5.83978] \sin(9x)\\
&+10^{-8}  [-2.27147,\ 3.12153] \sin(2x)  + 10^{-11}  [-2.67125,\ 4.28956] \sin(4x)\\
&+ 10^{-9}  [-0.946603,\ 1.2896] \sin(6x)  + 10^{-11}  [-4.54136,\ 7.88288] \sin(8x)\\
&+  10^{-11}  [-4.98652,\ 9.63311] \sin(10x) \\
&+ \sum_{k=11}^{\infty} \frac{[-7.26828,\ 33.0753]}{k^{7}} \sin(kx).
\end{aligned}
\]
By integrating the unbounded set 
\begin{equation*}
    X^{0,h} =\{h\in C_0(-\pi,\pi): x_k\in[-k^3,k^3], \text{ for $k\geq 11$ and $h_k =0$ for $1\leq k<10$}\},
\end{equation*}
we obtained the following estimate on high modes:
\[
\begin{aligned}
V(1,0,u^*)\sin(jx) &\in \frac{1}{j^3}( \left[ -6.29691,\ 6.29691 \right]\sin(x) + 10^{-2}  \left[ -2.59359,\ 2.59359 \right] \sin(2x) \\
&+ 10^{-1}  \left[ -3.97055,\ 3.97055 \right] \sin(3x) + 10^{-6}  \left[ -4.48913,\ 4.48913 \right] \sin(4x) \\
&+ 10^{-1}  \left[ -3.00694,\ 3.00694 \right] \sin(5x) + 10^{-3}  \left[ -1.05829,\ 1.05829 \right] \sin(6x) \\
&+ 10^{-2}  \left[ -1.2498,\ 1.2498 \right] \sin(7x) + 10^{-7}  \left[ -1.67188,\ 1.67188 \right] \sin(8x) \\
&+ 10^{-3}  \left[ -7.99605,\ 7.99605 \right] \sin(9x) + 10^{-5}  \left[ -2.63878,\ 2.63878 \right] \sin(10x) \\
&+ \sum_{k=11}^{\infty} \frac{[-67.993,\ 67.993]}{k^5} \sin(kx) ), \text{ for $j\geq11$}.
\end{aligned}
\]
\subsection{Results for the fractional Burgers equation}
For the Burgers equations with the fractional Laplacian we use the same strategy as for the Chafee--Infante equation, that is, for given parameters we first find an approximation of a fixed point for a time shift map $\varphi(1,0,\cdot):H^2_{\text{odd}}(-\pi,\pi)\to H^2_{\text{odd}}(-\pi,\pi),$ where $\varphi$ is the process generated by equation \eqref{eq:BurgersFractionalLaplacianEquation}. 
Then, we construct the set $X^0$ which is convex and compact in space $H^2_{\text{odd}}(-\pi,\pi)$ and contains the found approximation. We  validate rigorously that $\varphi(1,0,X^0)\subset X^0,$ which with use of Lemma \ref{lem:AbstractNonAutonomousOrbit} proves the existence of a periodic orbit.   
Then, we use the result on $C^1$ computation  to validate that these periodic orbits are locally attracting. As the initial data for the variational equation we take the set
\begin{equation*}
    X^{0,h}=\left\{h\in H^2_{\text{odd}}(-\pi,\pi): h_k\in\frac{[-1,1]}{k^2} , k\in\mathbb{N}\right\},
\end{equation*}
so we merge all initial data in one computation. 
As the set $X^{0,h}$ contains the ball $B_{H^2_\text{odd}(-\pi,\pi)}\left(0,\frac{1}{\pi}\right),$ the estimate on $\norm{V(1,0,X^0) X^{0,h}}_{H^2_{\text{odd}}(-\pi,\pi)}$ allows us to get an upper bound on the value of operator norm $\norm{V(0,1,X^0)}_{\mathcal{L}(H^2_\text{odd}(-\pi,\pi))}.$ We use Lemmas \ref{le:BurgersNormEstimate1}, \ref{le:BurgersNormEstimate2} which allows us to check if the following inequality holds
\begin{equation*}
    \norm{V(0,1,X^0)}_{\mathcal{L}(H^2_\text{odd}(-\pi,\pi))} <1.
\end{equation*}
This, from Lemma \ref{lem:AbstractNonAutonomousAttracting}, implies that the found periodic orbit is locally attracting. We perform this procedure for parameters 
$\alpha = \frac{1}{2} + \frac{k}{64}$ with $k = 64, 63, \ldots, 1,\;$  $\nu = \frac{1}{2}$, $b(t)=\sin(2\pi t)\sin(x)$ starting with $k=64$. We  decrease $k$ to see how close we can get to the critical value $\alpha = \frac{1}{2}$. 
If, for some parameter, the estimates blow up, we add additional $20$ modes that are explicitly represented. When the number of modes exceeds $200$, we end the procedure. Half of these modes are estimated using differential inclusion \eqref{inclusion}. 
For the variational equation, the number of explicitly represented modes is doubled. The procedure stopped at $\alpha = \frac{39}{64}.$ 
For this value and for the larger ones, we have proved the existence of periodic orbits. However, for this parameter, the estimated norm was greater than $1$, so local attraction is not proved. 
The attractivity was proved for $\alpha = \frac{44}{64}$ and higher values. It should be possible to obtain a proof of local attraction even for lower $\alpha$ by considering different norms related to eigenvalues of $V(1,0,u^*),$ where $u^*$ is numerically found approximation of fixed point for the time shift map. 
In Table \ref{tab:results}, we summarize the results of the computations for each $\alpha$.  Note that for the Burgers equation, we have not obtained decay estimates in every mode, in contrast to the estimates \eqref{eq:VarationalDecayingEstimate} established for the Chafee--Infante equation. For this we would need additional estimates, analogous to the ones in Lemmas \ref{lem:costimesSinSlowConvering}, \ref{lem:costimesSinDiverging}, for convolutions of two sine series, one of which having  increasing or slowly decreasing Fourier coefficients.

\begin{table}[htbp]
\centering

\begin{tabular}{c c r r c c c}

\toprule
$\alpha$ & Modes num. & $\mathcal{C}^0$  time (s) & $\mathcal{C}^1$ time (s) & Norm estimate & Periodic orbit & Is attracting  \\
\midrule
$1$ & 20 & 0.931884 & 3.51811 & 0.527725 & Validated & Validated \\
$\frac{63}{64}$ & 20 & 1.070080 & 3.71883 & 0.537911 & Validated & Validated \\
$\frac{62}{64}$ & 20 & 0.937165 & 3.41315 & 0.548725 & Validated & Validated \\
$\frac{61}{64}$ & 20 & 0.926880 & 3.45846 & 0.560281 & Validated & Validated \\
$\frac{60}{64}$ & 20 & 1.031430 & 3.66910 & 0.572615 & Validated & Validated \\
$\frac{59}{64}$ & 20 & 1.039940 & 3.65708 & 0.585792 & Validated & Validated \\
$\frac{58}{64}$ & 20 & 0.960964 & 6.11920 & 0.599885 & Validated & Validated \\
$\frac{57}{64}$ & 20 & 0.947647 & 3.48008 & 0.614978 & Validated & Validated \\
$\frac{56}{64}$ & 20 & 0.968047 & 3.55536 & 0.631160 & Validated & Validated \\
$\frac{55}{64}$ & 20 & 0.903815 & 3.36945 & 0.648535 & Validated & Validated \\
$\frac{54}{64}$ & 20 & 0.915306 & 3.28212 & 0.667228 & Validated & Validated \\
$\frac{53}{64}$ & 40 & 4.472560 & 22.6533 & 0.687355 & Validated & Validated \\
$\frac{52}{64}$ & 40 & 4.317000 & 22.2874 & 0.709070 & Validated & Validated \\
$\frac{51}{64}$ & 40 & 4.733100 & 20.1593 & 0.732538 & Validated & Validated \\
$\frac{50}{64}$ & 40 & 4.716560 & 24.6262 & 0.757954 & Validated & Validated \\
$\frac{49}{64}$ & 40 & 4.618390 & 22.6761 & 0.785534 & Validated & Validated \\
$\frac{48}{64}$ & 40 & 4.804520 & 23.7703 & 0.815517 & Validated & Validated \\
$\frac{47}{64}$ & 60 & 12.63170 & 70.4163 & 0.848196 & Validated & Validated \\
$\frac{46}{64}$ & 60 & 15.08880 & 69.7765 & 0.883946 & Validated & Validated \\
$\frac{45}{64}$ & 60 & 11.76120 & 68.6825 & 0.923228 & Validated & Validated \\
$\frac{44}{64}$ & 80 & 28.79050 & 150.907 & 0.966050 & Validated & Validated \\
$\frac{43}{64}$ & 80 & 26.85230 & 149.554 & 1.013750 & Validated & Not validated \\
$\frac{42}{64}$ & 100 & 47.10740 & 272.479 & 1.064740 & Validated & Not validated \\
$\frac{41}{64}$ & 120 & 82.90000 & 460.657 & 1.122450 & Validated & Not validated \\
$\frac{40}{64}$ & 140 & 124.6850 & 759.319 & 1.188230 & Validated & Not validated \\
$\frac{39}{64}$ & 180 & 268.4580 & 1579.23 & 1.262940 & Validated & Not validated \\
\bottomrule

\end{tabular}
\caption{The results for different values of parameter $\alpha.$  The computations
were conducted on single-threaded Intel Core i5-12600KF CPU. For a given $\alpha$, the second column contains the number of explicit modes used in computation, the third and fourth ones the time in seconds that $C^0$ and $C^1$ computations have taken, and fifth column contains the estimate on the  $\norm{V(0,1,X^0)}_{\mathcal{L}(H^2_\text{odd}(-\pi,\pi))}.$ The last two columns contain the information if the existence of periodic orbit and its local attractivity were validated. }
\label{tab:results}
\end{table}

\section{Appendix}\label{sec:Appendix}
\subsection{Some useful inequalities}
The following two results are basic, so we skip the proofs. We  use them several times in this work.
\begin{proposition}\label{pr:sumEstimation}
Let $s>1.$ The following inequality holds
\begin{equation}\label{eq:sumEstimation}
\sum_{i=1+n}^\infty \frac{1}{i^s}\leq \frac{n^{1-s}}{s-1} .
\end{equation}
\end{proposition}
 
\begin{proposition}
Let $s>1$. If $a,b>0$ then the following inequality holds
\begin{equation}\label{eq:inequality}
    (a+b)^s\leq 2^{s-1}(a^s+b^s).
\end{equation}
\end{proposition}
The following proposition is used several times in our estimates.
\begin{proposition}\label{eq:integralfracExtEstimation}
    Let $\gamma\in [0,1)$ and $t_0\leq t$ then for every $\varepsilon>0$ and $\kappa> 0$ there exist $\alpha\geq 
\kappa$ such that
\begin{equation*}
    \int_{t_0}^t \frac{1}{(t-s)^\gamma}e^{\alpha(s-t_0)}ds\leq \varepsilon e^{\alpha(t-t_0)}.
\end{equation*}
\end{proposition}

\subsection{Algebra}\label{sec:algebra} The decay of the Fourier coefficients for smooth periodic functions is related to their regularity: if a periodic function  $u:\mathbb{R}\to \mathbb{R}$ is of class $C^s$, then its coefficients must decay as $\frac{1}{i^s}.$ 
Clearly, the product of two $C^s$ functions also has regularity $C^s$. This is related to our results of this section, which state that if two functions, represented in the sine or cosine Fourier series, have some given decay of the Fourier coefficients of the form $O\left(\frac{1}{i^s}\right)$ for $s>1$ then their product must have the same decay. Moreover, we provide exact estimates for the Fourier coefficients of the product: such estimates  are needed in rigorous computations of nonlinear polynomial terms present in the equations. 
In the following Lemmas we assume that functions $u,v\in L^2(0,\pi)$ have the Fourier expansion given in terms of either sine or cosine series. Such functions can be extended to, respectively, odd and even functions on the interval $\mathbb{R}$.
Estimations in these lemmas also work for extended functions, as they have the same Fourier coefficients. Therefore, they can be used for estimating the nonlinearity in the Chafee--Infante and Burgers equations.

\begin{lemma}\label{lem:formulasSinTimesSin}
    Assume that $u,v\in L^2(0,\pi)$ and  their Fourier expansions in the sine series in $L^2(0,\pi)$ are the following
    \begin{equation*}
        u(x) = \sum_{k=1}^\infty u_k\sin(kx),
        \quad
        v(x) = \sum_{k=1}^\infty v_k\sin(kx).
    \end{equation*}
    Then if we consider the cosine expansion of the product $uv$ 
    \begin{equation*}
        (uv)(x) = (uv)_0 + \sum_{k=1}^\infty(uv)_k \cos(kx), 
    \end{equation*}
    its coefficients are given by the formulas
\begin{equation}\label{eq:sinsin_formula}
    (uv)_0 = \frac{1}{2}\sum_{i=1}^\infty u_iv_i\quad
    (uv)_k =
    \frac{1}{2}\sum_{i=1}^\infty u_{i+k}v_i +
    \frac{1}{2}\sum_{i=1}^\infty u_{i}v_{i+k}-
    \frac{1}{2}\sum_{i=1}^{k-1} u_{i}v_{k-i}.
\end{equation}
\end{lemma}

\begin{lemma}\label{lem:formulasCosTimesSin}
    Assume that $u,v\in L^2(0,\pi)$ and their  corresponding Fourier expansions  in the cosine and sine series in $L^2(0,\pi)$  are the following
    \begin{equation*}
    u(x) = u_0 + \sum_{k=1}^\infty u_k\cos(kx), \quad v(x) = \sum_{k=1}^\infty v_k \sin(kx).
    \end{equation*}
    Then if we consider the following sine   Fourier expansion of the product $uv$ 
    \begin{equation*}
        (uv)(x) = \sum_{i=k}^\infty(uv)_k \sin(kx), 
    \end{equation*}
    its coefficients are given by the formula
\begin{equation}
(uv)_k = u_0v_k +
\frac{1}{2}\sum_{i=1}^\infty u_iv_{i+k}
-\frac{1}{2}\sum_{i=1}^\infty u_{i+k}v_i
+\frac{1}{2}\sum_{i=1}^{k-1}  u_{i}v_{k-i}.
\end{equation}
\end{lemma}
The following lemmas give us the estimates on the result of multiplication of $u$ and $v$ which are both represented in the sine Fourier series. The first $n$ coefficients of $u$ and $v$ are given explicitly, and the coefficients indexed by numbers larger than $n$ are expressed by polynomial decay. For the proof of the Lemma see 
\begin{lemma}\label{lem:sinTimesSin}
Assume that $u,v\in L^2(0,\pi)$ are two functions having the following expansions in the sine Fourier series
\begin{equation}
u(x) = \sum_{k=1}^\infty u_k\sin(kx), \quad v(x) = \sum_{k=1}^\infty v_k \sin(kx).
\end{equation}
Moreover we assume that for some
$n\in\mathbb{N}$ and $s>1$ the following estimates hold
\begin{equation}
    u_k \in \frac{C_u[-1,1]}{k^s},\quad v_k \in \frac{C_v[-1,1]}{k^s}\quad\text{for $k>n.$}
\end{equation}
where $C_v>0,C_u>0$. Then the following estimates hold for the coefficients of the product $uv$ in the cosine Fourier series. For $k=0$ we have
    \begin{equation}
        (uv)_0 \in \frac{1}{2}\sum_{i=1}^n u_i v_i + \frac{C_uC_v}{2}\frac{ n^{1-2s}}{ 2s - 1 }[-1,1],
    \end{equation}
for $ 1\leq k \leq 2n$ we have
    \begin{equation}
        (uv)_k \in
        \frac{1}{2}\sum_{i=1}^n u_i v_{i+k}  +
        \frac{1}{2}\sum_{i=1}^n u_{i+k} v_{i} - \frac{1}{2}\sum_{i=1}^{k-1}u_{i}v_{k-i} + C_uC_v\frac{ n^{1-2s}}{ 2s - 1 } [-1,1],
    \end{equation}
    and for $k>2n$ we have
    \begin{equation}
        (uv)_k \in \frac{D[-1,1]}{k^s},\quad D =
    \frac{1}{2}\left(
    \sum_{i=1}^n (C_u|v_i| + C_v|u_i|)
    \left(1+\left(\frac{2n+1}{2n+1-i}\right)^s\right)+
    C_uC_v(2+2^s)\frac{n^{1-s}}{s-1}
    \right).
    \end{equation}
\end{lemma}
\begin{lemma}\label{lem:cosTimesSin}
Assume that $u,v\in L^2(0,\pi)$, and the Fourier expansions of these two functions in the cosine and sine series, are given, respectively, by
\begin{equation}
u(x) = u_0+\sum_{k=1}^\infty u_k\cos(kx), 
\quad 
v(x) = \sum_{k=1}^\infty v_k \sin(kx).
\end{equation}
Moreover, we assume that for some
$n\in\mathbb{N}$ and $s>1$ the following estimates hold
\begin{equation}
    u_k \in \frac{C_u[-1,1]}{k^s},\quad v_k \in \frac{C_v[-1,1]}{k^s}\quad\text{for $k>n.$}
\end{equation}
where $C_v>0,C_u>0$.
Then the following estimates on the coefficients of the product hold for $1\leq k\leq 2n$
\begin{equation}
  (uv)_{k}\in u_0v_k +
  \frac{1}{2} \left(
    \sum_{i = 1}^n v_{i+k}u_{i} -
    \sum_{i=  1}^n v_{i}u_{k + i} +
    \sum_{i = 1 }^{k-1}v_{i}u_{k - i}
    \right)
    + C_uC_v\frac{ n^{1-2s}}{ 2s - 1 } [-1,1],
\end{equation}
and  for $k>2n$ we have
\begin{equation}
        (uv)_k \in \frac{D[-1,1]}{k^s},
\end{equation}
with 
\begin{equation*}
    D = |u_0|C_v+
    \frac{1}{2}\left(
    \sum_{i=1}^n \left(C_u|v_i| + C_v|u_i|\right)
    \left(1+
    \left(\frac{2n+1}{2n+1-i}\right)^s
    \right)+
    C_uC_v(2+2^s)\frac{n^{1-s}}{s-1}
    \right).
\end{equation*}
\end{lemma}
 The next two lemmas are crucial in rigorous integration of the variational equations, as they allow us to estimate the term $Df(u)h$ when $h$ belongs to an unbounded set.
Namely, they allow to estimate the Fourier coefficients of the product $(uv)$ uniformly with respect $v$ even if $v$ is in some unbounded set in the space $C_0.$ Hence, if $Df(u)$ can be written as the cosine series and $h$ as the sine series which  satisfy the assumptions of Lemmas \ref{lem:costimesSinDiverging} or \ref{lem:costimesSinSlowConvering}, then we get the estimate on $Df(u)h$ even if  $h$ belongs to an unbounded set. 
 
\begin{lemma}\label{lem:costimesSinDiverging}
Assume that $u,v\in L^2(0,\pi)$ are two functions whose expansions in the cosine and sine Fourier series are the following
\begin{equation}
u(x) = u_0 + \sum_{k=1}^\infty u_k\cos(kx), \quad v(x) = \sum_{k=1}^\infty v_k \sin(kx).
\end{equation}
Moreover, assume that for some
$n\in\mathbb{N}$ and some $s_1>1, s_2\leq 0$ such that $s_1+s_2>1$ the following bounds hold for the coefficients of the expansions
\begin{equation}
    u_k \in \frac{C_u[-1,1]}{k^{s_1}},\quad
    v_k \in \frac{C_v[-1,1]}{k^{s_2}}\quad\text{for $k>n$}.
\end{equation} 
Then, the following estimates on the Fourier coefficients of the product $(uv)$ in the sine series hold for $1\leq k\leq 2n$
\begin{align*}
    (uv)_{k}\in u_0v_k &+
  \frac{1}{2} \left(
    \sum_{i = 1}^n v_{i+k}u_{i} -
    \sum_{i=  1}^n v_{i}u_{k + i} +
    \sum_{i = 1 }^{k-1}v_{i}u_{k - i}
    \right)
    \\&+ \frac{C_uC_v}{2}\left(1+\left(\frac{n+1}{n+1+k}\right)^{s_2} \right)
\frac{n^{1-s_1-s_2}}{s_1+s_2-1}[-1,1],
\end{align*}
while  for $k>2n$ we have
\begin{align*}
    (uv)_k \in \frac{D[-1,1]}{k^{s_1}},
\end{align*}
with the constant $D$ given by the formula
\begin{align*}
    D &= |u_0|C_v+
    \frac{1}{2}\sum_{i=1}^n
    \left(\frac{2n+1}{2n+1+i}\right)^{s_2}\left(C_v|u_i| +
    C_u|v_i|\left(\frac{1}{2n+1+i}\right)^{s_1-s_2}\right)
    \\
    &+
    \frac{1}{2}\sum_{i=1}^n\left(C_v|u_i|+C_u|v_i|\left(\frac{1}{2n+1-i}\right)^{s_1-s_2}\right)
    \\
    &+
    C_uC_v\left(\frac{(2n+1)(n+1)}{3n+2}\right)^{s_2}\left(
    \frac{n^{1-s_1-s_2}}{s_1+s_2-1} 
    \right) + \frac{1}{2}C_vC_u\frac{n^{1-s_1}}{s_1-1}
    .
\end{align*}
\end{lemma}
\begin{proof}
    For $1\leq k \leq 2n$ we write
\begin{equation*}
(uv)_k = u_0v_k +
\frac{1}{2}\sum_{i=1}^n u_iv_{i+k}
-\frac{1}{2}\sum_{i=1}^n u_{i+k}v_i
+\frac{1}{2}\sum_{i=1}^{k-1}  u_{i}v_{k-i}
+\frac{1}{2}\sum_{i=n+1}^\infty u_iv_{i+k} 
-\frac{1}{2}\sum_{i=n+1}^\infty u_{i+k}v_i
.
\end{equation*}
We estimate
\begin{equation*}
    |u_{i+k}v_i|\leq \frac{C_vC_u}{(i+k)^{s_1}i^{s_2}}\leq \frac{C_vC_u}{i^{s_1+s_2}},
\end{equation*}
\begin{equation*}
        |u_iv_{i+k}|\leq \frac{C_vC_u}{i^{s_1}(i+k)^{s_2}}=
        \frac{C_vC_u}{i^{s_1+s_2}}\left(1-\frac{k}{k+i}\right)^{s_2}
        \leq \frac{C_vC_u}{i^{s_1+s_2}}\left(1-\frac{k}{n+1+k}\right)^{s_2}.
\end{equation*}
The last inequality follows from fact that the sequence $\{(1-\frac{k}{k+i})^{s_2}\}_{k\geq 2n+1}$ is non-increasing with respect to $i.$ So with the use of Proposition \ref{eq:sumEstimation} we get the estimation for $1 \leq k\leq 2n.$
For $k>2n$ we write
\begin{align*}
    (uv)_k &= u_0v_k +
    \frac{1}{2}\sum_{i=1}^n u_iv_{i+k}
-\frac{1}{2}\sum_{i=1}^n u_{i+k}v_i
+\frac{1}{2}\sum_{i=n+1}^\infty u_iv_{i+k} -\frac{1}{2}\sum_{i=n+1}^\infty u_{i+k}v_i
\\&+\frac{1}{2}\sum_{i=1}^{n}  u_{i}v_{k-i} +
\frac{1}{2}\sum_{i=k-n}^{k-1}  u_{i}v_{k-i}
+ \frac{1}{2}\sum_{i=n+1}^{k-n-1}u_{i}v_{k-i},
\end{align*}
and estimate each sum separately.
We have
\begin{align*}
    \sum_{i=1}^{n} |u_{i}v_{i+k}|&
    \leq
     C_v\sum_{i=1}^n\frac{1}{(i+k)^{s_2}}|u_i|=
     \frac{C_v}{k^{s_2}}
    \sum_{i=1}^n\brakets{1-\frac{i}{k+i}}^{s_2}|u_i|
    \leq
    \frac{C_v}{k^{s_2}}
    \sum_{i=1}^n
    \brakets{1-\frac{i}{2n+1+i}}^{s_2}|u_i|,
    \\
    \sum_{i=1}^{n} |u_{i+k}v_i|
    &\leq
     C_u\sum_{i=1}^n\frac{1}{(i+k)^{s_1}}|v_i|
     \leq
     \frac{C_u}{k^{s_2}}\sum_{i=1}^n
     \brakets{1-\frac{i}{i+k}}^{s_2}\frac{1}{(i+k)^{s_1-s_2}}|v_i|
     \\&\leq
     \frac{C_u}{k^{s_2}}\sum_{i=1}^n\frac{1}{(i+2n+1)^{s_1-s_2}}
     \brakets{1-\frac{i}{2n+1+i}}^{s_2}|v_i|.
\end{align*}
In the above derivation we have used the fact that the sequence $\left\{\left(1-\frac{i}{k+i}\right)^{s_2}\right\}_{k\geq 2n+1}$ is non-increasing with respect to $k.$
Now we observe that for $s_2\leq 0,k\geq 2n+1,i\geq n+1 $ we have
\begin{equation*}
    \brakets{\frac{ki}{i+k}}^{s_2}\leq \brakets{\frac{(2n+1)(n+1)}{3n+2}}^{s_2}.
\end{equation*}
We estimate next two sums
\begin{align*}
   \sum_{i=n+1}^\infty |u_{i+k}v_{i}|&
    \leq 
    \frac{C_vC_u}{k^{s_2}}\sum_{i=n+1}^\infty\frac{k^{s_2}}{(i+k)^{s_1}i^{s_2}} 
    =
    \frac{C_vC_u}{k^{s_2}}
    \sum_{i=n+1}^\infty
    \brakets{\frac{ki}{i+k}}^{s_2} \frac{1}{(i+k)^{s_1-s_2}i^{2s_2}}
    \\&\leq
     \frac{C_vC_u}{k^{s_2}}
     \brakets{\frac{(n+1)(2n+1)}{3n+2}}^{s_2}
    \sum_{i=n+1}^\infty \frac{1}{i^{s_1+s_2}}
    \\&\leq
    \frac{C_vC_u}{k^{s_2}}
    \brakets{\frac{(n+1)(2n+1)}{3n+2}}^{s_2}
    \frac{n^{1-s_1-s_2}}{-1+s_1+s_2},
    \\
    \sum_{i=n+1}^\infty |u_iv_{i+k}|&\leq 
    \frac{C_vC_u}{ k^{s_2} } 
    \sum_{i=n+1}^\infty\frac{ k^{s_2}i^{s_2}}{i^{s_1+s_2} (i+k)^{s_2}}
    =
\frac{C_vC_u}{k^{s_2}} 
\sum_{i=n+1}^\infty\frac{1}{i^{s_1+s_2} }
\brakets{\frac{ki}{i+k}}^{s_2}\\
&\leq
\frac{C_vC_u}{k^{s_2}}
\brakets{\frac{(2n+1)(n+1)}{3n+2}}^{s_2} 
\frac{n^{1-s_1-s_2}}{-1+s_1+s_2}.
\end{align*}
For the next sums we compute
\begin{align*}
    \sum_{i=1}^{n}  |u_{i}v_{k-i}|
    &\leq
    \frac{C_v}{k^{s_2}}
    \sum_{i=1}^{n}|u_i|
    \brakets
    {\frac{k}{k-i}}^{s_2} 
    \leq
    \frac{C_v}{k^{s_2}}
    \sum_{i=1}^{n}|u_i|,
\\
    \sum_{i=k-n}^{k-1}  |u_{i}v_{k-i}| &= 
    \sum_{i=1}^{n}  |u_{k-i}v_{i}|\leq
    \frac{C_u}{k^{s_2}}
    \sum_{i=1}^{n} |v_i|\frac{k^{s_2}}{(k-i)^{s_1}}
    =
    \frac{C_u}{k^{s_2}}
    \sum_{i=1}^{n} |v_i|\frac{1}{(k-i)^{s_1-s_2}}\left(1+\frac{i}{k-i}\right)^{s_2}
    \\&\leq
    \frac{C_u}{k^{s_2}}
    \sum_{i=1}^{n} |v_i|\frac{1}{\brakets{2n+1-i}^{s_1-s_2}}.
\end{align*}
And the last one is estimated in the following way
\begin{align*}
\sum_{i=n+1}^{k-n-1}|u_{i}v_{k-i}|
\leq
\frac{C_uC_v}{k^{s_2}}
\sum_{i=n+1}^{k-n-1}\frac{1}{i^{s_1}}\brakets{\frac{k}{k-i}}^{s_2}
\leq
\frac{C_vC_u}{k^{s_2}}
\sum_{i=n+1}^{k-n-1}\frac{1}{i^{s_1}} \leq \frac{C_vC_u}{k^{s_2}}
\frac{n^{1-s_1}}{-1+s_1} 
,
\end{align*}
which ends the estimations and the proof.
\end{proof}

\begin{lemma}\label{lem:costimesSinSlowConvering}
Assume that $u,v\in L^2(0,\pi)$ are two functions whose expansions in the cosine and sine Fourier series are the following
\begin{equation}
u(x) = u_0 + \sum_{k=1}^\infty u_k\cos(kx), \quad v(x) = \sum_{k=1}^\infty v_k \sin(kx).
\end{equation}
Moreover, assume that for some
$n\in\mathbb{N}$ and some $s_1>1, s_2\in(0,1]$ such that $s_1-s_2>1$ the following bounds hold for the coefficients of the expansions
\begin{equation}
    u_k \in \frac{C_u[-1,1]}{k^{s_1}},\quad
    v_k \in \frac{C_v[-1,1]}{k^{s_2}}\quad\text{for $k>n$}.
\end{equation} 
Then, the following estimates on the Fourier coefficients of the product $(uv)$ in the sine series hold for $1\leq k\leq 2n$
\begin{align*}
    (uv)_{k}\in u_0v_k &+
  \frac{1}{2} \left(
    \sum_{i = 1}^n v_{i+k}u_{i} -
    \sum_{i=  1}^n v_{i}u_{k + i} +
    \sum_{i = 1 }^{k-1}v_{i}u_{k - i}
    \right)
    \\&+ C_uC_v\frac{n^{1-s_1-s_2}}{s_1+s_2-1}[-1,1],
\end{align*}
while  for $k>2n$ we have
\begin{align*}
    (uv)_k \in \frac{D[-1,1]}{k^{s_2}},
\end{align*}
where the constant $D$ is given by
\begin{align*}
    D &= |u_0|C_v+
    \frac{1}{2}\sum_{i=1}^n\left(C_v|u_i|+C_u|v_i|\frac{1}{(i+2n+1)^{s_1-s_2}}\right)
    \\
    &+ 
    \frac{1}{2}
    C_uC_v\frac{n^{1-s_1}}{-1+s_1} 
    + 
    \frac{1}{2}C_uC_v\frac{n^{1-s_1}}{-1+s_1}
    \\
    &+
    \frac{1}{2}\sum_{i=1}^n
    \left(C_v|u_i|+ C_u|v_i|\left(\frac{1}{2n+1-i}\right)^{s_1-s_2}\right)\left(1+\frac{i}{2n+1-i}\right)^{s_2}
    \\
    &+
    C_uC_v \left(\frac{2}{n+1}\right)^{s_2}\frac{n^{1+s_2-s_1}}{s_1-s_2-1}
    .
\end{align*}
\end{lemma}
\begin{proof}
       For $1\leq k \leq 2n$ we write
\begin{equation*}
(uv)_k = u_0v_k +
\frac{1}{2}\sum_{i=1}^n u_iv_{i+k}
-\frac{1}{2}\sum_{i=1}^n u_{i+k}v_i
+\frac{1}{2}\sum_{i=1}^{k-1}  u_{i}v_{k-i}
+\frac{1}{2}\sum_{i=n+1}^\infty u_iv_{i+k} 
-\frac{1}{2}\sum_{i=1}^\infty u_{i+k}v_i
.
\end{equation*}
    We start with observation  that for every $k\in\mathbb{N}$ we have
    \begin{equation*}
        |u_i v_{k+i}|\leq\frac{C_uC_v}{i^{s_1+s_2}}, 
        \quad 
        |u_{i+k} v_{i}|\leq\frac{C_uC_v}{i^{s_1+s_2}}.
    \end{equation*}
    These inequalities allow us to follow the same argument as in Lemma \ref{lem:sinTimesSin} to obtain estimates for coefficients $(uv)_k$ for $1\leq k\leq 2n$.
    For $k>2$ we write
\begin{align*}
    (uv)_k &= u_0v_k +
    \frac{1}{2}\sum_{i=1}^n u_iv_{i+k}
-\frac{1}{2}\sum_{i=1}^n u_{i+k}v_i
+\frac{1}{2}\sum_{i=n+1}^\infty u_iv_{i+k} -\frac{1}{2}\sum_{i=n+1}^\infty u_{i+k}v_i
\\&+\frac{1}{2}\sum_{i=1}^{n}  u_{i}v_{k-i} +
\frac{1}{2}\sum_{i=k-n}^{k-1}  u_{i}v_{k-i}
+ \frac{1}{2}\sum_{i=n+1}^{k-n-1}u_{i}v_{k-i},
\end{align*}
All terms except the last one are estimated in the same way as the sum in Lemma \ref{lem:formulasSinTimesSin}. Hence, we present only the results of these calculations which are given by the following formulas
\begin{align*}
     \sum_{i=1}^{n} |u_{i}v_{i+k}|&\leq
    \frac{C_v}{k^{s_2}}
    \sum_{i=1}^n|u_i|,
    \quad
    \sum_{i=1}^{n} |u_{i+k}v_{i}|\leq
    \frac{C_u}{k^{s_2}}
    \sum_{i=1}^n|v_i|\frac{1}{(i+2n+1)^{s_1-s_2}},
    \\
    \sum_{i=n+1}^{\infty} |u_{i}v_{i+k}|&\leq
    \frac{C_uC_v}{k^{s_2}}\frac{n^{1-s_1}}{-1+s_1},
    \quad
    \sum_{i=n+1}^{\infty} |u_{i+k}v_{i}|\leq
    \frac{C_u C_v}{k^{s_2}}
    \frac{n^{1-s_1}}{-1+s_1},
    \\
    \sum_{i=1}^{n}  |u_{i}v_{k-i}|&\leq
    \frac{C_v}{k^{s_2}}\sum_{i=1}^n
    |u_i|\brakets{1+\frac{i}{2n+1-i}}^{s_2},
    \quad 
    \\
    \sum_{i=k-n}^{k-1}|u_{i}v_{k-i}| &=\sum_{i=1}^{n}  |v_{i}u_{k-i}|  \leq
    \frac{C_u}{k^{s_2}}\sum_{i=1}^n
    |v_i|\brakets{\frac{1}{2n+1-i}}^{s_1-s_2}
    \brakets{1+\frac{i}{2n+1-i}}^{s_2}
    \\
    &=  
     \frac{C_u}{k^{s_2}}\sum_{i=1}^n
    |v_i|
\frac{(2n+1)^{s_2}}{(2n+1-i)^{s_1}}.
\end{align*}   
For the last sum we have
\begin{align*}
\sum_{i=n+1}^{k-n-1}|u_{i}v_{k-i}|
&\leq
C_uC_v
\sum_{i=n+1}^{k-n-1}\frac{1}{i^{s_1}}\frac{1}{(k-i)^{s_2}}
=
\frac{C_uC_v}{k^{s_2}}
\sum_{i=n+1}^{k-n-1}
\brakets{\frac{1}{i}+\frac{1}{k-i} }^{s_2}
\frac{1}{i^{s_1-s_2}}
\\&\leq
\frac{C_uC_v}{k^{s_2}}
\brakets{\frac{2}{n+1}}^{s_2}
\sum_{i=n+1}^{k-n-1}
\frac{1}{i^{s_1-s_2}}
\leq
\frac{C_uC_v}{k^{s_2}}
\brakets{\frac{2}{n+1}}^{s_2}
\frac{n^{1+s_2-s_1}}{s_1-s_2-1}.
\end{align*}
This shows the estimates from the assertion of the Lemma.
\end{proof}

The following simple lemmas are useful to implement operations on the infinite interval vectors. They can be used when working both with sine and cosine Fourier series with polynomial bounds.
\begin{lemma}
Assume that the sequence $\{u_i \}_{i=1}^\infty$ satisfies
\begin{equation*}
     u_i\in [u_i^-,u_i^+]\ \
     \text{ for }\ \ i\leq n \ \ \text{ and }\ \
     u_i \in \frac{[C_u^-,C_u^+]}{i^s}\ \
     \text{ for  }\ \ i>n .
\end{equation*}
If $k<n$, then
\begin{equation*}
    u_i \in [D_u^-,D_u^+]\quad\text{for } i>k,
\end{equation*}
where 
\begin{equation*}
    D_u^- =
    \min \{ u_{k+1}^-(k+1)^s,\ldots,u_{n}^-n^s ,C_u^- \}, \quad
    D_u^+ = \max \{ u_{k+1}^+(k+1)^s,\ldots,u_{n}^+n^s,C_u^+ \}.
\end{equation*}
\end{lemma}

\begin{lemma}
Assume that sequence $\{u_i \}_{i=1}^\infty$ satisfies
\begin{equation*}
     u_i\in [u_i^-,u_i^+]\ \
     \text{ for }\ \ i\leq n\ \  \text{ and }  \ \
     u_i \in \frac{[C_u^-,C_u^+]}{i^s}\ \
     \text{ for  }\ \ i>k.
\end{equation*}
Then for $s_1<s$ there holds
\begin{equation*}
    u_i \in [D_u^-,D_u^+]\quad\text{for }i>n,
\end{equation*}
where,
\begin{equation*}
    D_u^- =
    \min \left\{ 0,\frac{C_u^-}{(n+1)^{s-s_1}} \right\}, \quad
    D_u^+ =
    \max \left\{0,\frac{C_u^+}{(n+1)^{s-s_1}} \right\}.
\end{equation*}
\end{lemma}
\begin{lemma}\label{lem:add}
Assume that sequences $\{u_i \}_{i=1}^\infty$ and $\{v_i \}_{i=1}^\infty$ satisfy

\begin{equation}
    u_i \in \frac{[C_u^-,C_u^+]}{i^{s_1}},\quad
    v_i \in \frac{[C_v^-,C_v^+]}{i^{s_2}}\quad\text{for $i>n$},
\end{equation}
with the constants $s_1,s_2.$ Then
\begin{itemize}
    \item if $s_1=s_2$ then for $i>n$ we have $u_i+v_i\in
    \frac{[C_v^- + C_v^-,C_u^+ + C_v^+]}{i^s},$
    \item if $s_1 <s_2$ then for $i>n$ we have
    $u_i+v_i\in
    \frac{\left[C_u^-,C_u^+\right] + \frac{[0,1]}{(n+1)^{s_2-s_1}}\left[C_v^-,C_v^+\right] }{i^s_1},$
    \item if $s_1 >s_2$ then for $i>n$ we have
    $u_i+v_i\in
    \frac{\frac{[0,1]}{(n+1)^{s_1-s_2}}[C_u^-,C_u^+] + [C_v^-,C_v^+] }{i^{s_2}}.$
\end{itemize}
\end{lemma}
\begin{lemma}\label{lem:mult}
Assume that sequences $\{u_i \}_{i=1}^\infty$ and $\{v_i \}_{i=1}^\infty$ satisfy

\begin{equation}
    u_i \in \frac{[C_u^-,C_u^+]}{i^{s_1}}\quad \textrm{and}\quad
    v_i \in \frac{[C_v^-,C_v^+]}{i^{s_2}}\quad\text{for $i>n$},
\end{equation}
with some constants $s_1,s_2.$
Then for $i>n$ we have
$u_iv_i\in \frac{[C_u^-+C_v^-,C_u^+ + C_v^+]}{i^{s_1+s_2}}.$
\end{lemma}

\subsection{\texorpdfstring{$\mathcal{L}(C_0(0,\pi))$}{Lg} norm estimates}
The following lemmas allow us to estimate the norm of a linear operators defined on the space $C_0(0,\pi).$
\begin{lemma}\label{lem:C0operatorNormEstimation}
Assume that $A$ is a linear and bounded operator $A:C_0(0,\pi)\to C_0(0,\pi)$ such that for some $n\geq 1$
\begin{equation*}
    A\sin(ix)\in W_i,\quad AV_{rest}\subset R,\quad   \text{ for i}\in\{1,\dots,n\}, 
\end{equation*}
where 
\begin{equation*}
    V_{rest}:= \left\{v\in C_0(0,\pi):\; v_i = 0 \text{ for }i\in\{1,\ldots,n\}, \; v_i\in[-1,1] \text{ for }i\in \mathbb{N}\setminus\{1,\ldots,n\} \right\}.
\end{equation*}
Then we have
\begin{equation*}
    \norm{A}_{\mathcal{L}{(C_0(0,\pi))}}\leq 2\left(\sum_{i=1}^{n}\sup_{x\in W_i}\norm{x}_{C_0(0,\pi)}+ \sup_{x\in R}\norm{x}_{C_0(0,\pi)} \right).
\end{equation*}
\end{lemma}
\begin{proof}
We take the $v\in C_0(0,\pi)$ with $\norm{v}_{C_0(0,\pi)}\leq 1.$ Then the Fourier coefficient are given by the equation $v_i =\frac{2}{\pi} \int_0^\pi\sin(ix)v(x)dx.$
We estimate
\begin{equation*}
    |v_i|\leq \frac{2}{\pi} \int_0^\pi |\sin(ix)v(x) |\leq 2. 
\end{equation*}
We write 
\begin{equation*}
    v = \sum_{i=1}^{n-1}v_i\sin(ix) + v_{rest}.     
\end{equation*}
We observe that $v_{rest}\in 2V_{rest}.$ From the linearity of $A$ and the triangle inequality we get the assertion.
\end{proof}
\begin{lemma}\label{lem:C0NormEstimation}
Assume that $v\in C_0(0,\pi)$. Moreover, let 
\begin{equation*}
    |v_i|\leq \frac{C}{|i|^s} \text{ for $i\geq n+1$ },
\end{equation*}
for some $n\in\mathbb{N}, C>0,s>1$.  
 Then we have 
\begin{equation*}
    \norm{v}_{C_0(0,\pi)}\leq \sum_{i=1}^{n}|v_i| + C\frac{  n^{-s+1}}{s-1}
\end{equation*}
\end{lemma}
\begin{proof}
Firstly, let us observe that the Fourier series of $v$ converges in $C_0(0,\pi).$ Indeed, for every $k_1,k_2\in\mathbb{N},$ such that $k_1+1\leq k_2$ we have
\begin{equation*}
    \norm{\sum_{i=k_1+1}^{k_2} v_i\sin(x)}_{C_0(0,\pi)}\leq  \sum_{i=k_1+1}^{\infty }\frac{C}{|i|^s} \leq C\frac{  k_2^{-s+1}}{s-1},
\end{equation*}
so the Fourier series is converging to $v$. We estimate 
\begin{equation*}
    \norm{v}_{C_0(0,\pi)} \leq 
    \sum_{i=1}^\infty |v_i|\norm{\sin(ix)}_{C_0(0,\pi)}\leq \sum_{i=1}^{n}|v_i| + C\frac{  n^{-s+1}}{s-1},
\end{equation*}
which ends the proof.
\end{proof}
\subsection{Estimation of \texorpdfstring{$\mathcal{L}(H^2_\text{odd}(-\pi,\pi))$}{L(H2odd(-pi,pi))} norm} The next two lemmas allow us to estimate the norm of linear operators defined in the space \texorpdfstring{$H^2_{odd}(-\pi,\pi)$}{H2odd(-pi,pi)}, which is used in the results for the Burgers equation.
\begin{lemma}\label{le:BurgersNormEstimate1}
   Assume that $A:H^2_\text{odd}(-\pi,\pi)\to H^2_\text{odd}(-\pi,\pi)$ is a linear and bounded operator, that is, $A\in \mathcal{L}(H^2_\text{odd}(-\pi,\pi)),$ that satisfies
   $AV\subset R,$
   where $V =\{u\in H^2_\text{odd}(-\pi,\pi): u_k\in\frac{[-1,1]}{k^2} \},$ and $R\subset H^2_{per}(-\pi,\pi)$ is the rigorously found bound for $AV$.
   Then  
   \begin{equation}
       \norm{A}_{\mathcal{L}(H^2_\text{odd}(-\pi,\pi))}\leq \sup_{u\in R} \sqrt{ \sum_{k=1}^\infty|u_k|^2k^4}.
   \end{equation}
\end{lemma}
\begin{proof}
    Assume that $u\in H^2_\text{odd}(-\pi,\pi)$ with $\norm{u_{xx}}_{L^2}=1.$ Then we have $|u_k|\leq \frac{1}{\sqrt{\pi}k^2}$. This means that $u\in\sqrt{\frac{1}{\pi}} V$ and $Au \in\sqrt{\frac{1}{\pi}}R.$ Then we have 
    \begin{equation*}
        \norm{Au}\leq\sup_{u\in R} \norm{\frac{1}{\sqrt{\pi}}u}_{ H^2_\text{odd}(-\pi,\pi)} = \frac{1}{\sqrt{\pi}}\sup_{u\in R}\sqrt{{\pi}} \sqrt{\sum_{k=1}^\infty|u_k|^2k^4} = a,
    \end{equation*}
which ends the proof.
\end{proof}
The following Lemma is used to rigorously compute the $H^2_\text{odd}(-\pi,\pi)$ norm. For proof, we only need to use Proposition \ref{pr:sumEstimation}. 
\begin{lemma}\label{le:BurgersNormEstimate2}
    Assume that 
    \begin{equation*}
        u(x) = \sum_{k=1}^\infty u_k\sin(kx).
    \end{equation*}
    which satisfies  $u_k \in \frac{C[-1,1]}{k^{s}}$  for every $k>n,$ with $s>\frac{5}{2}$. Then we have
    \begin{equation*}
        \sum_{k=1}^\infty|u_k|^2k^4\in \sum_{k=1}^n|u_k|^2k^4 +C^2\frac{n^{1-2(s-2)}}{2(s-2)-1}[0,1].
    \end{equation*}
\end{lemma}

\printbibliography 

@article{ZPKuramotoII,
author = {P. Zgliczy\'{n}ski},
year = {2004},
pages = {157--185},
title = {Rigorous numerics for dissipative Partial Differential Equations II. Periodic orbit for the {K}uramoto--{S}ivashinsky PDE - a computer-assisted proof},
volume = {4},
journal = {Foundations of Computational Mathematics},
}

@article{ZPKuramotoIII,
author = {P. Zgliczy\'{n}ski},
year = {2010},
pages = {197--262},
volume={36},
journal={Topological Methods in Nonlinear Analysis},
title = {Rigorous Numerics for Dissipative PDEs III. An effective algorithm for rigorous integration of dissipative PDEs}
}

@article{ControlKapelaZgliczynski,
author = {T. Kapela and P. Zgliczy\'{n}ski},
year = {2007},
pages = {365--385},
title = {A {L}ohner-type algorithm for control systems and ordinary differential inclusions},
volume = {11},
journal = {Discrete and Continuous Dynamical Systems - Series B},
}

@article{ArioliBrusselator,
author = {G. Arioli},
year = {2021},
pages = {106079},
title = {Computer assisted proof of branches of stationary and periodic solutions, and {H}opf bifurcations, for dissipative PDEs},
volume = {105},
journal = {Communications in Nonlinear Science and Numerical Simulation},
}

@article{ChafeeInfanteOryginalny,
author = {N. Chafee and E. Infante},
title = {A bifurcation problem for a nonlinear partial differential equation of parabolic type},
journal = {Applicable Analysis},
volume = {4},
number = {1},
pages = {17-37},
year  = {1974},
publisher = {Taylor \& Francis},
}

@article{ChafeeInfanteLangaCarvalho,
author = {A. Carvalho and J. Langa and J. Robinson},
year = {2012},
month = {07},
pages = {},
title = {Structure and bifurcation of pullback attractors in a non-autonomous Chafee-Infante equation},
volume = {140},
journal = {Proceedings of the American Mathematical Society},
}

@article{ZgliczuCyrankaC1,
author = {J. Cyranka and P. Zgliczyński},
year = {2014},
month = {03},
pages = {},
title = {Existence of Globally Attracting Solutions for One-Dimensional Viscous Burgers Equation with Nonautonomous Forcing---A Computer Assisted Proof},
volume = {14},
journal = {SIAM Journal on Applied Dynamical Systems},
}

@article{ZgliczuKrawczyk,
author = {Z. Galias and P. Zgliczyński},
year = {2007},
month = {12},
pages = {4261-4272},
title = {Infinite Dimensional Krawczyk Operator for Finding Periodic orbits of Discrete Dynamical Systems.},
volume = {17},
journal = {I. J. Bifurcation and Chaos},
}

@book{ChepyzhovVishik,
author={V. Chepyzhov and M. Vishik},
title={Attractors for Equations of Mathematical Physics},
year = {2002},
publisher={American Mathematical Society},
place={Providence},
series={Colloquium
Publications},
volume={49},
}

@book{LangaCarvalhoBook,
author = {A. Carvalho and J. Langa and J. Robinson},
title = {Attractors for infinite-dimensional non-autonomous dynamical systems},
year = {2013},
month = {01},
pages = {},
publisher={Applied Mathematical Sciences},
volume = {182},
isbn = {978-1-4614-4580-7},
}

@article{MatanoLapNumber,
author = {H. Matano},
year = {1982},
month = {01},
pages = {401 - 441},
title = {Nonincrease of the lap-number of a solution for a one-dimensional semilinear parabolic equation},
volume = {29},
journal = {Journal of the Faculty of Science, University of Tokyo}
}

@article{FiedlerLapNumber,
author = {P. Brunovský and B. Fiedler},
title = {Numbers of zeros on invariant manifolds in reaction- diffusion equations},
journal = {Nonlinear Analysis: Theory, Methods \& Applications},
volume = {10},
number = {2},
pages = {179-193},
year = {1986},
issn = {0362-546X}
}

@book{hale1988asymptotic,
  title={Asymptotic Behavior of Dissipative Systems},
  author={J. Hale},
  series={Mathematical Surveys and Monographs},
  volume={25},
  publisher={American Mathematical Society},
  year={1988}
}

@article{HENRY1985165,
title = {Some infinite-dimensional Morse-Smale systems defined by parabolic partial differential equations},
journal = {Journal of Differential Equations},
volume = {59},
number = {2},
pages = {165-205},
year = {1985},
issn = {0022-0396},
author = {D. Henry}
}

@article{CYRANKAFOURIER,
author = {J. Cyranka},
year = {2014},
month = {04},
pages = {},
title = {Efficient and Generic Algorithm for Rigorous Integration Forward in Time of dPDEs: Part I},
volume = {59},
journal = {Journal of Scientific Computing},
}

@article{BrusselatorPrzyjetaPraca,
author = {J. Banaśkiewicz and P. Kalita and P. Zgliczynski},
year = {2024},
month = {11-12},
pages = {815-862},
title = {Computer-assisted validation of the existence of periodic orbit in the Brusselator system},
volume = {29},
journal = {Advances in Differential Equations},
}

@article{ArioliKoch2010,
  author    = {G. Arioli and H. Koch},
  title     = {Integration of dissipative partial differential equations: A case study},
  journal   = {SIAM Journal on Applied Dynamical Systems},
  year      = {2010},
  volume    = {9},
  number    = {3},
  pages     = {1119--1133},
}

@article{KiselevNazarovShterenberg2008,
  title = {Blow up and regularity for fractal Burgers equation},
  author = {A. Kiselev and F. Nazarov and R. Shterenberg},
  journal = {Dynamics of Partial Differential Equations},
  year = {2008},
  volume = {5},
  number = {3},
  pages = {211--240},
}

@misc{CAPD,
 title = {CAPD library},
 url = {http://capd.ii.uj.edu.pl}
}

@misc{CodeBurgers,
 title = {Code for Computer-Assisted Proofs for Burgers equation},
 url = {https://github.com/Kruci-no/BurgersEquationC1Computations.git}
}

@misc{CodeChaffeInfante,
 title = {Code for Computer-Assisted Proofs for Chafee--Infante equation},
 url = {https://github.com/Kruci-no/Chafee-Infante-ComputerAssistedProofVersion_2}
}

@article{KalitaZgliczynski2020,
  author       = {P. Kalita and P. Zgliczyński},
  title        = {On non-autonomously forced Burgers equation with periodic and Dirichlet boundary conditions},
  journal      = {Proceedings of the Royal Society of Edinburgh. Section A, Mathematics},
  year         = {2020},
  volume       = {150},
  number       = {4},
  pages        = {2025--2054},
}

@misc{wilczak2024symbolic,
  title        = {Symbolic dynamics for the Kuramoto-Sivashinsky PDE on the line II},
  author       = {D. Wilczak and P. Zgliczyński},
  year         = {2024},
  archivePrefix= {arXiv},
  eprint       = {2405.17087},
  note         = {arXiv preprint}
}

@misc{wilczak2025selfconsistent,
  title        = {Self-consistent bounds method for dissipative PDEs},
  author       = {D. Wilczak and P. Zgliczyński},
  year         = {2025},
  archivePrefix= {arXiv},
  eprint       = {2502.09760},
  note         = {arXiv preprint}
}
\end{document}